\documentclass{elsarticle}
\usepackage[english]{babel}

\newcommand{\bv}{\mbox{\boldmath{$v$}}}

\newcommand{\bzero}{\mbox{\boldmath{$0$}}}
\usepackage{stmaryrd}
\usepackage{mathrsfs}
\usepackage{lineno,hyperref}
\usepackage{bookmark}
\usepackage{bm}
\usepackage{amsmath,amsfonts,amsmath,amssymb,amsthm,cases}
\usepackage[ruled]{algorithm2e}
\usepackage{multirow}
\usepackage{longtable}
\usepackage{fullpage}

\usepackage{graphics,graphicx,epsfig,epstopdf}
\usepackage{color,natbib}
\usepackage{epstopdf}
\usepackage{subfigure}
\usepackage[margin=10pt,font=small,labelfont=bf,labelsep=period]{caption}
\biboptions{sort&compress}

\newtheorem{thm}{Theorem}
\newtheorem{lem}{Lemma}[section]
\newtheorem{prop}{Proposition}[section]

\newtheorem{exam}{Example}
\newtheorem{remark}{Remark}[section]
\newtheorem{definition}{Definition}[section]
\newproof{pf}{Proof}

\numberwithin{equation}{section}

\graphicspath{{images/}}

\begin{document}
	\begin{frontmatter}	
		\title{Shape design with phase field methods for structural hemivariational inequalities in contact problems} 
		\tnotetext[tnote1]{The work was supported in part by the National Key Basic Research Program (2022YFA1004402), the National Natural Science Foundation of China (12471377, 12401528), the Science and Technology Commission of Shanghai Municipality (22DZ2229014) and the Fundamental Research Funds for the Central Universities (30924010837).}

		\author[ecnu]{Yixin Tan}
		\address[ecnu]{School of Mathematical Sciences, East China Normal University, Shanghai 200241, China}
		\ead{52265500020@stu.ecnu.edu.cn}
		
		\author[njust]{Fang Feng\corref{cor1}}
		\ead{ffeng@njust.edu.cn}
		\address[njust]{School of Mathematics and Statistics, Nanjing University of Science and Technology, Nanjing 210094, China}

		\author[ecnu,kll]{Shengfeng Zhu\corref{cor1}}
		\ead{sfzhu@math.ecnu.edu.cn}
		\address[kll]{Key Laboratory of Ministry of Education \& Shanghai Key Laboratory of Pure Mathematics and Mathematical Practice}
		\cortext[cor1]{Corresponding authors.}
		
		\begin{abstract}
			We develop mathematical models for shape design and topology optimization in structural contact problems involving friction between elastic and rigid bodies. The governing mechanical constraint is a nonlinear, non-smooth, and non-convex hemivariational inequality, which provides a more general and realistic description of frictional contact forces than standard variational inequalities, but is also more challenging due to its non-convexity. For energy-type shape functionals, the Eulerian derivative of the hemivariational inequality is derived through rigorous shape sensitivity analysis. The rationality of a regularization approach is justified by asymptotic analysis, and this method is further applied to handle the non-smoothness of general shape functionals in the sensitivity framework. Based on these theoretical results, a numerical boundary variational method is proposed for shape optimization. For topology optimization, three phase-field algorithms are developed: a gradient-flow phase-field method, a phase-field method with second-order regularization of the cost functional, and a phase-field method coupled with topological derivatives. To the best of our knowledge, these approaches are new for shape design in hemivariational inequalities. Various numerical experiments confirm the accuracy and effectiveness of the proposed shape and topology optimization algorithms.
			
		\end{abstract}
		
		
	\end{frontmatter}
	
	\section{Introduction}
	Shape design of elastic contact structures with geometric and mechanical constraints has broad applications in the control and optimization of structural mechanics.
	Processes of contact between deformable bodies frequently occur in various structural and mechanical systems, leading to highly nonlinear behavior due to inequality constraints and complex boundary interactions.
	In recent years, advanced topology and shape optimization algorithms, such as convolution–thresholding and phase-field approaches, have been successfully developed for related physical models, including heat transfer and elasticity problems \cite{CDDX2024}.
	These methods provide efficient numerical frameworks that can be extended to the design and optimization of elastic contact interfaces.
	For modelling of frictional contact in mechanics with conventional variational inequalities, the friction force is the subdifferential of a convex functional. However, the friction force in practice is often very complex, and merely considering convex functional is not enough
	to characterize its properties. Therefore, hemivariational inequalities caused by nonsmooth and nonconvex functionals have attracted much attention during recent years. The study of shape and topology optimization for hemivariational inequalities is significant in, e.g., structural design. To the best of our view, however, research on numerical shape and topology optimization of structural hemivariational inequalities is absent in literature.
	Shape optimization of variational inequalities in contact mechanics have been studied in
	\cite{Bretin2023,BHOP14,CD21,HOP12,ABCJ2023,BCOR22,LSW,HL,HS, MBV23} (see \cite{LWY} for applications in superconductors). It is worth noting that the convexity of variational inequalities can be utilized for shape sensitivity analysis in shape optimization. In \cite{SZ1988}, a dual formulation method was proposed to overcome the non-smoothness of variational inequalities.
	The	asymptotic analysis of variational inequalities with applications to
	optimum design in elasticity was considered \cite{LDNS17}. Level set methods for shape and topology optimization in variational inequalities were proposed \cite{CD20,Chaudet23}.  The topological derivatives combined with shape optimization algorithm in variational inequalities are used in \cite{FLSS07}. 

	More generally than variational inequalities, hemivariational inequalities first introduced by Panagiotopoulos in the early 1980s \cite{Panagiotopoulos1983} are concerned with more challenging non-smooth and non-convex functionals appearing in contact mechanics. Mathematical modelling, theoretical analysis and numerical methods for hemivariational inequalities are developed recently \cite{DMP2003,HS19,HMP1999,MOS2013,OJB22,NP1995,Panagiotopoulos1993,ZeidIIB}.
	Although playing an essential role for design and control in contact mechanics, to the best of our knowledge, however, numerical modeling on shape and topology optimization for hemivariational inequalities is absent in literature except shape design for Stokes hemivariational inequalities was recently considered \cite{FYM23}.

	Phase field modeling was popularly used in topology optimization \cite{LFLK2025,BC03,BGSSSV12,JLXZ,LXZ,LY22,LYZ25,LZ25,WYZ04,ZWS06,TNK10,QHZ2022,HQZ}. Different from sharp interface descriptions of the level set method \cite{OS,Yaji,YINT10} for shape design \cite{BO}, the phase field method is a physically diffuse interface tracking technique modelling phase transitions in material sciences \cite{CH58,M87}. The structural shape is represented implicitly by a phase field function defined on a fixed design domain \cite{DBH12}. Numerical shape and topological changes can happen during evolutions of the phase field function through a dynamic gradient flow resulting from sensitivity analysis. Generally, second-order Allen-Cahn equation \cite{AC79} and fourth-order Cahn-Hilliard equation \cite{CH58} associated respectively with $L^2$ and $H^{-1}$ gradient flows of some energy are considered to evolve the interface of two materials in topology optimization. Phase field method of Cahn-Hilliard was used to topology optimization of a variational inequality for bodies in unilateral contact with given friction \cite{MK}.
	
	We consider in this paper shape and topology optimization for a structural hemivariational inequality derived from the frictional elastic contact problem. 
	For the energy cost functional, we derive the Eulerian derivative of the hemivariational inequality. However, for the general cost functional, we have to use the regularization to deal with the non-convex and non-smooth term to derive the adjoint problem. Then, using energy-type functionals as an example, we illustrate the rationale behind the regularization methods. Still, we present the shape sensitivity analysis, based on which a shape optimization algorithm is proposed.
	
	Furthermore, we present a phase field method with an Allen-Cahn equation for topology optimization of the structural hemivariational inequality. In addition, we consider a phase-field method with second-order regularization of the cost functional and a phase-field method coupled with the topological derivative for the topology optimization. Numerical results verify the effectiveness of the three algorithms.
	
	Hemivariational inequalities present significant mathematical and numerical challenges. The inherent non-smoothness and non-convexity complicate the proofs of existence and uniqueness of solutions. Moreover, traditional sensitivity analysis methods are inapplicable in this context. To overcome these issues, we establish the existence of material derivatives in Lemma 3.2 and derive the corresponding Eulerian derivative in Theorem 1. From a numerical perspective, the problem becomes highly nonlinear after regularization. To tackle this, we employ Newton's method and propose two novel phase field approaches with/without the topological derivative within the framework of topology optimization.

	The rest of the paper is organized as follows. In Section \ref{sec2}, we first introduce a hemivariational inequality in elastic structure and present the existence and uniqueness results for weak solutions to weak forms. Then, a mathematical model in shape design of elastic hemivariational inequalities is built. In Section \ref{sec3}, we develop the sensitivity analysis of the hemivariational inequality for the energy functional. In Section \ref{sec4}, we regularize the model due to non-smoothness and then perform shape sensitivity analysis for the general cost functional. Still, we take energy functional as an example to illustrate the rationality of regularization method. In Section \ref{sec5}, a phase field method is proposed for simultaneous shape and topology optimization of the regularized model problem. Both algorithms with and without topological derivatives are developed. Section \ref{sec8} introduces numerical issues and presents numerical results. Brief conclusions are drawn in the last Section \ref{secfinal}.
	
	\section{Modelling shape design problems in elastic hemivariational inequality}\label{sec2}
	Let us first introduce the generalized (Clarke) directional derivative and the generalized gradient of a local Lipschitz continuous functional (\cite{Clarke1983}). For a normed space $X$, we denote by $\parallel \cdot \parallel_X$, $X^{*}$, and $\langle \cdot, \cdot \rangle$ respectively as its norm, its topological dual, and the duality pairing between $X^{*}$ and $X$. For simplicity in exposition in the following, we always assume $X$ is a Banach space, unless stated otherwise.
	
	\begin{definition}\label{dlipschitz}
		Let $\psi:X \rightarrow \mathbb{R}$ be a local Lipschitz continuous functional on a Banach space $X$.
		The generalized (Clarke) directional derivative of $\psi$ at $x\in X $ in a direction $v\in X$ is defined by
		\[\psi^0(x;v) = \limsup_{y\rightarrow x,\,\varrho \searrow 0}\frac{\psi(y+\varrho  v)-\psi(y)}{\varrho }.\]
		The generalized gradient (subdifferential) of  $ \psi $ at x is defined by
		\[\partial \psi(x) =\{\zeta\in X^*\mid \psi^{0}(x;v) \geq \langle \zeta,v\rangle\ \forall v\in X\}.\]
		In particular, with the generalized subdifferential, one can compute the generalized directional derivative through
		\begin{equation}\label{eq0}
			\psi^0(x;v) = \max\{\langle \xi,v\rangle\,|\,\xi\in \partial \psi(x)\}.
		\end{equation}

	\end{definition}
	From \cite[Proposition 3.23 (i)]{MOS2013}, we have
	\begin{equation}\label{psi1}
		\psi^0(x;-v)=(-\psi)^0(x;v)\quad v\in X.
	\end{equation}
	A locally Lipschitz function $\psi: X\rightarrow \mathbb{R}$ is said to be \emph{regular} in the sense of Clarke at $x\in X$ if for all $v\in X$, the directional derivative $$\psi'(x;v):= \lim\limits_{\lambda\searrow 0}\frac{\psi(x+\lambda v)-\psi(x)}{\lambda}$$ exists and $\psi'(x;v)=\psi^0(x;v)$.
	The function $\psi$ is regular in the sense of Clarke on $X$ if it is regular at every point $x\in X$. From \cite[(i) and (iv) of Proposition 3.35]{MOS2013}, we derive that if $\psi_1, \psi_2$ are regular, 
	\begin{equation}\label{psi2}
		(\psi_1+\psi_2)^0(x;v)=\psi_1^0(x;v)+\psi_2^0(x;v).
	\end{equation}
	
	Let $\Omega\subset\mathbb{R}^d$ ($d=2,3$) be a reference configuration of an isotropic linear elasticity. The boundary denoted by $\partial\Omega$ is made up of four parts: $\Gamma_C$, $\Gamma_N$, $\Gamma_D$ and $\Gamma$ with $\partial\Omega=\overline{\Gamma_C}\cup \overline{\Gamma_N} \cup \overline{\Gamma_D}\cup \overline{\Gamma}$ and ${\Gamma_C}\cap {\Gamma_N} \cap {\Gamma_D} \cap {\Gamma}=\emptyset$, in which the Lebesgue measure of $\Gamma_D$ meas\,($\Gamma_D)>0$. Assume that the body subject to the action of
	a surface traction on $\Gamma_N$ and clamped on $\Gamma_D$ is traction free on $\Gamma$ and in contact with a rigid foundation on $\Gamma_C$ (see Fig. \ref{model_elastic} for illustrations of physical settings). For a vector $\bm{v}$ on the boundary $\partial\Omega$, denote by $v_{\nu}=\bm{v\cdot \nu}$ and ${v_{ \tau}}=\bm v\cdot\bm \tau$
	its normal component and tangential component, respectively. 
	We use $\mathbb{S}^d$ for the space of second-order symmetric tensors on $\mathbb{R}^d$ with inner product $\bm T:\bm U =\sum_{i,j=1}^d T_{ij}U_{ij}$ and corresponding norm $\|\bm U\|_{\mathbb{S}^d}=(\bm U:\bm U)^{\frac{1}{2}},\,\forall\, \bm T, \bm U\in \mathbb{S}^d $.
	For a tensor $\bm{\sigma}\in\mathbb{S}^d$,
	define its normal component and tangential component as $\sigma_{\nu}=\bm{\sigma \nu\cdot\nu}$ and 
	$ \sigma_{\tau}=\bm{\sigma}\bm \nu\cdot \bm \tau$, respectively. Denote by $\mathcal{S}(\bm{\sigma})=\frac{1}{2}(\bm{\sigma} +\bm{\sigma}^T)$ the symmetric part of $\bm{\sigma}$.
	
	Given a body force $\bm{f}\in L^2(\Omega)^d$ and surface load $\bm{g}_N\in L^2(\Gamma_N)^d$, we consider the following bilateral elastic contact problem with friction for elastic displacement $\bm{u}:\Omega\rightarrow\mathbb{R}^d$ at equilibrium
	\begin{equation}\label{Eq:constitute}
		\left\{
		\begin{aligned}
				\bm \sigma=\mathbb{C} \bm{\varepsilon}\quad & \quad\text{in}\ \Omega, \\
				{\rm div}\,\bm \sigma +\bm f=\bzero \quad & \quad \text{in}\ \Omega,\\
				\bm u=\bzero \quad  & \quad \text{on}\ \Gamma_D,\\
				\bm \sigma \bm \nu=\bm g_N \quad & \quad\text{on}\ \Gamma_N,\\
				\bm \sigma \bm \nu=\bm 0 \quad & \quad\text{on}\ \Gamma,\\
				u_{\nu}=0,\quad - \sigma_{\tau}\in \partial j_{\tau}(\bm{x},u_{\tau})\quad & \quad\text{on}\ \Gamma_C,
		\end{aligned}
		\right.
	\end{equation}
	where $\bm{\varepsilon}=\bm{\varepsilon}(\bm u)$ and $\bm{\sigma}=\bm{\sigma}(\bm u)$ are the strain tensor and stress tensor of the structure, respectively, $\mathbb{C}\colon\Omega\times\mathbb{S}^d\rightarrow\mathbb{S}^d $ is an elasticity operator, and $j_{\tau}$ is a potential functional.
	\begin{figure}[htbp]
		\centering
		\includegraphics[width=6.8cm,height=4.cm]{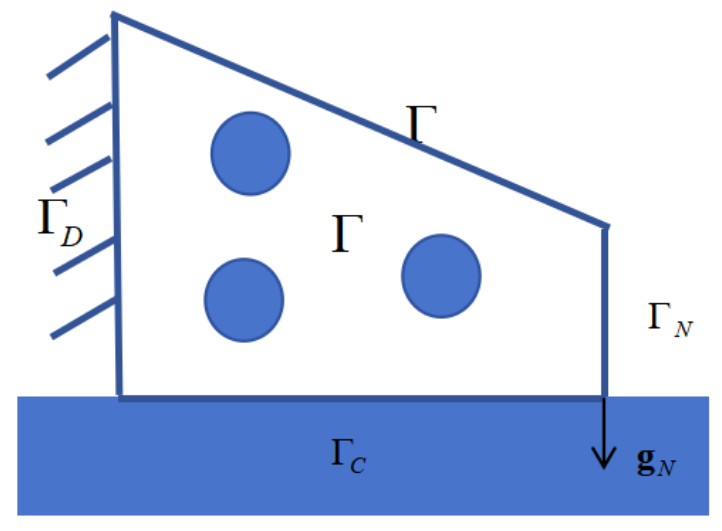}
		\caption{Illustration of problem setting.}\label{model_elastic}
	\end{figure}
	In \eqref{Eq:constitute}, we assume that $\mathbb{C}=(C_{ijkl})_{1\le i,j,k,l\le d}$ is symmetric, bounded and pointwise stable:
	\begin{equation}
		\label{Eq:pofun0}
		\left\{\begin{array}{ll}
			(a)\, C_{ijkl}(\bm x)=C_{jikl}(\bm x)=C_{klij}(\bm x),\quad \bm x\in \mathbb{R}^3,\quad 1\leq i,j,k,l\leq d,\\
			(b)\,  C_{ijkl}\in L^{\infty}_{loc}(\mathbb{R}^3),\quad 1\leq i,j,k,l\leq d,\\
			(c)\,(\mathbb{C}{\bm \varepsilon}):\bm \varepsilon \geq m_{\mathbb{C}}\|\bm \varepsilon\|_{\mathbb{S}^d}^2,\\
		\end{array}
		\right.
	\end{equation}
	where $m_{\mathbb{C}}>0$.
	%
	For bilateral elastic contact on $\Gamma_C$ of  \eqref{Eq:constitute}, the potential functional $j_{\tau}:\Gamma_C\times \mathbb{R}\rightarrow \mathbb{R}$ is assumed to satisfy
	\begin{equation}
		\label{Eq:pofun1}
		\left\{
		\begin{aligned}
			&(a)\ j_{\tau}(\cdot,  z)~{ \rm is~ measurable ~on~} \Gamma_C{\rm~for ~all}~z \in \mathbb{R}\
			{\rm and}~ j_{\tau}(\cdot,  z_0(\cdot))\in L^1(\Gamma_{C})\ {\rm for\ some\ } z_0
			\in L^2(\Gamma_{C});\\
			&(b)\ j_{\tau}(\bm x, \cdot)~{\rm is ~locally ~Lipschitz ~on~}\mathbb{R}~{\rm for~a.e.}~\bm x\in \Gamma_C;\\
			&(c)\ |\partial j_{\tau}(\bm x, z)|\le\bar{c}_0+\bar{c}_1| z|~{\rm for~a.e.}~\bm x\in \Gamma_C,
			~{\rm for~ all~}~ z\in \mathbb{R} ~~{\rm with} ~\bar{c}_0,\bar{c}_1\geq 0;\\
			&(d)\ j_{\tau}^0(\bm x,z_1; z_2- z_1)+ j_{\tau}^0(\bm x, z_2;z_1- z_2)
			\leq \alpha_{j_{\tau}}| z_1- z_2|^2~{\rm for~a.e.} ~\bm x\in\Gamma_C,\\
			&~~~~{\rm for ~all }~ z_1, z_2 \in \mathbb{R} ~{\rm with}~ \alpha_{j_{\tau}}\geq 0.
		\end{aligned}
		\right.
	\end{equation}
	For notational brevity, $j(\bm x,z)$ and $j_{\tau}^0(\bm x, v; w)$ will be simplified as $j(z)$ and $j_{\tau}^0(v;w)$, respectively. For simplicity, we assume that $j_{\tau}$ is regular and $|\partial j_{\tau}|$ is bounded.Introduce a Hilbert space $Q=L^2(\Omega;\mathbb{S}^d)$ of functions associated with the canonical
	inner product
	$$(\bm \sigma, \bm s)_Q:=\int_\Omega \sigma_{ij}s_{ij},$$
	where Einstein's summation convention is used.  The associated norm is denoted by $\|\cdot\|_Q=(\cdot,\cdot)_Q^{1/2}$.  Let
	$X= \bm{W}:=\{\bm{v}\in H^1(\Omega;\mathbb{R}^d):\bm{v}|_{\Gamma_D}=\bzero \},$
	which is equipped with the norm
		$		\|\bv\|_{\bm W}:=\|\bm \varepsilon(\bm v)\|_Q$ for any $\bm v\in\bm W.$
	Let
	$\bm V_1=\bm V_1(\Omega)=\{\bm{v}\in \bm W:  v_{\nu}|_{\Gamma_C}=0\}$.
	
	Let $(\cdot,\cdot)_\Omega$, $(\cdot,\cdot)_{\Gamma_C}$, $(\cdot,\cdot)_{\Gamma_N}$, $(\cdot,\cdot)_{\Gamma_D}$, and $(\cdot,\cdot)_{\Gamma}$ denote the inner products in $L^2(\Omega)^d$, $L^2(\Gamma_C)^d$, and $L^2(\Gamma_N)^d$, $L^2(\Gamma_D)^d$, and $L^2(\Gamma)^d$, respectively. Consider a weak form of \eqref{Eq:constitute}: Find $\bm u\in\boldsymbol V_1$ and
	$\xi_{\tau}\in L^2(\Gamma_{C})$ such that
	\begin{equation} \label{varia}
		a(\bm u,\bm v)-(\xi_{\tau}, v_{\tau})_{\Gamma_C}=(\bm f,\bm v)_\Omega+(\bm g_N,\bm v)_{\Gamma_N}
		\quad\forall\,\bm v\in \boldsymbol V_1,
	\end{equation}
	where the bilinear form $a(\bm u, \bm v)=(\mathbb{C}\bm \varepsilon(\bm u),\bm \varepsilon(\bm v))_Q$ and $-\xi_{\tau}\in \partial j_{ \tau}( u_{\tau})$ a.e.\ on $\Gamma_{C}$.
	
	Assume $\alpha_{j_{\tau}}$ in \eqref{Eq:pofun1} and $m_{\mathbb{C} }$ in \eqref{Eq:pofun0} satisfy
	\begin{equation}\label{small}
		\alpha_{j_{\tau}}<\lambda_1 m_{\mathbb{C}},
	\end{equation}
	where $\lambda_1$ is the smallest eigenvalue of the variational problem: Find $\overline{\lambda} \in \mathbb{R}^+$ and $\overline{\bm{u}}\in \boldsymbol V_1$ such that
	\begin{equation*}
		(\bm \varepsilon(\overline{\bm u}), \bm \varepsilon(\bm v))_Q
		= \overline{\lambda}( \overline{u}_{\tau} ,v_{\tau})_{\Gamma_C}\quad\forall\, \bm v\in \boldsymbol V_1.
	\end{equation*}
	The Sobolev trace inequality over $\bm V_1$ takes the form
	\begin{equation}\label{trace}
		\|v_{\tau}\|_{\Gamma_C}\leq \lambda_1^{-1/2}\|\bm v\|_{1},
	\end{equation}
	where $\|\cdot\|_{1}$ denotes a norm of $H^1(\Omega;\mathbb{R}^d)$, which is equivalent to the energy norm $\|\cdot\|_{\bm{W}}$ by Korn's inequality.
	By \eqref{eq0}, we obtain that \eqref{varia} is equivalent to the following hemivariational inequality \cite{FHH19}: Find $\bm u\in \bm V_1$ such that
	\begin{equation}\label{hemi}
		\quad a(\bm u,\bm v)+\hat{J}^0( u_{\tau};v_{\tau})\geq (\bm f,\bm v)_\Omega+(\bm g_N,\bm v)_{\Gamma_N}\quad \forall\,\bm v\in \bm V_1,
	\end{equation}
	where
	$$
	\hat{J}(z):=\int_{\Gamma_C} j_\tau(z), \quad z \in L^2\left(\Gamma_C ; \mathbb{R}^d\right) \quad{\rm and}\quad \hat{J}^0( u_{\tau};v_{\tau})=\int_{\Gamma_C}j_\tau^0(u_\tau;v_\tau).
	$$
	The presence of non-smoothness and non-convexity of functional make theoretical study of existence and uniqueness of solutions be more challenging than traditional variational inequalities. Based on \eqref{Eq:pofun0}, \eqref{Eq:pofun1} and \eqref{trace},  there exists a unique solution to \eqref{hemi}, i.e., \eqref{varia} (\cite[Theorem 3.1]{HSB17}).
	
	Referring to \cite{BBK13,HSB17}, a concrete example of $j_{\tau}$ is
	\begin{equation}\label{Eq:potential}
		j_{\tau}( z)=\int_0^{| z|}\mu(t) ,
	\end{equation}
	where  
	$$\mu(t)=(a-b)e^{-\alpha t}+b,$$
	with  $a>b$ and $\alpha$ being non-negative constants. It can be verified that \eqref{Eq:potential} satisfies the assumption \eqref{Eq:pofun1}. 
	Then, the contact condition $- \sigma_{\tau}\in \partial j_{\tau}( u_{\tau})$ in \eqref{Eq:constitute} implies
	$$- \sigma_{\tau}\in\mu(| u_{\tau}|)\partial | u_{\tau}|,$$
	which is equivalent to 
	\begin{equation}\label{physical}
		\left\{
		\begin{aligned}
			& |\sigma_{\tau}|\leq \mu(0), \ &&\text{if}\,\, u_{\tau}= 0,\\ 
			&-\sigma_{\tau}=\mu(| u_{\tau}|)\frac{u_{\tau}}{| u_{\tau}|}, \ &&\text{if}\ 
			u_{\tau}\neq 0.
		\end{aligned}
		\right.
	\end{equation}
	Here, $\mu(t)$ is interpreted as a friction bound function and \eqref{physical} is a version of Coulomb law \cite{MOS2013}.
	
	Assumed that $\Gamma_D$, $\Gamma_C$, and $\Gamma_N$ are fixed and $\Gamma$ is free to move. We consider optimal design of the mechanical structure under frictional contact conditions by minimizing an objective of energy type subject to a volume constraint:
	\begin{equation}\label{Obj}
		\min_{{\rm Vol}(\Omega)=\mathcal{C}}J(\Omega):=\frac{1}{2}a(\bm u,\bm u)-(\bm f,\bm u)_\Omega-(\bm g_N,\bm u)_{\Gamma_N}+\hat{J}(u_{\tau}),
	\end{equation}
	where $\bm{u}$ is the solution of \eqref{hemi}, ${\rm Vol}(\Omega)$ denotes the volume of $\Omega$ and the prescribed constant $\mathcal{C}>0$.

	\section{Shape sensitivity analysis}\label{sec3}
	Let $D\subset \mathbb{R}^d$ be an open bounded domain with smooth boundary. Let $\mathcal{A}=\{\Omega\,|\,\Omega\subset\subset D\}$  be an admissible set of subdomains. Define a shape functional $J:\mathcal{A}\rightarrow \mathbb{R}$ with 
	$\Omega\mapsto J(\Omega)$.
	Introduce a time parameter $t\in \mathcal{J}:=[0,\tau_0)$ with $\tau_0$ being a positive number. Let $\bm V\in C_0^{1}(D;\mathbb{R}^d)$ be a velocity field. Consider the \emph{perturbation of identity approach} with a mapping $F_t:\overline{D}\rightarrow \overline{D}$ and $F_{t}=I+t\bm V$, where $I$ denotes the identity operator. Set $F_t(\partial D)=\partial D$ and $\Omega_t=F_t(\Omega){\subset\subset D}$. Then the \emph{Eulerian derivative} of $J$ at $\Omega$ in the direction $\bm V$ is defined by 
	\begin{equation}
		dJ(\Omega;\bm V):=\lim_{t\searrow 0}\frac{J(\Omega_t)-J(\Omega)}{t},
	\end{equation}
	if this limit exists. 
	Let $F_t^{-1}$ denote the  inverse of $F_t$ and $\bm{f}\in H^1(D;\mathbb{R}^d)$.
	\begin{lem}\cite{FYZ25}The following results on shape calculus hold:
		\begin{equation}
			\begin{aligned}
				&(1)\ {\rm D}F_t|_{t=0}=I,\quad &
				(2)\ &\frac{{\rm d}}{{\rm d}t}F_t\Big|_{t=0}=\bm V,\\
				&(3)\ \frac{\rm d}{{\rm d}t}{\rm D}F_t\Big|_{t=0}={\rm D}\bm V=\bigg(\frac{\partial V_i}{\partial x_j}\bigg)_{i,j=1}^d, &(4)\ &\frac{\rm d}{{\rm d} t}({\rm D}F_t^{-1})\Big|_{t=0}=-{\rm D}\bm V,\\
				&(5)\ \frac{\rm d}{{\rm d}t}({\rm D}F_t^{-T})\Big|_{t=0}=-{\rm D}\bm V^T, &(6)\ & \frac{\rm d}{{\rm d}t}\gamma_1(t)\Big|_{t=0}={\rm div}\bm V,
			\end{aligned}
		\end{equation}
		where $\gamma_1(t)={\rm det}({\rm D}F_t)$ with ${\rm det}({\rm D}F_t)$ denoting the determinant of the Jacobian ${\rm D}F_t$.
	\end{lem}
	\begin{definition}
		The \emph{material derivative} of a function $\bm u$ is defined as the following limit in a certain space
		\begin{equation}\label{23}
			\dot{\bm u}=\lim_{t\searrow 0} \frac{\bm u_t\circ F_t- \bm u}{t}.
		\end{equation}
	\end{definition}
	
	Assume that the symmetric fourth-order tensor $\mathbb{C}=\{c_{ijkl}\}\ (i,j,k,l=1,\cdots d)$ satisfies $c_{ijkl}(\cdot)\in \mathcal{C}^1(\mathbb{R}^d)$.
	Introduce a function space:
	\begin{equation}
		\bm V_1(\Omega_t):=\{\bm{v}\in H^1(\Omega_t;\mathbb{R}^d):\bm{v}|_{\Gamma_D}=\bzero,v_{\nu}|_
		{\Gamma_{C}}=0\}.
	\end{equation}
	Consider the hemivariational inequality defined on $\Omega_t$: Find $\bm u(\Omega_t)\in \bm V_1(\Omega_t)$ such that
	\begin{equation}\label{eqt}
		\quad \int_{\Omega_t}\mathbb{C}\bm \varepsilon(\bm u_t):\bm \varepsilon(\bm v_t)+\int_{\Gamma_C}j_{\tau}^0( u_{\tau t};v_{\tau t})\geq \int_{\Omega_t}\bm f\circ F_t\cdot\bm v_t+\int_{\Gamma_N}\bm g_N\cdot\bm v_t \quad \forall\,\bm v_t\in \bm V_1(\Omega_t).
	\end{equation}
	Let $C$ be a constant that may take different values under different circumstances. 
	\begin{lem}\label{weak}
		It holds that
		\begin{equation}\label{result}
			\|\bm u^t - \bm u\|_{1}\leq Ct,
		\end{equation}
		where $ \bm u^t=\bm u(\Omega_t)\circ F_t$.
	\end{lem}
	\begin{proof}
		Let $ \bm u^t=\bm u(\Omega_t)\circ F_t\in \bm V_1(\Omega)$. Choose $\bm v_t=\bm \psi\circ F_t^{-1} $ in \eqref{eqt} for $\bm \psi\in  \bm V_1(\Omega) $, we deduce
		\begin{equation}\label{eqto}
			\int_{\Omega} \gamma_1(t)\mathbb{C}\circ F_t\bm \varepsilon^t(\bm u^t):\bm \varepsilon^t(\bm \psi)+\int_{\Gamma_C}j_{\tau}^0(u^t_{\tau}; \psi_{\tau})\geq \int_{\Omega}\gamma_1(t)\bm f\circ F_t\cdot \bm\psi+\int_{\Gamma_N}\bm g_N \cdot \bm \psi,
		\end{equation}	 
		with
		\[
		\bm \varepsilon^t(\bm u^t)=\mathcal{S}({\rm D}\bm u^t	{\rm D}F_t^{-1}),\quad \gamma_1(t)={\rm det}({\rm D}F_t).
		\]
		Let $\bm \psi =-\bm u^t$ in \eqref{eqto}, we have that $\|\bm u^t\|_{1}\leq C$. Let $\bm v=\bm u^t-\bm u$ in \eqref{hemi} and $\bm \psi =\bm u-\bm u^t$ in \eqref{eqto}. Adding these two inequalities, one gets
		\begin{align}\label{comb}
			&\int_{\Omega}\Big[\mathbb{C}\circ F_t\Big((\gamma_1(t)-1)\bm \varepsilon^t(\bm u^t)+\big(\bm \varepsilon^t(\bm u^t)-\bm\varepsilon(\bm u^t)\big)\Big)+\Big(\mathbb{C}\circ F_t\bm \varepsilon(\bm u^t)-\mathbb{C}\bm \varepsilon(\bm u)\Big)\Big]:\bm \varepsilon(\bm u-\bm u^t)\notag\\
			&+\int_{\Omega}\mathbb{C}\circ F_t\gamma_1(t)\bm \varepsilon^t(\bm u^t):\Big(\bm \varepsilon^t(\bm u-\bm u^t)-\bm\varepsilon(\bm u-\bm u^t)\Big)+\int_{\Gamma_C}j_{\tau}^0( u_{\tau}^t; u_{\tau}- u_{\tau}^t )+\int_{\Gamma_C}j_{\tau}^0(u_{\tau};u_{\tau}^t- u_{\tau} )\notag\\
			&\geq\int_{\Omega}(\gamma_1(t)\bm f\circ F_t-\bm f)\cdot(\bm u-\bm u^t)+\int_{\Gamma_N}\bm g_N \cdot(\bm u-\bm u^t).
		\end{align}
		Notice that
		\begin{align}
			\Big(\mathbb{C}\circ F_t\bm \varepsilon(\bm u^t)-\mathbb{C}\bm \varepsilon(\bm u)\Big):\bm \varepsilon(\bm u-\bm u^t)=\Big((\mathbb{C}\circ F_t-\mathbb{C})\bm \varepsilon(\bm u^t)+\mathbb{C}\bm \varepsilon(\bm u^t)-\mathbb{C}\bm \varepsilon(\bm u)\Big):\bm \varepsilon(\bm u-\bm u^t).
		\end{align}	
		Combining \eqref{Eq:pofun1}(d) and the trace inequality in \eqref{trace}, it yields
		\begin{equation}\label{hj}
			\int_{\Gamma_C}\big(j_{\tau}^0( u_{\tau}^t; u_{\tau}-u_{\tau}^t )+j_{\tau}^0(u_{\tau}; u_{\tau}^t- u_{\tau} )\big)\leq \alpha_{j_{\tau}}\lambda_1^{-1}\|\bm u^t-\bm u\|_{1}^2.
		\end{equation}
		Inserting \eqref{hj} into \eqref{comb}, we obtain by \eqref{Eq:pofun0}
		\begin{equation}\label{import}
			\begin{aligned}
				&\int_{\Omega}\Big[\mathbb{C}\circ F_t\Big((\gamma_1(t)-1)\bm \varepsilon^t(\bm u^t)+\big(\bm \varepsilon^t(\bm u^t)-\bm\varepsilon(\bm u^t)\big)\Big)+(\mathbb{C}\circ F_t-\mathbb{C})\bm \varepsilon(\bm u^t)\Big]:\bm \varepsilon(\bm u-\bm u^t)\\
				&+\int_{\Omega}\mathbb{C}\circ F_t\gamma_1(t)\bm \varepsilon^t(\bm u^t):\big(\bm \varepsilon^t(\bm u-\bm u^t)-\bm\varepsilon(\bm u-\bm u^t)\big)-\int_{\Omega}(\gamma_1(t)\bm f\circ F_t-\bm f)\cdot(\bm u-\bm u^t)\\
				&\quad	
				\geq (m_{\mathbb{C}}-\alpha_{j_{\tau}}\lambda_1^{-1})\|\bm u^t-\bm u\|_{1}^2.
			\end{aligned}
		\end{equation}
		Notice that for \eqref{import}
		\begin{equation}\label{efwoijewf}
			\begin{aligned}
				&\int_{\Omega}\mathbb{C}\circ F_t\big((\gamma_1(t)-1)\bm \varepsilon^t(\bm u^t)\big):\bm \varepsilon(\bm u-\bm u^t)\leq \|\mathbb{C}\circ F_t\|_{L^{\infty}(\Omega)}\|\gamma_1(t)-1\|_{L^{\infty}(\Omega)}\Tilde{C}_1\|\bm u^t\|_{1}\| \bm u^t-\bm u\|_{1},\\
				&\int_{\Omega}(\mathbb{C}\circ F_t-\mathbb{C})\bm \varepsilon(\bm u^t):\bm \varepsilon(\bm u-\bm u^t)\leq \|\mathbb{C}\circ F_t-\mathbb{C}\|_{L^{\infty}(\Omega)}\|\bm u^t\|_{1}\| \bm u^t-\bm u\|_{1},\\
				&\int_{\Omega}\mathbb{C}\circ F_t\gamma_1(t)\bm \varepsilon^t(\bm u^t):\big(\bm\varepsilon^t(\bm u-\bm u^t)-\bm\varepsilon(\bm u-\bm u^t)\big)\leq \Tilde{C}_1\Tilde{C}_2\|\mathbb{C}\circ F_t\gamma_1(t)\|_{L^{\infty}(\Omega)}\|\bm u^t\|_{1}\| \bm u^t-\bm u\|_{1},\\
				&	\int_{\Omega}(\gamma_1(t)\bm f\circ F_t-\bm f)(\bm u^t-\bm u)\leq \big(\|\gamma_1(t)-1\|_{L^{\infty}(\Omega)}\|\bm f\circ F_t\|_{L^{2}(\Omega)}+\| \bm f\circ F_t-\bm f\|_{L^{2}(\Omega)} \big)\|\bm u^t-\bm u\|_{1},
			\end{aligned}
		\end{equation}
		where $\Tilde{C}_1 = \max\{\|{\rm D}F_t^{-1}\|_{L^{\infty}(\Omega)},\|{\rm D}F_t^{-T}\|_{L^{\infty}(\Omega)}\}$ and  $\Tilde{C}_2 = \max\{\|{\rm D}F_t^{-1}-I\|_{L^{\infty}(\Omega)},\|{\rm D}F_t^{-T}-I\|_{L^{\infty}(\Omega)}\}$.
		Since
		\begin{equation}\label{common2}
			\begin{aligned}
				&\Big\|{\rm D}F_t^{-1}\Big\|_{L^{\infty}(\overline{D})}=\Big\|{\rm D}F_t^{-T}\Big\|_{L^{\infty}(\overline{D})}\leq C, \\ &\lim_{t\searrow 0}\Big\|\frac{{\rm D}F_t^{-1}-I}{t}+{\rm D}\bm V\Big\|_{L^{\infty}(\overline{D})}=\lim_{t\searrow 0}\Big\|\frac{{\rm D}F_t^{-T}-I}{t}+{\rm D}\bm V^T\Big\|_{L^{\infty}(\overline{D})}=0,\\
				&\lim_{t\searrow 0}\Big\|\frac{\bm f\circ F_t-\bm f}{t}-{\rm D} \bm f\bm V\Big\|_{L^{2}(\overline{D})}=0,
				\\
				&\lim_{t\searrow 0}\Big\|\frac{\gamma_1(t)-1}{t}-{\rm div}\bm V\Big\|_{L^{\infty}(\overline{D})}=0,\quad  \lim_{t\searrow 0}\Big\|\frac{\mathbb{C}\circ F_t-\mathbb{C}}{t}-\nabla \mathbb{C}\cdot\bm V\Big\|_{L^{\infty}(\overline{D})}=0,
			\end{aligned}
		\end{equation}
		with $\nabla \mathbb{C}\cdot\bm V=\{\nabla c_{ijkl}\cdot \bm V\}\ (i,j,k,l=1,\cdots d)$, which can be used in \eqref{efwoijewf}, we combine with \eqref{small} and the fact $\|\bm u^t\|_{1}\leq C$ to obtain the desired result \eqref{result} from \eqref{import}.
		
	\end{proof}
	
	Let $\bm z^t=({\bm u^t-\bm u})/{t}$. By Lemma \ref{weak}, one gets $\|\bm z^t\|_{1}\leq C$. Then, we have
	\begin{equation}\label{jiojio}
		\bm z^t \rightharpoonup \dot{\bm{u}}\in \bm V_1 \quad {\rm as}\ t\rightarrow 0^+.
	\end{equation}
	Furthermore, if we choose $\gamma \boldsymbol{v}=v_\tau$ for $\boldsymbol{v} \in \boldsymbol{W}$, where $\gamma: H^1(\Omega) \rightarrow L^2\left(\Gamma_C\right)$ is a trace operator, we can obtain strong convergence by compactness of $\gamma$ from (\cite{BS2005}) as 
	\begin{equation}\label{strongconv}
		\frac{u_{\tau}^t- u_{ \tau}}{t}\rightarrow \dot{u}_\tau \text{ in } L^2(\Gamma_C)  \quad {\rm as}\ t\rightarrow 0^+.
	\end{equation}
	
	\begin{lem}\label{Lemma33}
		For $\bm u,\  \bm u^t\in \bm V_1$, there hold as
		\begin{align}
			&\|\bm\varepsilon^t(\bm u^t)-\bm \varepsilon(\bm u^t)\|_{Q}\leq Ct, \label{eq1} \\   
			&\|\bm\varepsilon(\bm u^t)-\bm \varepsilon(\bm u)\|_{Q}\leq Ct. 
			\label{eq2}
		\end{align}
	\end{lem}
	\begin{proof}
		\begin{equation*}
			\begin{aligned}
				\bm\varepsilon^t(\bm u^t)-\bm \varepsilon(\bm u^t)&=\mathcal{S}({\rm D}\bm u^t{\rm D}F_t^{-1})-\mathcal{S}({\rm D}\bm u^t)=\mathcal{S}({\rm D}\bm u^t({\rm D}F_t^{-1}-I)).
			\end{aligned}
		\end{equation*}
		Since the second line of \eqref{common2} holds and $\|\bm u^t\|_{1}\leq C$,
		we can derive \eqref{eq1}. From Lemma \ref{weak}, we obtain \eqref{eq2}.
	\end{proof}
	
	\begin{thm}\label{thm1}
		For the shape functional $J$ in \eqref{Obj} with Lipschitz boundary $\Gamma$, the Eulerian derivative holds as
		\begin{align}\label{Euler1}
			dJ(\Omega;\bm V)=&\int_{\Omega}\bigg[\bigg(\frac{1}{2}\bm \sigma(\bm u):\bm \varepsilon(\bm u)-\bm f\cdot\bm u\bigg){I}-{\rm D}\bm u^T\mathbb{C}\bm \varepsilon(\bm u)\bigg]:{\rm D}\bm V
			+\bigg(\frac{1}{2}\bm \varepsilon(\bm u):\bm \varepsilon(\bm u)\nabla\mathbb{C}-{\rm D}\bm f^T\bm u\bigg)\cdot\bm V.
		\end{align}
	\end{thm}
	
	\begin{proof}
		By definition of the Eulerian derivative,
		\begin{align}
			dJ(\Omega;\bm V)=&\lim_{t\searrow 0}\frac{J(\Omega_t)-J(\Omega)}{t}=\frac{1}{2}\mathrm{I}-\mathrm{II}+\mathrm{III}-\mathrm{IV},
		\end{align}
		with
		\begin{align*}
			{\mathrm{I} }:&= \lim_{t\searrow 0} \frac{1}{t}\bigg[\int_{\Omega_t} \mathbb{C}\bm \varepsilon(\bm u_t): \bm \varepsilon(\bm u_t)-\int_{\Omega} \mathbb{C}\bm \varepsilon(\bm u): \bm \varepsilon(\bm u)\bigg]\\
			{\mathrm{II}}:&= \lim_{t\searrow 0} 
			\frac{1}{t}\bigg[\int_{\Omega_t} \bm f\cdot\bm u_t-\int_\Omega \bm f\cdot\bm u\bigg]\\
			{\mathrm{III}}: &= \lim_{t\searrow 0} \int_{\Gamma_C}\frac{j_{\tau}(u_{\tau t})-j_{\tau}( u_{\tau })}{t}\\
			{\mathrm{IV}}:&= \lim_{t\searrow 0}\frac{1}{t}(\bm {g},\bm{u}_t-\bm{u})_{\Gamma_N}.
		\end{align*}
		Notice that
		\begin{align*}
			{\mathrm{I} }:&= \lim_{t\searrow 0} \frac{1}{t}\bigg[\int_{\Omega}\gamma_1(t) \mathbb{C}\circ F_t\bm \varepsilon^t(\bm u^t):\bm \varepsilon^t(\bm u^t)- \int_{\Omega} \mathbb{C}\bm \varepsilon(\bm u): \bm \varepsilon(\bm u)\bigg]\\
			& = {\mathrm{I}_1}+{\mathrm{I}}_2+{\mathrm{I}_3}+{\mathrm{I}_4}\\
		\end{align*}
		with
		\begin{align*}
			&{\mathrm{I}_1}:=\lim_{t\searrow 0}\frac{\int_{\Omega}\big[(\gamma_1(t)-1)\mathbb{C}\circ F_t\bm \varepsilon^t(\bm u^t)+\mathbb{C}\circ F_t\big(\bm \varepsilon^t(\bm u^t)-\bm \varepsilon(\bm u^t)+ \bm \varepsilon(\bm u^t)-\bm \varepsilon(\bm u)\big)\big]:\bm \varepsilon(\bm u^t)}{t}\\
			&{\mathrm{I}_2}:=\lim_{t\searrow 0}\frac{\int_{\Omega}(\gamma_1(t)\mathbb{C}\circ F_t \bm \varepsilon^t(\bm u^t):(\bm \varepsilon^t(\bm u^t)-\bm \varepsilon(\bm u^t))}{t}\\
			&{\mathrm{I}_3}:=\lim_{t\searrow 0}\frac{\int_{\Omega}(\mathbb{C}\circ F_t-\mathbb{C})\bm \varepsilon(\bm u):\bm \varepsilon(\bm u^t)}{t}\\
			&{\mathrm{I}_4}:=\lim_{t\searrow 0}\frac{\int_{\Omega}\mathbb{C}\bm \varepsilon(\bm u):\big(\bm \varepsilon(\bm u^t)-\bm \varepsilon(\bm u)\big)}{t}.
		\end{align*}
		On one hand,
		\begin{align*}
			&{\mathrm{I}_1-\Big(\int_\Omega\mathbb{C}\bm \varepsilon(\bm u):\bm \varepsilon(\bm u){\rm div}\bm V-\mathbb{C}\bm\varepsilon(\bm u):\mathcal{S}({\rm D}\bm u{\rm D}\bm V)+\mathbb{C} \bm \varepsilon(\dot{\bm u}):\bm \varepsilon(\bm u)}\Big)= \mathcal{Q}_1+\mathcal{Q}_2+\mathcal{Q}_3,
		\end{align*}
		where
		\begin{align*}
			&\mathcal{Q}_1:=\lim_{t\searrow 0}\int_{\Omega}{\widetilde{\gamma_1(t)}\mathbb{C}\circ F_t\bm \varepsilon^t(\bm u^t):\bm \varepsilon(\bm u^t)}-\int_\Omega\mathbb{C}\bm \varepsilon(\bm u):\bm \varepsilon(\bm u){\rm div}\bm V,\\
			& \mathcal{Q}_2:=\lim_{t\searrow 0}\frac{\int_{\Omega}  \mathbb{C}\circ F_t(\bm \varepsilon^t(\bm u^t)-\bm \varepsilon(\bm u^t)): \bm \varepsilon(\bm u^t)
			}{t}+\int_\Omega\mathbb{C}\bm\varepsilon(\bm u):\mathcal{S}({\rm D}\bm u{\rm D}\bm V),\\
			&\mathcal{Q}_3:=\lim_{t\searrow 0}\frac{\int_{\Omega}\mathbb{C} \circ F_t( \bm \varepsilon(\bm u^t)-\bm \varepsilon(\bm u) ):\bm \varepsilon(\bm u^t)}{t}-\int_\Omega\mathbb{C}\varepsilon(\dot{\bm u}):\bm \varepsilon(\bm u),
		\end{align*}
		with $\widetilde{\gamma_1(t)}=(\gamma_1(t)-1)/t$.
		Here, by \eqref{common2} and Lemma \ref{Lemma33}, one can get
		\begin{align*}
			|\mathcal{Q}_1|=&	\bigg|\lim_{t\searrow 0}\int_\Omega\widetilde{\gamma_1(t)}\mathbb{C}\circ F_t\bm \varepsilon^t(\bm u^t):(\bm \varepsilon(\bm u^t)-\bm \varepsilon(\bm u))+\lim_{t\searrow 0}\int_\Omega\widetilde{\gamma_1(t)}\mathbb{C}\circ F_t(\bm \varepsilon^t(\bm u^t)-\bm \varepsilon(\bm u^t)):\bm \varepsilon(\bm u)\\
			&+\lim_{t\searrow 0}\int_\Omega\widetilde{\gamma_1(t)}\mathbb{C}\circ F_t(\bm \varepsilon(\bm u^t)-\bm \varepsilon(\bm u)):\bm \varepsilon(\bm u)+\lim_{t\searrow 0}\int_\Omega\widetilde{\gamma_1(t)}(\mathbb{C}\circ F_t-\mathbb{C})\bm \varepsilon(\bm u):\bm \varepsilon(\bm u)\\
			&+\lim_{t\searrow 0}\int_\Omega\big(\widetilde{\gamma_1(t)}-{\rm div}\bm V\big)\mathbb{C}\bm \varepsilon(\bm u):\bm \varepsilon(\bm u)\bigg|	=0,
		\end{align*}
		where each term in the absolute value can be shown to be 0. Similarly, we have
		\[|Q_2|=0,\quad |Q_3|=0.\]
		Then, we derive
		\begin{equation}\label{i1}
			{\mathrm I}_1=\int_\Omega\mathbb{C}\bm \varepsilon(\bm u):\bm \varepsilon(\bm u){\rm div}\bm V-\mathbb{C}\bm\varepsilon(\bm u):\mathcal{S}({\rm D}\bm u{\rm D}\bm V)
			+\mathbb{C} \bm \varepsilon(\dot{\bm u}):\bm \varepsilon(\bm u).
		\end{equation}
		Notice that
		\begin{align*}
			&\quad {\mathrm{I}_2}+\int_{\Omega}\mathbb{C}\bm \varepsilon(\bm u):\mathcal{S}({\rm D}\bm u {\rm D}\bm V)\\
			&= \lim_{t\searrow 0} \int_{\Omega}\frac{1}{t}{\Big(\gamma_1(t) \mathbb{C}\circ F_t\bm \varepsilon^t(\bm u^t)-\mathbb{C}\bm \varepsilon^t(\bm u^t)+\mathbb{C}\bm \varepsilon^t(\bm u^t)-\mathbb{C}\bm \varepsilon(\bm u^t)+\mathbb{C}\bm \varepsilon(\bm u^t)-\mathbb{C}\bm \varepsilon(\bm u)\Big):(\bm \varepsilon^t(\bm u^t)-\bm \varepsilon(\bm u^t))}\\
			&\quad\quad+\int_{\Omega}\mathbb{C}\bm \varepsilon(\bm u):\Big[ \frac{\bm \varepsilon^t(\bm u^t)-\bm \varepsilon(\bm u^t)}{t}+\mathcal{S}({\rm D}\bm u {\rm D}\bm V)\Big] = \mathcal{R}_1+\mathcal{R}_2+\mathcal{R}_3, 
		\end{align*}
		where
		\begin{align*}
			&\mathcal{R}_1:=\lim_{t\searrow 0}\frac{\int_{\Omega}(\gamma_1(t) \mathbb{C}\circ F_t\bm \varepsilon^t(\bm u^t)-\mathbb{C}\bm \varepsilon^t(\bm u^t)):(\bm \varepsilon^t(\bm u^t)-\bm \varepsilon(\bm u^t))}{t}+\frac{\int_{\Omega}(\mathbb{C}\bm \varepsilon^t(\bm u^t)-\mathbb{C}\bm \varepsilon(\bm u^t):(\bm \varepsilon^t(\bm u^t)-\bm \varepsilon(\bm u^t))}{t},\\
			&\mathcal{R}_2:=\lim_{t\searrow 0}\frac{\int_{\Omega}(\bm \varepsilon(\bm u^t)-\bm \varepsilon(\bm u)):(\bm \varepsilon^t(\bm u^t)-\bm \varepsilon(\bm u^t))}{t},\\
			&\mathcal{R}_3:=\lim_{t\searrow 0}\int_{\Omega}\mathbb{C}\bm \varepsilon(\bm u):\bigg[\frac{\bm \varepsilon^t(\bm u^t)-\bm \varepsilon(\bm u^t)}{t}+\mathcal{S}({\rm D}\bm u {\rm D}\bm V)\bigg].
		\end{align*}
		By \eqref{common2} and Lemma \ref{Lemma33}, we derive 
		\begin{equation}\label{r1}
			\mathcal{R}_1=0, \quad \mathcal{R}_2=0,\quad 	\mathcal{R}_3=0.
		\end{equation}
		Then
		\begin{equation}\label{i2}
			{\mathrm{I}_2}=-\int_{\Omega}\mathbb{C}\bm \varepsilon(\bm u):\mathcal{S}({\rm D}\bm u {\rm D}\bm V).
		\end{equation}
		By simple calculation, we have
		\begin{align}\label{i3}
			\Big|{\mathrm{I}_3}-\int_{\Omega}\nabla \mathbb{C}\cdot\bm V\bm \varepsilon(\bm u):\bm \varepsilon(\bm u)\Big|=\Big|\lim_{t\searrow 0}
			&\int_\Omega\Big[\frac{\mathbb{C}\circ F_t-\mathbb{C}}{t}\bm \varepsilon(\bm u)-\nabla\mathbb{C}\cdot\bm V \bm \varepsilon(\bm u)\Big]:\bm \varepsilon(\bm u^t)\notag\\
			&+\int_{\Omega}\nabla\mathbb{C}\cdot\bm V \bm \varepsilon(\bm u):(\bm \varepsilon(\bm u^t)-\bm \varepsilon(\bm u))\Big|=0,
		\end{align}
		and
		\begin{equation}\label{i4}
			\Big|{\mathrm{I}_4}- \int_\Omega\mathbb{C}\varepsilon(\dot{\bm u}):\bm \varepsilon(\bm u)\Big|
			= \Big|\lim_{t\searrow 0}\int_{\Omega}\mathbb{C}\bm \varepsilon(\bm u):\Big(\frac{\bm \varepsilon(\bm u^t)-\bm \varepsilon(\bm u)}{t}- \bm \varepsilon(\dot{\bm u})\Big)\Big|=0.
		\end{equation}
		By \eqref{i1}, \eqref{i2}, \eqref{i3} and \eqref{i4}, we have
		\begin{equation}\label{I}
			\mathrm{I}=\int_{\Omega}\mathbb{C}\bm \varepsilon(\bm u):\bm \varepsilon(\bm u){\rm div}\bm V+\int_{\Omega}2\mathbb{C}\bm \varepsilon(\bm u):\big[\bm \varepsilon(\dot{\bm u})-\mathcal{S}({\rm D}\bm u {\rm D}\bm V)\big]+\int_{\Omega}\nabla \mathbb{C}\cdot\bm V\bm \varepsilon(\bm u):\bm \varepsilon(\bm u).
		\end{equation}
		Similarly, we can derive
		\begin{equation}\label{II}
			\mathrm{II}=\int_{\Omega}\big({\rm D}\bm f\bm V\cdot \bm u+\bm f\cdot\dot{\bm u}+\bm f\cdot{\bm u}{\rm div} \bm V\big).
		\end{equation}
		Notice that
		\begin{align*}
			&\bigg|{\mathrm{III}} -\int_{\Gamma_C}j_{\tau}^0( u_{\tau};\dot{ u}_{ \tau})\bigg|
			=\bigg| \lim_{t\searrow 0}\int_{\Gamma_C} \frac{j_{\tau}(u_{\tau}^t)-j_{\tau}(u_{\tau })}{t}-\int_{\Gamma_C}j_{\tau}^0( u_{\tau};\dot{u}_{\tau})
			\bigg|\\
			\leq& \bigg|\lim_{t\searrow 0} \int_{\Gamma_C}\frac{j_{\tau}( u_{\tau}+t\dot{ u}_{\tau})-j_{\tau}( u_{\tau })}{t}-\int_{\Gamma_C}j_{\tau}^0(u_{\tau};\dot{ u}_{ \tau})\bigg|+\bigg|\lim_{t\searrow 0} \int_{\Gamma_C}\frac{j_{\tau}(u_{\tau}+t\frac{ u_{\tau}^t- u_{\tau}}{t})-j_{\tau}( u_{\tau }+t\dot{ u}_{\tau})}{t}\bigg|=0,
		\end{align*}
		where \eqref{Eq:pofun1} (b) and \eqref{strongconv} have been used.
		Then, we have
		\begin{equation}\label{III}
			\mathrm{III}=\int_{\Gamma_C}j_{\tau}^0(u_{\tau};\dot{ u}_{\tau}).
		\end{equation}
		Furthermore, since there is no deformation occurring on the boundary $\Gamma_N$,
		\begin{align*}
			&\bigg|{\mathrm{IV}} -\int_{\Gamma_N}\bm g_N\cdot \dot{\bm u}\bigg|\leq C\lim_{t\searrow 0}  \|\bm{z}^t-\dot{\bm{u}} \|_1\|\bm g_N\|_{L^2(\Gamma_N)^d}=0,
		\end{align*}
		where \eqref{jiojio} and the trace theorem have been used. Therefore, we have
		\begin{equation}\label{IV}
			\mathrm{IV}=\int_{\Gamma_N}\bm g_N\cdot \dot{\bm u}.
		\end{equation}
		Combining \eqref{I}, \eqref{II}, \eqref{III} and \eqref{IV}, we derive
		\begin{equation}\label{mid}
			\begin{aligned}
				dJ(\Omega;\bm V)=&\frac{1}{2}\int_{\Omega}\mathbb{C}\bm \varepsilon(\bm u):\bm \varepsilon(\bm u){\rm div}\bm V+\int_{\Omega}\mathbb{C}\bm \varepsilon(\bm u):\big[\bm \varepsilon(\dot{\bm u})-\mathcal{S}({\rm D}\bm u{\rm D}\bm V)\big]+\int_{\Gamma_C}j_{\tau}^0(u_{\tau};\dot{ u}_{\tau})\\
				&+\frac{1}{2}\int_{\Omega}\nabla \mathbb{C}\cdot\bm V\bm \varepsilon(\bm u):\bm \varepsilon(\bm u)-\int_{\Omega}\big({\rm D}\bm f\bm V\cdot \bm u+\bm f\cdot\dot{\bm u}+\bm f\cdot{\bm u}{\rm div} \bm V\big)-\int_{\Gamma_N}\bm g_N\cdot\dot{\bm u}.
			\end{aligned}
		\end{equation}
		
		On the other hand, we choose $\bm v=\dot{\bm u}$ and $\bm v=-\dot{\bm u}$ in  hemivariational inequality \eqref{hemi}. Then, we derive
		\begin{equation} \label{varia2}
			(\mathbb{C} {\bm \varepsilon}(\bm u),{\bm \varepsilon}(\bm \dot{\bm u}))_{\Omega}+ \hat{J}^0\left(u_\tau ; \dot{u}_\tau\right)
			-(\bm f,\dot{\bm u})_{\Omega}-(\bm g_N ,\dot{\bm u})_{\Gamma_N}\geq 0,
		\end{equation}
		and 
		\begin{equation} \label{varia22}
			(\mathbb{C} {\bm \varepsilon}(\bm u),{\bm \varepsilon}(-\bm \dot{\bm u}))_{\Omega}+\hat{J}^0\left(u_\tau ;-\dot{u}_\tau\right)
			+(\bm f,\dot{\bm u})_{\Omega} +(\bm g_N ,\dot{\bm u})_{\Gamma_N}\geq 0.
		\end{equation}
		respectively. By \eqref{psi1} and \eqref{psi2}, we have
		\begin{equation}\label{equiv}
			\hat{J}^0\left(u_\tau ;-\dot{u}_\tau\right)=-\hat{J}^0\left(u_\tau ; \dot{u}_\tau\right)
		\end{equation}
		Inserting \eqref{equiv} into \eqref{varia22}, we derive
		\begin{equation} \label{varia2b}
			(\mathbb{C} {\bm \varepsilon}(\bm u),{\bm \varepsilon}(\bm \dot{\bm u}))_{\Omega}+\hat{J}^0\left(u_\tau ; \dot{u}_\tau\right)
			-(\bm f,\dot{\bm u})_{\Omega} -(\bm g_N ,\dot{\bm u})_{\Gamma_N}\leq 0.
		\end{equation}
		Combining \eqref{varia2} and \eqref{varia2b}, we have
		\begin{equation} \label{result1}
			(\mathbb{C} {\bm \varepsilon}(\bm u),{\bm \varepsilon}(\bm \dot{\bm u}))_{\Omega}+\hat{J}^0\left(u_\tau ; \dot{u}_\tau\right)
			-(\bm f,\dot{\bm u})_{\Omega} -(\bm g_N ,\dot{\bm u})_{\Gamma_N}=0.
		\end{equation}
		Inserting \eqref{result1} into \eqref{mid} implies that \eqref{Euler1} holds.

	\end{proof}

	\section{Shape optimization based on regularization and shape sensitivity analysis}\label{sec4}
	\subsection{Regularization on the shape functional of energy type in \eqref{Obj}}\label{sec4.1}
	In order to design a gradient-type shape optimization algorithm, we need to differentiate the state w.r.t. shape. The non-smoothness difficulty caused by the state and the corresponding potential functional $j_{ \tau}(u_{ \tau})$ can be overcome by regularization. Then a standard method is applied to the approximating differentiable case. Shape sensitivity analysis for variational inequalities was performed in \cite{CD20,CD21,Chaudet23}. 
	For any small fixed $\epsilon>0$, a smoothing function of $|s_1|$ can be given by
	\begin{equation}\label{smooth}
		\theta( s_1,\epsilon)=\left\{%
		\begin{aligned}
			&| s_1|,\quad{\rm if }\ |s_1|> \epsilon,\\
			&\Lambda(\epsilon;s_1),\quad{\rm if } \ | s_1| \leq  \epsilon,
		\end{aligned}
		\right.
	\end{equation}
	where $$\Lambda(\epsilon;s_1)=-\frac{1}{8\epsilon^3}| s_1|^4+\frac{3}{4\epsilon}\bm | s_1|^2+\frac{3\epsilon}{8}.$$ Define
	\begin{equation}\label{smooth-potential}
		j_{\tau,\epsilon}(u_{\tau})=\int_0^{\theta(u_{\tau},\epsilon)} \mu(t)\, \mathrm{d} t.
	\end{equation}
	By similar arguments as in \cite[Lemma 3.1]{MAJ17}, $j_{\tau,\epsilon}$ is twice Fr\'{e}chet differentiable. The dependence of the regularization term $\theta\left(s_1, \epsilon\right)$ on the small parameter $\epsilon$ can be described as follows: When $\epsilon$ is relatively large, $\theta\left(s_1, \epsilon\right)$ and $j_{\tau, \epsilon}\left(u_\tau\right)$ vary smoothly with the input $u_\tau$, primarily showing a linear relationship. However, as $\epsilon$ decreases, $\theta\left(s_1, \epsilon\right)$ becomes steeper near $u_\tau=0$, and the polynomial part significantly influences the numerical solution. This indicates that with smaller $\epsilon$, the regularization term exerts a stronger constraint on the stability and smoothness of the solution.
	
	Then, we consider the objective in \eqref{Obj} subject to an approximated regularized system
	\begin{equation}\label{State}
		\left\{
		\begin{aligned}
			&-\text{div}\,\bm \sigma(\bm u_{\epsilon})=\bm f &\text{in} &\,\,\Omega,\\
			&\bm \sigma(\bm u_{\epsilon}) \bm \nu = \bm g_N &\text{on}  &\,\,\Gamma_N,\\
			&\bm \sigma(\bm u_{\epsilon}) \bm \nu = \bm 0 &\text{on}  &\,\Gamma,\\
			&\bm  u_{\epsilon} =\bm 0 &\text{on}  &\,\Gamma_D,\\
			&u_{\nu,\epsilon} =0 &\text{on}  &\,\Gamma_C,\\
			&- \sigma_{\tau}( u_{\tau,\epsilon})=\partial j_{\tau,\epsilon}( u_{\tau,\epsilon}) &\text{on}  &\,\Gamma_C,
		\end{aligned}
		\right.
	\end{equation}
	where 
		\begin{equation}\label{smeq}
			\partial j_{\tau,\epsilon}( u_{\tau,\epsilon})\\=\left\{
			\begin{aligned}
				&(a-b)\frac{ u_{\tau,\epsilon}}{| u_{ \tau,\epsilon}|}e^{-\alpha|u_{ \tau,\epsilon}|}+b\frac{ u_{\tau,\epsilon}}{| u_{\tau,\epsilon}|}, \quad{\rm if}\ | u_{ \tau,\epsilon}|>\epsilon,\\
				&(a-b)\bigg(-\frac{ u_{\tau,\epsilon}^3 }{2\epsilon^3}+\frac{3u_{ \tau,\epsilon}}{2\epsilon}\bigg)e^{-\alpha\Lambda(\epsilon;u_{ \tau,\epsilon})}+b\bigg(-\frac{u_{ \tau,\epsilon}^3}{2\epsilon^3}+\frac{3 u_{ \tau,\epsilon}}{2\epsilon}\bigg),\quad{\rm if}\  | u_{\tau,\epsilon}|\leq  \epsilon.
			\end{aligned}
			\right.
		\end{equation}
	The variational form of \eqref{State} reads: Find $\bm u_{\epsilon}\in \bm V_1$ such that
	\begin{equation} \label{varias}
		a(\bm u_{\epsilon},\bm v)+ (\partial  j_{ \tau,\epsilon}( u_{\tau,\epsilon}),v_{\tau})_{\Gamma_C}= (\bm f,\bm v)_{\Omega}+(\bm g_N,\bm v)_{\Gamma_N}
		\quad\forall\,\bm v\in \bm V_1,
	\end{equation}
	which is equivalent to the following minimization problem
	\begin{equation}\label{mini}
		\min_{\bm v\in \bm V_1}\ \frac{1}{2}a(\bm v,\bm v)+\int_{\Gamma_C}j_{\tau,\epsilon}(v_{\tau})-(\bm f,\bm v)_{\Omega}-(\bm g_N,\bm v)_{\Gamma_N}.
	\end{equation}
	We assume that $j_{\tau,\epsilon}$ satisfies \eqref{Eq:pofun1} $(d)$. 
	According to \cite[Theorem 4.3]{Han21}, there exists a unique minimizer $\bm u_\epsilon\in \bm V_1$ to \eqref{mini}, i.e., \eqref{varias}. 
	To illustrate the rationale behind regularization, we provide the following asymptotic analysis.
	\begin{thm} \label{thm2}
		The solution $\bm u_{\epsilon}$ in \eqref{varias} converges to $\bm u$ in \eqref{varia} strongly in $\bm V_1$ as $\epsilon\rightarrow 0^+$.
	\end{thm}
	\begin{proof}
		We consider the regularized formulation of \eqref{hemi}: Find $\bm u_{\epsilon}\in \bm V_1$ such that
		\begin{equation}\label{hemi1}
			\quad a(\bm u_{\epsilon},\bm v)+\int_{\Gamma_{C}}j_{\tau,\epsilon}^0(u_{ \tau,\epsilon}; v_{\tau})  \geq (\bm f,\bm v)_{\Omega}+(\bm g_N,\bm v)_{\Gamma_N}\quad \forall\,\bm v\in \bm V_1.
		\end{equation}
		Taking $\bm v= \bm u-\bm u_{\epsilon}$ in \eqref{hemi1} and $\bm v= \bm u_{\epsilon}-\bm u$ in \eqref{hemi} and adding the two resulting inequalities, we derive
		\begin{equation}\label{eq:mid}
			\begin{aligned}
				&a(\bm u-\bm u_{\epsilon},\bm u-\bm u_{\epsilon})\\
				\leq &\int_{\Gamma_{C}}\Big(j_{ \tau,\epsilon}^0( u_{\tau,\epsilon}; u_{\tau}- u_{ \tau,\epsilon})+j^0_{\tau}( u_{ \tau}; u_{ \tau,\epsilon}- u_{ \tau}) \Big)\\
				=  &\int_{\Gamma_{C}}\Big(j_{ \tau}^0( u_{\tau,\epsilon}; u_{\tau}- u_{ \tau,\epsilon})+j^0_{\tau}( u_{ \tau}; u_{ \tau,\epsilon}- u_{ \tau}) \Big)+\Big(j_{ \tau,\epsilon}^0( u_{\tau,\epsilon}; u_{\tau}- u_{ \tau,\epsilon})-j_{ \tau}^0( u_{\tau,\epsilon}; u_{\tau}- u_{ \tau,\epsilon}) \Big).
			\end{aligned}
		\end{equation}
		By \eqref{Eq:pofun1}(d) and the trace inequality \eqref{trace}, we obtain
		\begin{equation}\label{eq:mid2}
			\begin{aligned}
				\int_{\Gamma_{C}}\Big(j_{ \tau}^0( u_{\tau,\epsilon}; u_{\tau}- u_{ \tau,\epsilon})+j^0_{\tau}( u_{ \tau}; u_{ \tau,\epsilon}- u_{ \tau}) \Big)\,\,{\rm d}s 
				&\leq \alpha_{j_{\tau}}\|  u_{ \tau,\epsilon}- u_{ \tau}\|^2_{\Gamma_C}\\
				&\leq \alpha_{j_{\tau}} \lambda_1^{-1} \|\bm u- \bm u_{\epsilon}\|_{1}^2.
			\end{aligned}
		\end{equation}
		Inserting \eqref{eq:mid2} into \eqref{eq:mid}, we have
		\begin{align}
			a(\bm u-\bm u_{\epsilon},\bm u-\bm u_{\epsilon})\leq \int_{\Gamma_C}\Big(j_{ \tau,\epsilon}^0( u_{\tau,\epsilon}; u_{\tau}- u_{ \tau,\epsilon})-j_{ \tau}^0( u_{\tau,\epsilon}; u_{\tau}- u_{ \tau,\epsilon}) \Big)\,{\rm d}s + \alpha_{j_{\tau}} \lambda_1^{-1} \|\bm u- \bm u_{\epsilon}\|_{1}^2.
		\end{align}
		Applying the coercivity of the bilinear form $a(\cdot,\cdot)$
		in \eqref{Eq:pofun0}(c), we derive
		\begin{align}\label{resi}
			(m_{\mathbb{C}}-\alpha_{j_{\tau}}\lambda_1^{-1} )\|\bm u- \bm u_{\epsilon}\|_{1}^2 \le \int_{\Gamma_C}\Big(j_{ \tau,\epsilon}^0( u_{\tau,\epsilon}; u_{\tau}- u_{ \tau,\epsilon})-j_{ \tau}^0( u_{\tau,\epsilon}; u_{\tau}- u_{ \tau,\epsilon}) \Big).
		\end{align}
		From \eqref{smeq}, we have
		\begin{enumerate}
			\item If $|u_{\tau,\epsilon}|>\epsilon$,  we obtain immediately
			\[
			j_{ \tau,\epsilon}^0( u_{\tau,\epsilon}; u_{\tau}- u_{ \tau,\epsilon})-j_{ \tau}^0( u_{\tau,\epsilon}; u_{\tau}- u_{ \tau,\epsilon})=0.
			\]
			\item If $|u_{\tau,\epsilon}|<\epsilon$.
			As $\epsilon\searrow 0$, we have
			\begin{equation}\label{derive}
				j_{ \tau,\epsilon}^0( u_{\tau,\epsilon}; u_{\tau}- u_{ \tau,\epsilon})-j_{ \tau}^0( u_{\tau,\epsilon}; u_{\tau}- u_{ \tau,\epsilon})=(\partial j_{\tau,\epsilon}(u_{\tau,\epsilon}), u_{\tau}-u_{\tau,\epsilon})-j_{ \tau}^0( u_{\tau,\epsilon}; u_{\tau}- u_{ \tau,\epsilon})\leq 0.
			\end{equation}
		\end{enumerate}
		Now, by the smallness condition $m_{\mathbb{C}}>\alpha_{j_{\tau}}\lambda_1^{-1}$ and \eqref{resi}, we have
		$$\lim_{\epsilon\searrow 0}\|\bm u- \bm u_{\epsilon}\|_{1}=0.$$
	\end{proof}
	
	Define a Lagrangian
	\begin{align}
		\mathcal{L}(\Omega,\bm v,\bm w)&=\frac{1}{2}a(\bm v,\bm v)+\int_{\Gamma_C}j_{\tau,\epsilon}(v_{\tau})-(\bm f,\bm v)_\Omega-(\bm g_N,\bm v)_{\Gamma_N}\\\nonumber
		& \quad	+( \mathbb{C}\bm{\varepsilon}(\bm v),\bm{\varepsilon}(\bm w))_Q + ( \partial j_{\tau,\epsilon}(v_{\tau}),w_{\tau} )_{\Gamma_C}  -(\bm f,\bm w)_\Omega-(\bm g_N,\bm w)_{\Gamma_N},
	\end{align}
	where $\bm v,\bm w\in H^1(\mathbb{R}^d)^d$. The saddle point of $\mathcal{L}$ is characterized by 
	\begin{align}
		\frac{\partial \mathcal{L}}{\partial \bm w}(\Omega,\bm u_{\epsilon},\bm p)\delta \bm w=\frac{\partial \mathcal{L}}{\partial \bm v}(\Omega,\bm u_{\epsilon},\bm p)\delta \bm v=0,
	\end{align}
	for any $\delta \bm w,\delta \bm v\in H^1(\mathbb{R}^d)^d$. Then, we derive the following adjoint problem
	\begin{equation}\label{Adjoint1}
		\left\{
		\begin{aligned}
			&-\text{div}\,\bm \sigma(\bm p)=\text{div}\,\bm \sigma(\bm u_{\epsilon})+\bm f &\text{in} &\,\,\Omega,\\
			&\bm \sigma(\bm p) \bm \nu = \bm 0 &\text{on}  &\,\,\Gamma_N\cup \Gamma ,\\
			&\bm p=\bm 0 &\text{on}  &\,\Gamma_D,\\
			&p_{\nu} =0 &\text{on}  &\,\Gamma_C,\\
			& \sigma_{\tau}( p_{\tau})= - j''_{\tau,\epsilon}(u_{\tau,\epsilon})  p_{ \tau} &\text{on}  &\,\Gamma_C.
		\end{aligned}
		\right.
	\end{equation}
	According to \eqref{State}, we have $\text{div}\,\bm \sigma(\bm u_{\epsilon})+\bm f=0$. 
	Note that $\bm p = \bzero$ is the solution of \eqref{Adjoint1}. Then, we have
	\begin{equation}\label{cost_smooth}
		\mathcal{L}(\Omega,\bm u_{\epsilon},\bm p) = \frac{1}{2}a(\bm u_{\epsilon},\bm u_{\epsilon})+\int_{\Gamma_C}j_{\tau,\epsilon}(u_{\tau,\epsilon})-(\bm f,\bm u_{\epsilon})_\Omega-(\bm g_N,\bm u_{\epsilon})_{\Gamma_N}\equiv J_\epsilon(\Omega)
	\end{equation}
	as an approximation of objective $J(\Omega)$ in \eqref{Obj}.
	By the similar technique in Theorem \ref{thm1}, we can derive the following Eulerian derivative:
	\begin{align}\label{Euler2}
		dJ_\epsilon(\Omega;\bm V)=&\int_{\Omega}\bigg[\bigg(\frac{1}{2}\mathbb{C}\bm \varepsilon(\bm u_{\epsilon}):\bm \varepsilon(\bm u_{\epsilon})-\bm f\cdot\bm u_{\epsilon}\bigg)I-{\rm D}\bm u_{\epsilon}^T\mathbb{C}\bm \varepsilon(\bm u_{\epsilon})\bigg]:{\rm D}\bm V+\bigg(\frac{1}{2}\bm \varepsilon(\bm u_{\epsilon}):\bm \varepsilon(\bm u_{\epsilon})\nabla\mathbb{C}-{\rm D}\bm f^T\bm u_{\epsilon}\bigg)\cdot\bm V.
	\end{align}
	
	Now, we provide an asymptotic analysis for Eulerian derivatives of the regularization approach in shape optimization problem governed by hemivariational inequality. 
	\begin{prop}
		It holds that $dJ_\epsilon(\Omega;\bm V)$ \eqref{Euler2} converges to $dJ(\Omega;\bm V)$ \eqref{Euler1} as $\epsilon\searrow 0$.
	\end{prop}
	\begin{proof}
		By Theorem \ref{thm2}, we have $\bm u_{\epsilon}\rightarrow \bm u\,$ as $\epsilon\searrow 0$. By a direct derivation, we have
		\begin{align}
			|	dJ_\epsilon(\Omega;\bm V)-dJ(\Omega;\bm V)|=&\Big|\int_{\Omega}\Big[\frac{1}{2}\mathbb{C}\bm \varepsilon(\bm u_{\epsilon}):\bm \varepsilon(\bm u_{\epsilon})-\frac{1}{2}\mathbb{C}\bm \varepsilon(\bm u):\bm \varepsilon(\bm u)+(\bm f\bm u-\bm f\bm u_{\epsilon})\Big]I:{\rm D}\bm V\notag\\
			&-\int_{\Omega}\Big[{\rm D}\bm u_{\epsilon}^T\mathbb{C}\bm \varepsilon(\bm u_{\epsilon})-{\rm D}\bm u^T\mathbb{C}\bm \varepsilon(\bm u)\Big]:{\rm D}\bm V
			-\int_{\Omega}\Big[{\rm D}\bm f^T\bm u_{\epsilon}-{\rm D}\bm f^T\bm u\Big]\cdot\bm V\notag\\
			&+\int_{\Omega}\Big[\frac{1}{2}\bm \varepsilon(\bm u_{\epsilon}):\bm \varepsilon(\bm u_{\epsilon})\nabla\mathbb{C}-\frac{1}{2}\bm \varepsilon(\bm u):\bm \varepsilon(\bm u)\nabla\mathbb{C}\Big]\cdot\bm V\Big|.
		\end{align}
		By the boundedness of $\bm u_{\epsilon}$ and $\bm u$ in $H^1(\Omega;\mathbb{R}^d)$ and $\bm u_{\epsilon}\rightarrow \bm u$ as $\epsilon\searrow 0$ in $H^1(\Omega;\mathbb{R}^d)$, we derive
		\begin{align*}
			\Big|\mathbb{C}\bm \varepsilon(\bm u_{\epsilon}):\bm \varepsilon(\bm u_{\epsilon})-\mathbb{C}\bm \varepsilon(\bm u):\bm \varepsilon(\bm u)\Big|=\Big|\mathbb{C}\bm \varepsilon(\bm u_{\epsilon}):\bm \varepsilon(\bm u_{\epsilon}-\bm u)+\mathbb{C}\bm \varepsilon(\bm u_{\epsilon}-\bm u):\bm \varepsilon(\bm u)\Big|\rightarrow 0\quad{\rm as}\ \epsilon\searrow 0.
		\end{align*}
		Similarly, we obtain 
		\[
		|	dJ_\epsilon(\Omega;\bm V)-dJ(\Omega;\bm V)|\rightarrow 0\quad{\rm as}\ \epsilon\searrow 0.
		\]
	\end{proof}

	
	\begin{remark}
		It should be noted that the strict sensitivity analysis in Theorem \ref{thm1} is only applicable to energy functional of \eqref{Obj}. For a general shape functional, deriving an Eulerian derivative is not as straightforward as it is for the energy functional. Considering the rationality of regularization in the above derivation, we consider using regularization methods to handle the non-smooth term in a general shape functional.
	\end{remark}
	
	\subsection{General shape functional}
	We consider shape design model problem of elastic structures in hemi-variational inequalities:
	\begin{equation}\label{J}
		\min_{{\rm Vol}(\Omega)=C} J(\bm u) = \int_{\Omega} m(\bm u)  + \int_{\partial \Omega} r(\bm u) ,
	\end{equation}
	where $\bm u$ satisfies \eqref{Eq:constitute} and $m(\cdot),r(\cdot):\mathbb{R}^d\rightarrow\mathbb{R}$ with their derivatives denoted respectively by $m'$ and $r'$ are Lipschitz continuous. Assume, moreover, that for all $\bm u, \bm v \in \mathbb{R}^d$,
	\begin{align*}
		|m(\bm u)|\le C_1(1+|\bm u|^2),\quad |r(\bm u)|\leq C_1(1+|\bm u|^2),\\
		|m'(\bm u)\cdot \bm v|\leq C_1|\bm u\cdot \bm v|,\quad |r'(\bm u)\cdot \bm v|\leq C_1|\bm u\cdot \bm v|,
	\end{align*}
	for some constants $C_1>0$, where $|\cdot|$ denotes the Euclidean norm. Notice that  this general shape functional in \eqref{J} was studied for shape design in variational inequalities \cite{CD20,CD21,MAJ17}.
	
	We use the Lagrange multiplier method to deal with the mechanical constraint \eqref{varias} as a smoothing approximation to that in \eqref{J}. Define 
	\begin{equation}
		\begin{aligned}
			L(\Omega,\bm v,\bm q)= J(\bm{v}) + ( \mathbb{C}\bm{\varepsilon}(\bm v),\bm{\varepsilon}(\bm q))_Q + ( \partial j_{\tau,\epsilon}(v_{\tau}),q_{\tau} )_{\Gamma_C}  -(\bm f,\bm q)_{\Omega}-(\bm g_N,\bm q)_{\Gamma_N}.
		\end{aligned}
	\end{equation}
	At a saddle point, the directional derivative of the Lagrangian in a direction $\bm \psi\in \bm V_1$ w.r.t. $\bm q$ vanishes as 
	\begin{equation*}
		\begin{aligned}
			\frac{\partial L}{\partial \bm q}(\Omega,\bm u_{\epsilon},\bm p_\epsilon)(\bm \psi)=& (\mathbb{C}  \bm{\varepsilon}(\bm u_{\epsilon}),\bm{\varepsilon}(\bm \psi) )_Q +(\partial j_{\tau,\epsilon}(u_{\tau,\epsilon}) , \psi_{\tau})_{\Gamma_C}-(\bm f,\bm \psi )_\Omega - (\bm g_N , \bm \psi)_{\Gamma_N}  \\
			=&(-\big(\text{div}\, \mathbb{C} \bm{\varepsilon}(\bm u_{\epsilon})+\bm f \big) , \bm \psi )_{\Omega}+(\mathbb{C} \bm{\varepsilon}(\bm u_{\epsilon}) \bm \nu -\bm g_N),\bm \psi)_{\Gamma_N}
			- (\mathbb{C}  \bm{\varepsilon}(\bm \psi)\bm \nu , \bm u )_{\Gamma_D} \\
			&+(\mathbb{C}  \bm{\varepsilon}(\bm u_{\epsilon}) \bm \nu ,\bm \psi )_{\Gamma} +\big( \mathbb{C}  \bm{\varepsilon}(\bm u_{\epsilon}) \bm \nu \cdot \bm \tau  +\partial j_{\tau,\epsilon}(u_{\tau,\epsilon}),  \psi_{\tau} \big)_{\Gamma_C} \\
			=&0,
		\end{aligned}
	\end{equation*}
	which implies the regularized state problem (\ref{State}).
	The derivative of the Lagrangian w.r.t. $\bm v$ in a direction $\bm \psi\in \bm V_1$ at the saddle point reads:
	\begin{equation}
		\begin{aligned}
			\frac{\partial L}{\partial \bm v}(\Omega,\bm u_{\epsilon},\bm p_\epsilon)(\bm \psi)
			=&( m'(\bm u_{\epsilon}), \bm \psi)  +
			(r'(\bm u_{\epsilon}), \bm \psi) + (\bm{\sigma}(\bm \psi) ,\bm{\varepsilon}(\bm p_\epsilon) )_Q+ (j''_{\tau,\epsilon}(u_{\tau,\epsilon})p_{\epsilon,\tau}, \psi_{\tau})_{\Gamma_C}   
			\\
			=& \big(-\text{div}( \bm{\sigma}(\bm p_\epsilon)) + m'(\bm u_{\epsilon}), \bm \psi\big)_\Omega + ( \bm{\sigma}(\bm p_\epsilon)\bm \nu +r'(\bm u_{\epsilon}), \bm \psi)_{\Gamma_N} +( \bm{\sigma}(\bm p_\epsilon)\bm \nu +r'(\bm u_{\epsilon}), \bm \psi)_{\Gamma_D}
			\\
			&+  \big(  \bm{\sigma}(\bm p_\epsilon)\bm \nu \cdot \bm \tau  + j''_{\tau,\epsilon}( u_{\tau,\epsilon})  p_{\tau} +r'( \bm u_{\epsilon})\cdot\bm \tau, \psi_{\tau}\big)_{\Gamma_C}+\big(\bm{\sigma}(\bm p_\epsilon) \bm \nu \cdot\bm \nu +r'(\bm u_{\epsilon}) \cdot\bm \nu , \psi_{\nu}  \big)_{\Gamma_C} \\
			&+( \bm{\sigma}(\bm p_\epsilon)\bm \nu + r'(\bm u_{\epsilon}), \bm \psi)_\Gamma 
			=0,
		\end{aligned}
	\end{equation}
	which implies the adjoint system
	\begin{equation} \label{Adjiont_pro}
		\left\{
		\begin{aligned}
			-\text{div}\, \bm \sigma(\bm p_\epsilon) & = -m'(\bm u_{\epsilon}) \qquad &&\text{ in } \Omega, \\
			\bm \sigma(\bm p_\epsilon) \bm \nu &= -r'(\bm u_{\epsilon}) \qquad &&\text{ on } \Gamma_N \cup \Gamma, \\
			\bm p_\epsilon & = \bm{0} \qquad &&\text{ on } \Gamma_D, \\
			p_{\epsilon,\nu} & =0  \qquad &&\text{ on } \Gamma_C, \\
			-\sigma_{ \tau}( p_{\epsilon,\tau})- j''_{\tau,\epsilon}(u_{\tau,\epsilon})  p_{\epsilon,\tau}  & =r'( \bm u_{\epsilon})\cdot\bm \tau  &&\text{ on } \Gamma_C,
		\end{aligned}
		\right.
	\end{equation}
	where
	\begin{equation}
		j''_{\tau,\epsilon}( u_{\tau,\epsilon}) =\left\{
		\begin{aligned}
			&-\alpha(a-b)e^{-\alpha| u_{\tau,\epsilon}|}, \quad{\rm if}\ | u_{\tau,\epsilon}|>\epsilon,\\
			&(a-b)\Big(-\frac{3 \vert u_{\tau,\epsilon} \vert^2}{2\epsilon^3}+\frac{3}{2\epsilon}-\frac{\alpha \vert  u_{\tau,\epsilon} \vert^6}{4\epsilon^6}+\frac{3\alpha \vert u_{ \tau,\epsilon} \vert^4}{2\epsilon^4}-\frac{9 \alpha \vert  u_{ \tau,\epsilon} \vert^2}{4\epsilon^2}\Big)e^{-\alpha \Lambda(\epsilon;u_{\tau,\epsilon})}\\
			&\quad+b\bigg(-\frac{3\vert  u_{ \tau,\epsilon} \vert^2}{2\epsilon^3}+\frac{3}{2\epsilon}\bigg), \quad{\rm if}\  | u_{\tau,\epsilon}|\leq  \epsilon.
		\end{aligned}
		\right.
	\end{equation}
	The weak form of \eqref{Adjiont_pro} reads: Find $\bm p_\epsilon \in \bm V_1$ such that
	\begin{equation}\label{Adjoint}
		\begin{aligned}
			(\mathbb{C}  \bm{\varepsilon}(\bm p_\epsilon) ,\bm{\varepsilon}(\bm v) )_Q + ( j''_{\tau,\epsilon}( u_{\tau,\epsilon})p_{\epsilon,\tau},v_{\tau} )_{\Gamma_C} = -( m'(\bm u_{\epsilon}), \bm v )&_{\Omega}-
			( r'(\bm u_{\epsilon}), \bm v)_{\Gamma_N}-
			( r'(\bm u_{\epsilon}), \bm v)_{\Gamma} \\
			&- (r'(\bm u_{\epsilon})\cdot \bm \tau,v_{\tau})_{\Gamma_C}\quad \forall\, \bm{v}\in \bm V_1.
		\end{aligned}
	\end{equation}
	In the following text, $\bm u_{\epsilon}$ and $\bm p_{\epsilon}$ are simplified respectively as $\bm u$ and $\bm p$  for notational simplicity.
	\begin{remark}
		For compliance 
		$$J(\Omega)= \int_{\Gamma_N} \bm g_N \cdot \bm u ,$$
		the associated adjoint state $\bm p \in \bm V_1$ for shape optimization is defined as weak solution of
		\begin{equation}\label{Adjiont_compli}
			(\mathbb{C}  \bm{\varepsilon}(\bm p) ,\bm{\varepsilon}(\bm v) )_Q + ( j''_{\tau,\epsilon}( u_{\tau})p_{\tau},v_{\tau} )_{\Gamma_C} =-(\bm g_N, \bm v)_{\Gamma_N} \quad \forall\, \bm{v}\in V_1.
		\end{equation}
	\end{remark}
\begin{lem} \label{lem4}
Let $f \in W^{1,1}\left(\mathbb{R}^{d}\right)$. Then $F(\Omega)=\int_{\Omega} f$ is shape differentiable at any open set  $\Omega$ and its Eulerian derivative reads for $\bm V \in W^{1,\infty}\left(\mathbb{R}^{d}; \mathbb{R}^{d}\right)$: $$dF(\Omega;\bm V)=\int_{\Omega} \operatorname{div} (f \bm V). $$
If, in addition, $\Omega$ is bounded and Lipschitz,
$$ dF(\Omega;\bm V)=\int_{\partial \Omega} f V_\nu,$$
where $V_\nu=\bm V \cdot \bm \nu$.
\end{lem}

\begin{lem} \label{lem5}
Let $\Omega$  be class  $\mathcal{C}^{2}$. Let $g \in W^{2,1}\left(\mathbb{R}^{d}\right)$. For $\bm V \in \mathcal{C}^{1, \infty}\left(\mathbb{R}^{d};\mathbb{R}^{d}\right)$, then
$$G(\Omega):=\int_{\partial \Omega} g$$
has an Eulerian derivative
$$
dG(\Omega;\bm V)=\int_{\partial \Omega}\left(\frac{\partial g}{\partial \bm \nu}+\kappa g\right)V_\nu,
$$
where the mean curvature $\kappa={\rm div}\, \bm \nu$.
\end{lem}

The augmented Lagrangian method is used to deal with the volume constraint in \eqref{J}. Let
\begin{equation}\label{calJ}
\mathcal{J}(\Omega)=  \int_{\Omega}m(\bm u) + \int_{\partial \Omega } r( \bm u)+\ell ({\rm Vol}(\Omega) - C) + \frac{\gamma}{2}({\rm Vol}(\Omega) -C)^2
\end{equation}
with $\ell$ and $\gamma$ being a Lagrange multiplier being a penalty factor, respectively.

\begin{prop} Let assumptions in Lemma 1 and Lemma 2 hold. Then, $\mathcal{J}(\Omega)$ \eqref{calJ} 
is shape differentiable at any shape $\Omega \in \mathcal{A}$, and its Eulerian derivative
\begin{equation}\label{eqcon}
	d\mathcal{J}(\Omega;\bm V) = \int_{\Gamma}  \bigg(m(\bm u)+\mathbb{C} \bm{\varepsilon}(\bm u):\bm{\varepsilon}(\bm p)-\bm f \cdot \bm p + \frac{\partial r(\bm u)}{\partial \bm \nu} + \kappa r(\bm u) + \ell + \gamma ({\rm Vol}(\Omega) -C) \bigg) V_\nu.
\end{equation}
\end{prop}

\begin{proof}
Introduce a Lagrangian for the physical constraint as
\begin{equation}
	\begin{aligned}
		\mathcal{L}(\Omega,\bm u,\bm p)=& \mathcal{J}(\Omega)+ ( \mathbb{C}  \bm{\varepsilon}(\bm u),\bm{\varepsilon}(\bm p))_Q + (\partial j_{\tau,\epsilon}(u_{\tau}),  p_{\tau} )_{\Gamma_C}  -(\bm f , \bm p )_{\Omega} - ( \bm g_N , \bm p )_{\Gamma_N}.
	\end{aligned}
\end{equation}
The Eulerian derivative can be obtained by differentiating
\begin{equation*}
	\mathcal{J}(\Omega)=\mathcal{L}(\Omega,\bm u(\Omega),\bm p(\Omega)).
\end{equation*}
By the chain rule theorem (cf. \cite{GCF2021}), it reduces to the partial derivative of $\mathcal{L}$ w.r.t. $\Omega$ in the direction $\bm V$:
\begin{equation*}
	d\mathcal{J}(\Omega)(\bm V)=\frac{\partial \mathcal{L}}{\partial \Omega} (\Omega,\bm u(\Omega), \bm p(\Omega) ).
\end{equation*}
According to Lemma~\ref{lem4} and Lemma~\ref{lem5}, we obtain
\begin{equation}
	\begin{aligned}
		d\mathcal{J}(\Omega;\bm V) =& \int_{\partial \Omega} \big( m(\bm u)+\mathbb{C}  \bm{\varepsilon}(\bm u):\bm{\varepsilon}(\bm p)- \bm f \cdot \bm p + \ell + \gamma ({\rm Vol}(\Omega) -C) \big) V_\nu\\
		&- \int_{\Gamma_N} \bigg(\frac{\partial (\bm g_N \cdot \bm p) }{\partial \bm \nu}+ \kappa \bm g_N \cdot \bm p \bigg) V_\nu  + \int_{\partial \Omega} \bigg( \frac{\partial r(\bm u)}{\partial \bm \nu} + \kappa r(\bm u) \bigg) V_\nu \\
		& + \int_{\Gamma_C} \bigg(\frac{\partial (\partial j_{\tau,\epsilon}( u_{\tau})p_{\tau})}{\partial \bm \nu}+ \kappa (\partial j_{\tau,\epsilon}( u_{\tau}) p_{\tau} ) \bigg) V_\nu.
	\end{aligned}
\end{equation}
Taking into account the boundary condition $\bm u=\bm p=0$ on $\Gamma_D$ implies that
\begin{equation}\label{joijw}
	\begin{aligned}
		d\mathcal{J}(\Omega;\bm V) =& \int_{\Gamma_N} \bigg(m(\bm u)+ \mathbb{C}  \bm{\varepsilon}(\bm u):\bm{\varepsilon}(\bm p) -\bm f \cdot \bm p - \frac{\partial (\bm g_N \cdot \bm p)}{\partial \bm \nu}- \kappa (\bm g_N \cdot \bm p) \bigg) V_\nu    \\
		& + \int_{\Gamma_C} \bigg(m(\bm u)+ \mathbb{C} \bm{\varepsilon}(\bm u):\bm{\varepsilon}(\bm p) - \bm f \cdot \bm p + \frac{\partial (\partial j_{\tau,\epsilon}( u_{\tau}) p_{\tau})}{\partial \bm \nu}+ \kappa (\partial j_{\tau,\epsilon}(u_{\tau}) p_{\tau} ) \bigg) V_\nu  \\
		&+ \int_{\Gamma}  \big(m(\bm u)+\mathbb{C}  \bm{\varepsilon}(\bm u):\bm{\varepsilon}(\bm p)-\bm f \cdot \bm p \big) V_\nu - \int_{\Gamma_D} \big( m(\bm u)+ \mathbb{C}  \bm{\varepsilon}(\bm u):\bm{\varepsilon}(\bm p) \big) V_\nu  \\
		&+ \int_{\partial \Omega} \bigg (\ell + \gamma ({\rm Vol}(\Omega) -C) + \frac{\partial r(\bm u)}{\partial \bm \nu} + \kappa r(\bm u) \bigg) V_\nu.
	\end{aligned}
\end{equation}
Since $\Gamma_C,\Gamma_D$ and $\Gamma_N$ do not deform, \eqref{joijw} allows \eqref{eqcon} to hold.
\end{proof}

To increase smoothness as well as extend velocity, we introduce a $H^{1}$-gradient flow \cite{GLZ}: Find $\mathcal{V} \in H^1(D;\mathbb{R}^d)$ such that $\forall\, \bm V \in H^1(\Omega;\mathbb{R}^d)$
\begin{equation}\label{Smoothing}
\int_{\Omega} \omega {\rm D}\mathcal{V} : {\rm D}\bm V +\mathcal{V} \cdot \bm V = - dJ(\Omega;\bm V),
\end{equation} 
where $\omega >0$ is diffusion parameter. The domain deformations are realized by moving meshes after discretizations. A suitable time step size $\tau_n>0$ is set to ensure objective to decrease.

In addition, the Uzawa scheme is employed for updating the Lagrange multiplier as
\begin{equation}
\begin{aligned}\label{Eq30}
	& \ell_{n+1}=\ell_n + \gamma_n({\rm Vol}(\Omega)-C),\\
	& \gamma_{n+1}=\rho_\gamma  \gamma_n,\quad n=0,1,2,\cdots,
\end{aligned}
\end{equation}
where $\gamma_n>0$ and a increasing factor $\rho_\gamma > 1$ is introduced for better numerical performance. An upper bound of $\rho_\gamma$ can be set to avoid numerical instabilities. The algorithm stops if the iteration attains the maximal number denoted by $N_m$ or the stopping criteria holds:
\begin{equation}\label{stop}
\frac{\vert J(\Omega_{n+1})-J(\Omega_n) \vert}{|J(\Omega_n)|} < {\rm Tol},
\end{equation}
where Tol denotes the error tolerance. The whole shape optimization algorithm for a structural hemivariational inequality is organized in Algorithm 1.
\begin{algorithm}[htbp]
\DontPrintSemicolon
\SetAlgoLined
\KwIn {$N_m,\rho,a,b,\alpha,\epsilon,C$}
initialization ($n\leftarrow 0$): $\Omega_0$, $\ell_0,\gamma_0$\;
\While {\rm algorithm does not stop}{
	Solve the regularized hemivariational inequality \eqref{varias}.\;
	Solve adjoint problem \eqref{Adjoint}. \;
	Solve the gradient flow \eqref{Smoothing}.\;
	Update Lagrange multiplier and penalty factor by \eqref{Eq30}. \;
	$\Omega_{n+1} \gets \Omega_{n} + \tau_n \mathcal{V}$.\;
	$n\leftarrow n+1$.\;
}
\KwOut {Final shape}
\caption{Shape optimization of a structural hemivariational inequality}\label{Agl1}
\end{algorithm}

\section{Topology optimization with two phase field methods}\label{sec5}

\subsection{Phase field method (I)}
Consider the elastic structure represented by the phase field function as shown in Fig. \ref{Phase_Feild}. The elasticity system for the state is defined over the entire domain $D$ using the ersatz material method. $\Omega_0$ is assumed to be filled with a weak material to mimic void to avoid singularities in the stiffness matrix and $\Omega_1$ denotes the elastic material whose elasticity tensor is $A$. In addition, interpolation is required for the material in the diffuse interface $\xi$.
\begin{figure}[htbp]
\centering
\includegraphics[width=10cm,height=4cm]{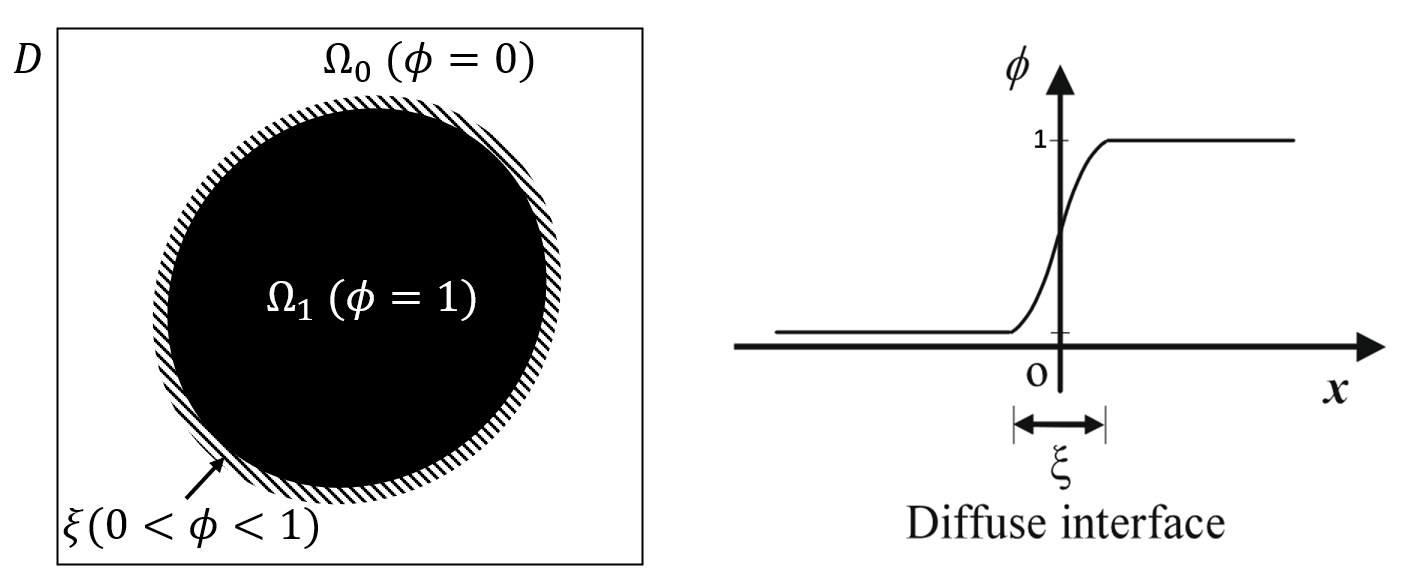}
\caption{Implicit domain representation by a phase field function.}\label{Phase_Feild}
\end{figure}
The subdomain representation in $D$ is then formulated as (see Fig. \ref{Phase_Feild} for illustration)
\begin{equation}
\left \{ \begin{aligned}
	&\phi(\bm x)=1, &&\text{ if } \bm x \in \Omega_1, \\
	& 0< \phi(\bm x) <1, && \text{ if } \bm  x \in \xi,  \\
	& \phi(\bm x) = 0, && \text{ if } \bm x \in \Omega_0,
\end{aligned} \right.
\end{equation}
where $(\Omega_1 \setminus \xi) \supset \Omega$ and $(\Omega_0 \cup \xi) \supset D \setminus \Omega$. The Van der Waals free energy of the system is given by
\begin{equation}
\mathcal M(\phi)= \int_{D} \bigg(\frac{\varsigma  }{2} \vert \nabla \phi \vert^2 + \frac{1}{\varsigma} \psi(\phi)\bigg){\rm d}x,
\end{equation}
where $\varsigma >0$ is a diffusive parameter. The first term represents the interaction energy term of the field in mean field theory. The second term represents a double-well potential with $\psi'(0)=\psi'(1)=0$. From \cite{TNK10}, the change of the phase field function w.r.t. time is assumed to be linearly dependent upon the direction in which the free energy function is minimized:
\begin{equation}\label{Allen-Cahn}
{\partial_t\phi} = - \mathcal{N}(\phi)  \mathcal{M}'(\phi),
\end{equation}
where $\mathcal N(\phi)$ is a mobility. Take, e.g, $\mathcal N(\phi)\equiv 1$ for simplicity.

The boundary $\partial D$ of the working domain is partitioned into four disjoint  parts
\[
\partial D=\partial D_{D}\cup \partial D_N\cup\partial D_C\cup \partial D_0,
\]
where the non-intersecting boundaries $\partial D_D$, $\partial D_C$, $\partial D_N$ and $\partial D_0$ correspond to Dirichlet boundary, the contact boundary, the non-homogeneous Neumann and homogeneous Neumann boundary conditions, respectively. Furthermore, it satisfies
{$\Gamma_D\subset \partial D_D,\, \Gamma_C\subset \partial D_C$  and $\Gamma_N\cup\Gamma=\partial D_N\cup \partial D_0$}.

Let the interpolation function $k(\phi)=\min(\phi^p,k_{\rm min})$ with $0<k_{\rm min}\ll 1$ and positive integer $p$ (typically $p=3$). We set a virtual physical property $\mathcal{F}^*$ of the entire domain by
\begin{equation}
\mathcal{F}^{*}(\phi)=\left \{ \begin{aligned}
	& \mathbb{C}, &&\text{ in }   \Omega_1, \\
	& k(\phi) \mathbb{C}, && \text{ in }  \xi,  \\
	& k_{\rm min} \mathbb{C}, && \text{ in }  \Omega_0.
\end{aligned} \right.
\end{equation}
Let 
$
\bm V_D=\{\bm v \in H^1(D;\mathbb{R}^d):\bm v|_{\partial D_D}=\bm 0,\,v_{\nu}|_{\partial D_C}=0 \}.
$
Then the weak form of state reads: Find $\bm u \in \bm V_D$ such that
\begin{equation} \label{State_weak1}
\int_D\mathcal F^{*}(\phi) \bm{\varepsilon}(\bm u):\bm{\varepsilon}(\bm v) +\int_{\partial D_C} \partial j_{\tau,\epsilon}( u_{\tau})v_{\tau} = \int_D \bm f \cdot \bm v +\int_{\partial D_N}  \bm g_N \cdot \bm v \quad \forall\, \bm v \in \bm V_D.
\end{equation}
The weak form of adjoint is to find $\bm p\in \bm V_D$ such that
\begin{equation}\label{Phase_Adjoint}
\begin{aligned}
	\int_D\mathcal F^{*}(\phi) \bm{\varepsilon}(\bm p):\bm{\varepsilon}(\bm v) + \int_{\partial D_C} j''_{\tau,\epsilon}( u_{\tau})  p_{\tau} v_{\tau}= -&\int_{D} m'(\bm u)\cdot \bm v -
	\int_{\partial D_N\cup \partial D_0} r'(\bm u) \cdot \bm v \\
	&-\int_{\partial D_C} r'(\bm u)\cdot \bm \tau v_{\tau} \quad \forall \,\bm{v}\in \bm V_D.
\end{aligned}
\end{equation}   
\begin{remark}
For compliance 
\[ J(\phi)= \int_{\partial D_N} \bm g_N \cdot \bm u,
\] 
the associated adjoint state $\bm p \in \bm V_D$ satisfies:
\begin{align}\label{Phase_Adjoint_J1}
	\int_D\mathcal F^{*}(\phi) \bm{\varepsilon}(\bm p):\bm{\varepsilon}(\bm v) +  \int_{\partial D_C}j''_{\tau,\epsilon}( u_{\tau})  p_{\tau}v_{\tau} = -
	\int_{\partial D_N} \bm g_N \cdot \bm v \quad \forall \,\bm{v}\in \bm V_D.   
\end{align}   
For the energy 
\begin{equation}
	J(\phi) = \frac{1}{2}\int_{D}  \mathcal F^{*}(\phi) \bm \varepsilon(\bm u):\bm \varepsilon(\bm u)+ \int_{\partial D_C} j_{\tau,\epsilon}(u_{\tau}) - \int_{\partial D_N} \bm g_N \cdot \bm u + \int_{D} \bm f \cdot \bm u,
\end{equation}
we obtain the following system for adjoint
\begin{equation} \label{Adjiont_pro1}
	\left\{
	\begin{aligned}
		-{{\rm div}}\,( \mathcal F^{*}(\phi) \bm\varepsilon(\bm p) ) & ={\rm{div}}\,(\mathcal F^{*}(\phi) \bm\varepsilon(\bm u) ) +\bm f  \quad &&{\rm{ in }}\ D \\
		\bm \sigma(\bm p) \bm \nu &= 0\quad &&{\rm{ on }} \ \partial D_N \cup \partial D_0 \\
		\bm p & = \bm{0} \quad &&{\rm{ on }}\ \partial D_D \\
		p_{\nu} & =0  \quad &&{\rm{ on }}\ \partial D_C \\
		\sigma_{ \tau}( p_{ \tau}) +j''_{\tau,\epsilon}(u_{\tau})  p_{ \tau}& = 0  &&{\rm{ on }}\ \partial D_C.
	\end{aligned}
	\right.
\end{equation}
\end{remark}

The phase field function evolves with a pseudo-time $t$ in the interval $[0,T]$ with $T>0$. The evolution equation is formulated as
\begin{equation}\label{Phase_EV}
\left \{ \begin{aligned}
	& {\partial_t \phi}=\kappa_1 \Delta \phi - \psi'(\phi), \ &&{\rm in}\ [0,T]\times D, \\
	&\phi(0,\cdot)=\phi_0 \quad &&{\rm in}\ D,\\
	& \frac{\partial \phi}{\partial \bm \nu} = 0, &&\text{ on } \partial D, 
\end{aligned} \right.
\end{equation}
where $\kappa_1$ is a positive diffusion coefficient and $\phi_0$ is an initial design. 
To move the diffuse interface in the direction in which the objective decreases, we set the double well potential $\psi(\phi)$ to satisfy the following conditions:
\begin{equation}\label{Phase_Condition}
\psi(0)=0, \,\, \psi(1)=h_1 J'(\phi_{t_1}), \,\, \psi'(0)=\psi'(1)=0, \, (h_1>0),
\end{equation}
where $J'(\phi)$ is the sensitivity of the objective w.r.t. $\phi$, $\phi_{t_1}$ is the value of $\phi$ at time $t_1$ with $t_1=0$ and the constant $h_1>0$. Let us set \cite{TNK10}
\begin{equation}\label{Phase_EV1}
\psi(\phi)= \mathcal{H}(\bm x) w(\phi) + \mathcal{G}(\bm x) g(\phi),
\end{equation}
where
\begin{equation}\label{Phase_EV2}
w(\phi)=\phi^2(1-\phi^2), \,\, g(\phi)=\phi^3 (6 \phi^2 -15 \phi +10),
\end{equation}
$\mathcal{H}(\bm x)$ determines the height of the wall of the double well potential, which affects the thickness of the diffuse interface. Set $\mathcal{H}(\bm x)=\frac{1}{4}$ here. Choose $\mathcal{G}(\bm x)= h_1 J'(\phi_{t_1})$. In order to avoid complex parameter settings, we first perform the following normalization:
\begin{equation}\label{Phase_EV3}
\mathcal{G}(\bm x) = \eta \frac{J'(\phi_{t_1})}{\| J'(\phi_{t_1}) \|_{L^2(D)}}.
\end{equation}
Substituting Eqs. \eqref{Phase_EV1}-\eqref{Phase_EV3} into \eqref{Phase_EV}, we obtain
\begin{equation} \label{Phase_Evolution}
{\partial_t}\phi = \kappa_1 \Delta \phi + \phi(1-\phi)\bigg( \phi - \frac{1}{2} -30 \eta \frac{J'(\phi_{t_1})}{\| J'(\phi_{t_1}) \|_{L^2(D)}} (1-\phi)\phi \bigg).
\end{equation}

For better stability in temporal discretization, we consider a ``semi-implicit'' scheme, i.e., implicit discretization for diffusion \cite{QHZ2022} and the special discretization \cite{QHZ2022} for reaction:
\begin{equation}\label{semiimplicit}
\frac{\phi_{n}^{l+1}-\phi_{n}^{l}}{\Delta t}= \kappa_1 \Delta \phi_{n}^{l+1} + \left \{
\begin{aligned}
	\phi_{n}^{l+1}(1-\phi_{n}^{l})y(\phi_{n}^{l}) \quad &\text{ for } y(\phi_{n}^{l}) \leq 0, \\
	\phi_{n}^{l}(1-\phi_{n}^{l+1})y(\phi_{n}^l) \quad &\text{ for } y(\phi_{n}^{l}) > 0, \\
\end{aligned}
\right.
\end{equation}
where $\Delta t>0$ is a suitable time step, $l =0,1,... \mathcal{T} -1$ with $ \mathcal{T} =[T/\Delta t]$, and
\begin{equation}
y(\phi_n^l)=\phi_n^l - \frac{1}{2} - 30 \eta \frac{J'(\phi_{t_1})}{\| J'(\phi_{t_1}) \|_{L^2(D)}} (1-\phi_n^l)\phi_n^l
\end{equation}
with $n$ and $l$ being the outer iteration and inner iteration indexes, respectively. Noting that the box condition $[0,1]$ might fail to be satisfied in numerical computation, we resort to a projection step
\begin{equation}
\phi_n^{l+1} \leftarrow \min \{ \max \{0,\phi_n^{l+1} \}
,1 \}.
\end{equation}

\subsubsection{Sensitivity analysis}
We derive the Fréchet derivative of the objective in the elastic hemivariational inequality w.r.t. the phase field function (see \cite{TNK10} for compliance of linear elasticity).
The general objective is defined as \eqref{J} and $\vert \Omega \vert \approx \int_{D} \phi $. Let $\mathcal{J}(\phi)=J(\bm{u}(\phi))$.
\begin{prop}
The sensitivity holds as
\begin{equation*}
	\mathcal{J}'(\phi) = \mathcal {F^{*}}'(\phi) \bm{\varepsilon}(\bm u):\bm{\varepsilon}(\bm p).
\end{equation*}
\end{prop}
\begin{proof}
First, the Gâteaux derivative of the objective in a direction $\theta$ is
\begin{equation} \label{Phase_de1}
	\begin{aligned}
		\left \langle \mathcal{J}'(\phi),\theta \right \rangle
		&=\int_{D} m'(\bm u)\cdot \bm z  + \int_{\partial D} r'(\bm u)\cdot \bm z,
	\end{aligned}
\end{equation}
where $\bm z = \left \langle \bm u'(\phi), \theta \right \rangle \in \bm V_D$. Using the state \eqref{State_weak1}, we define a Lagrangian
\begin{equation}
	\begin{aligned}
		\mathfrak{L} (\phi,\bm u,\bm p):=& \int_{D} m(\bm u) + \int_{\partial D} r(\bm u) + \int_{D} \mathcal F^{*}(\phi) \bm{\varepsilon}(\bm u):\bm{\varepsilon}(\bm p)  +\int_{\partial D_C} \partial j_{\tau,\epsilon}( u_{\tau})  p_{\tau}\\
		&- \int_{D} \bm f \cdot \bm p -\int_{\partial D_N} \bm g_N \cdot \bm p
	\end{aligned}
\end{equation}
with $\bm u$ and $\bm p$  being the displacements of state and adjoint state, respectively. Then, the derivative can be expressed as
\begin{equation} \label{Lagrain_Pha}
	\left \langle \mathcal{J}'(\phi), \theta \right \rangle= \left \langle \frac{\partial \mathfrak{L} }{\partial \phi} (\phi,\bm u,\bm p), \theta \right \rangle + \left \langle \frac{\partial \mathfrak{L} }{\partial \bm u} (\phi,\bm u,\bm p),\left \langle \bm u'(\phi),\theta \right \rangle \right \rangle.
\end{equation}
For the second term on the right-hand side, it holds that $\forall \bm z\in \bm V_D$
\begin{equation} \label{Second_term}
	\begin{aligned}
		\left \langle \frac{\partial \mathfrak{L} }{\partial \bm u} (\phi,\bm u,\bm p),\bm z \right \rangle =& \int_{D} m'(\bm u)\cdot \bm z+\int_{\partial D} r'(\bm u)\cdot\bm z+\int_{D} \mathcal F^{*}(\phi) \bm{\varepsilon}(\bm p):\bm{\varepsilon}(\bm z)+\int_{\partial D_C} j''_{\tau,\epsilon}(u_{\tau})  p_{\tau}  z_{\tau}= 0
	\end{aligned}
\end{equation}
using \eqref{Phase_Adjoint}. Thus, the second term on the right-hand side of Eq. \eqref{Lagrain_Pha} vanishes. Then, we have from \eqref{Lagrain_Pha}
\begin{equation}
	\left \langle J'(\bm u),\theta \right \rangle=\left \langle \frac{\partial \mathfrak{L} }{\partial \phi} (\phi,\bm u,\bm p),\theta \right \rangle.
\end{equation}
Then, computing the derivative of $\eqref{State_weak1}$ w.r.t. $\phi$ in the direction $\theta$ and setting $\bm v=\bm p$ in \eqref{State_weak1} imply that
\begin{equation}\label{comob}
	\int_{D} \mathcal {F^{*}}' (\phi) \bm{\varepsilon}(\bm u):\bm{\varepsilon}(\bm p)  + \int_{D} \mathcal F^{*}(\phi) \bm{\varepsilon}(\bm z):\bm{\varepsilon}(\bm p)  + \int_{\partial D_C} j''_{\tau,\epsilon} ( u_{\tau} ) z_{\tau}  p_{\tau}  = 0,
\end{equation}
where $\bm z=\left \langle \bm u'(\phi),\theta \right \rangle$. Combining \eqref{comob} with \eqref{Second_term}, we derive
\begin{equation}\label{Phase_de}
	\int_{D} \mathcal {F^{*}}'(\phi) \bm{\varepsilon}(\bm u):\bm{\varepsilon}(\bm p) =\int_{D} m'(\bm u)\cdot \bm z + \int_{\partial D} r'(\bm u) \cdot\bm z.
\end{equation}
Substituting $\eqref{Phase_de}$ into $\eqref{Phase_de1}$ allows the conclusion to hold.
\end{proof}
Analogous to \eqref{calJ}, we use the augmented Lagrangian method for volume constraint. The Lagrangian reads:
\[\mathcal{I}(\phi)=\mathcal{J}(\phi)+ \ell\Sigma(\phi) + \frac{\gamma}{2}\Sigma^2(\phi),\]
where $\Sigma(\phi) =\int_{D} \phi {\rm d}x -C$. Then
\[
\mathcal{I}'(\phi) = \mathcal {F^{*}}'(\phi) \bm{\varepsilon}(\bm u):\bm{\varepsilon}(\bm p)+\ell+ \gamma \Sigma(\phi).
\]
Now we present Algorithm \ref{Agl2} of phase field method (I).

\subsection{Phase field method (II)}
We now consider a second-order regularization \cite{PIJ2023,PIJ2023S} of the cost functional in the phase field method to facilitate the convergence of the gradient flow equations. 
Define $W$ as a double-well potential as
\begin{equation}
W(\varphi)=\frac{1}{2}\left(\varphi^2-1\right)^2.
\end{equation}
For any $\epsilon > 0$ and $\gamma > 0$, the functional $\mathcal{G}_\epsilon: L^1(D) \to \mathbb{R}$ is defined by
\begin{equation}
\mathcal{G}_\epsilon(\varphi) = \frac{1}{2} \int_{\Omega} \frac{1}{\epsilon} \left(-\epsilon \Delta \varphi + \frac{1}{\epsilon} W'(\varphi) \right)^2  + \frac{\gamma}{2} \int_{\Omega} \left(\frac{\epsilon}{2} |\nabla \varphi|^2 + \frac{1}{\epsilon} W(\varphi) \right),
\end{equation}
if $\varphi \in W^{2,2}(D)$. Otherwise, if $\varphi \in L^1(D) \setminus W^{2,2}(D)$, the functional $\mathcal{G}_\epsilon(\varphi)$ is assigned the value $\infty$.

Consider the domain $D$ containing an unknown inclusion $\Omega \subset \subset  D$. The aim is to determine the optimal shape of $\Omega$ such that an objective  $J$ is minimized. The objective $J$ depends on the geometry of $\Omega$ and a physical field $u$, which is defined either within $D$, inside the inclusion $\Omega$, or in the exterior domain $\Omega_0 = D \setminus \overline{\Omega}$.

To represent the unknown inclusion $\Omega$, a design variable $\rho: D \to \{-1, 1\}$ is introduced. This design variable is defined as:
\begin{equation}
\rho = 2 \chi_1 - 1 \quad \text{in } D,
\end{equation}
where $\chi_1$ is the characteristic function of $\Omega$:
\begin{equation}
\chi_1 =
\begin{cases} 
	1, & \text{if } x \in \Omega, \\
	0, & \text{if } x \in D \setminus \Omega.
\end{cases}
\end{equation}
Under this definition, $\rho = 1$ in the inclusion $\Omega$ and $\rho = -1$ in the exterior region $\Omega_0 = D\setminus \overline{\Omega}$.

The phase field method (II) is based on replacing the discontinuous design variable $\rho$ with a continuous phase field function $\varphi: D \to \mathbb{R}$, which serves as an order parameter. In the phase field framework, the interface separating the material components is represented approximately as a narrow transition region around the level set ${\varphi = 0}$, referred to as the diffusive interface. The regions $\Omega$ and $\Omega_0$, corresponding to different materials, are approximately described by the subdomains ${\varphi > 0}$ and ${\varphi < 0}$, respectively.

For sufficiently smooth vector fields $\varphi$ and compactly supported smooth perturbations $\delta \varphi$, the $L^2$-gradient $d \mathcal{G}_\epsilon$ of the functional $\mathcal{G}_\epsilon$ is defined through the variational derivative by the relation:
\begin{equation}
\partial_{\hat{\lambda }}\mathcal{G}_\epsilon(\varphi +\hat{ \lambda }\delta \varphi) \big|_{\hat{\lambda}=0} = \int_{D} d \mathcal{G}_\epsilon \, \delta \varphi.
\end{equation}
Through explicit computation, the $L^2$-gradient $d \mathcal{G}_\epsilon$ is determined as:
\begin{equation}
d \mathcal{G}_\epsilon = -\frac{1}{\epsilon} \Delta \boldsymbol{\beta}_\epsilon + \frac{1}{\epsilon^3} W^{\prime\prime}(\varphi) \boldsymbol{\beta}_\epsilon + \frac{\gamma}{2 \epsilon} \boldsymbol{\beta}_\epsilon,
\end{equation}
where 
\begin{equation}\label{mue}
\boldsymbol{\beta}_\epsilon=W^{\prime}(\varphi)-\epsilon^2 \Delta \varphi,
\end{equation}
and $W^{\prime\prime}(\varphi)$ 
characterizes the curvature of the potential landscape.
\subsubsection{Gradient flow}
The relaxed optimization problem is then given by
\begin{equation}\label{Eqr22}
\begin{aligned}
	& \inf _{\varphi}\left\{\mathcal{J}^*(\varphi)=\int_{\partial D_N} \boldsymbol{g}_N \cdot \boldsymbol{u} +\widetilde{\eta} \mathcal{G}_\epsilon(\varphi)\right\}, \\
	& \text { subject to }
	\int_D \rho(\varphi)\mathbb{C} \bm\varepsilon(\boldsymbol{u}): \bm\varepsilon(\boldsymbol{v})+\int_{\partial D_C}\partial j_{\tau, \epsilon}(u_\tau) v_\tau=\int_{D}\boldsymbol{f} \cdot\boldsymbol{v}+\int_{\partial D_N}\boldsymbol{g}_N\cdot \boldsymbol{v} \quad \forall \boldsymbol{v} \in \bm{V}_D,
\end{aligned}
\end{equation}
where $\widetilde{\eta}>0$ is a constant.
Then
\begin{equation}\label{Eqtw22}
\begin{aligned}
	\left\langle \mathcal{J}^{*\prime}(\varphi), \theta\right\rangle=\int_{\partial D_N} \boldsymbol{g}_N \left\langle\boldsymbol{u}^{\prime}(\varphi), \theta\right\rangle +\widetilde{\eta} \partial \mathcal{G}_\epsilon= \int_{\partial D_N} \boldsymbol{g}_N\cdot \boldsymbol{v} +\widetilde{\eta} \partial \mathcal{G}_\epsilon,
\end{aligned}
\end{equation}
where $\boldsymbol{v}=\left\langle\boldsymbol{u}^{\prime}(\varphi), \theta\right\rangle$. Using the linear elasticity state problem from \eqref{Eqr22}, we define a Lagrangian \begin{equation}
\begin{aligned}
	\Tilde{L}(\varphi, \boldsymbol{u}, \boldsymbol{p})=\int_{\partial D_N} \boldsymbol{g}_N \cdot \boldsymbol{u} +\widetilde{\eta} \mathcal{G}_\epsilon(\varphi)+
	\int_D \rho(\varphi) \mathbb{C}\varepsilon(\boldsymbol{u}): \varepsilon(\boldsymbol{p}) +\int_{\partial D_C}\partial j_{\tau, \epsilon}(u_\tau) p_\tau-\int_{D}\boldsymbol{f}\cdot \boldsymbol{p}-\int_{\partial D_N}\boldsymbol{g}_N\cdot \boldsymbol{p},
\end{aligned}
\end{equation}
where $\boldsymbol{u}$ is the displacement and $\boldsymbol{p}$ is the adjoint state. The derivative can be expressed as
\begin{equation}\label{Eqrt22}
\left\langle \mathcal{J}^{*\prime}(\varphi), \theta\right\rangle=\bigg\langle\frac{\partial \Tilde{L}}{\partial \varphi}(\varphi, \boldsymbol{u}, \boldsymbol{p}), \theta\bigg\rangle+\bigg\langle\frac{\partial \Tilde{L}}{\partial \boldsymbol{u}}(\varphi, \boldsymbol{u}, \boldsymbol{p}),\left\langle\boldsymbol{u}^{\prime}(\varphi), \theta\right\rangle\bigg\rangle.
\end{equation}
Consider the case where the second term is zero. Replacing $\left\langle\boldsymbol{u}^{\prime}(\varphi), \theta\right\rangle$ with $\boldsymbol{v}$, the second term is
\begin{equation}\label{Eqrrt22}
\begin{aligned}
	\bigg\langle\frac{\partial \Tilde{L}}{\partial \boldsymbol{u}}(\varphi, \boldsymbol{u}, \boldsymbol{p}), \boldsymbol{v}\bigg\rangle=\int_{\partial D_N} \boldsymbol{g}_N \cdot \boldsymbol{v} +\int_D \rho(\varphi) \mathbb{C} \varepsilon(\boldsymbol{v}): \varepsilon(\boldsymbol{p}) +\int_{\partial D_C} j^{''}_{\tau, \epsilon}(u_\tau) v_{\tau}p_\tau=0.
\end{aligned}
\end{equation}
In the case $\boldsymbol{p}$ satisfies the adjoint equation, the second term of \eqref{Eqrt22} can be ignored. On the other hand, the derivative of the state equation w.r.t. $\varphi$ in the direction $\theta$ is
\begin{equation}\label{Eqrrty22}
\begin{aligned}
	&\int_D \langle \rho^{\prime}(\varphi),\theta \rangle \mathbb{C}\varepsilon(\boldsymbol{u}): \varepsilon(\boldsymbol{p})
	+\int_D \rho(\varphi) \mathbb{C}\varepsilon(\left\langle\boldsymbol{u}^{\prime}(\varphi), \theta\right\rangle): \varepsilon(\boldsymbol{p})
	+\int_{\partial D_C}j^{''}_{\tau, \epsilon}(u_\tau)(\left\langle\boldsymbol{u}^{\prime}(\varphi), \theta\right\rangle_{\tau}) p_\tau\\&=\int_D \rho^{\prime}(\varphi) \mathbb{C}\varepsilon(\boldsymbol{u}): \varepsilon(\boldsymbol{p})
	+\int_D \rho(\varphi)\mathbb{C} \varepsilon(\boldsymbol{v}): \varepsilon(\boldsymbol{p})
	+\int_{\partial D_C}j^{''}_{\tau, \epsilon}(u_\tau)v_{\tau} p_\tau=0,
\end{aligned}
\end{equation}
where we set $\boldsymbol{\psi}=\boldsymbol{p}$.
If we compare this with \eqref{Eqrrt22} and \eqref{Eqrrty22}, we have
\begin{equation}\label{Eqwr22}
\int_{\partial D_N} \boldsymbol{g}_N \cdot \boldsymbol{v}=\int_D \rho^{\prime}(\varphi) \mathbb{C}\varepsilon(\boldsymbol{u}): \varepsilon(\boldsymbol{p}).
\end{equation}
Substituting \eqref{Eqwr22} into \eqref{Eqtw22}, the gradient of the objective is
\begin{equation}
\mathcal{J}^{*\prime}(\varphi)=\rho^{\prime}(\varphi)\mathbb{C} \varepsilon(\boldsymbol{u}): \varepsilon(\boldsymbol{p}) +\widetilde{\eta} \partial \mathcal{G}_\epsilon.
\end{equation}

The gradient flow of the Willmore functionals has been studied in \cite{PR69,QC38,MY44}, where the $L^2$ gradient flow equation for the cost functional is expressed as:
\begin{equation}
\epsilon^2 \partial_t \varphi =\epsilon \rho^{\prime}(\varphi) \mathbb{C}\varepsilon(\boldsymbol{u}): \varepsilon(\boldsymbol{p})+ \Delta \boldsymbol{\beta}_\epsilon - \frac{1}{\epsilon^2} W^{\prime \prime}(\varphi) \boldsymbol{\beta}_\epsilon - \frac{\gamma}{2} \boldsymbol{\beta}_\epsilon,
\end{equation}
with $\boldsymbol{\beta}_\epsilon$ being defined as in \eqref{mue}. For simplicity, we set $\epsilon = 1$ from \cite{PIJ2023S} and reformulate the equation as:
\begin{equation}\label{noep}
\begin{aligned}
	\partial_t \phi =\rho^{\prime}(\varphi)\mathbb{C} \varepsilon(\boldsymbol{u}): \varepsilon(\boldsymbol{p})+ \Delta \boldsymbol{\beta} - W^{\prime\prime}(\varphi) \boldsymbol{\beta} - \frac{\gamma}{2} \boldsymbol{\beta}, \quad \boldsymbol{\beta} = W^{\prime}(\varphi) - \Delta \varphi \quad \text{in } D \times (0, T). \\
	\nabla \varphi \cdot \bm \nu = \nabla \boldsymbol{\beta} \cdot \bm \nu = 0 \quad \text{on } (0, T) \times \partial D, \quad \varphi|_{t=0} = \varphi_0 \quad \text{in } D.
\end{aligned}
\end{equation}
Equation \eqref{noep} establishes a well-posed initial-boundary value problem for a weakly nonlinear fourth-order parabolic problem. The global existence and uniqueness of strong solutions for this system, as well as for more general phase field models, have been rigorously shown in \cite{PP23}. See Algorithm \ref{Agl3} for solving the topology optimization model in structural hemivariational inequalities. Algorithm \ref{Agl3} associated with Willmore functional is new in literature for topology optimization.

\subsection{Phase field method (I) coupling with topological derivative}\label{sec7}
To mitigate the initial design-dependent issue, the topological derivative can be combined with the phase field method \cite{TNK10}. Similar ideas can be found in the level set method \cite{TZ2023,ADJT2005}. Consider an open and bounded domain $D \subset \mathbb{R}^d$ ($d=2,3$) subject to a non-smooth perturbation confined in a small region $\omega_{\zeta }(\hat{\bm{x}})=\hat{\bm{x}}+\zeta  \omega$ of size $\zeta>0$, such that $\overline{\omega_\zeta } \subset D$, where $\hat{\bm{x}}$ is an arbitrary point of $D$ and $\omega$ represents a fixed domain in $\mathbb{R}^d$. See sketch in Fig. \ref{topological_derivative}. 
\begin{figure}[htbp]
\centering
\includegraphics[width=11cm,height=4cm]{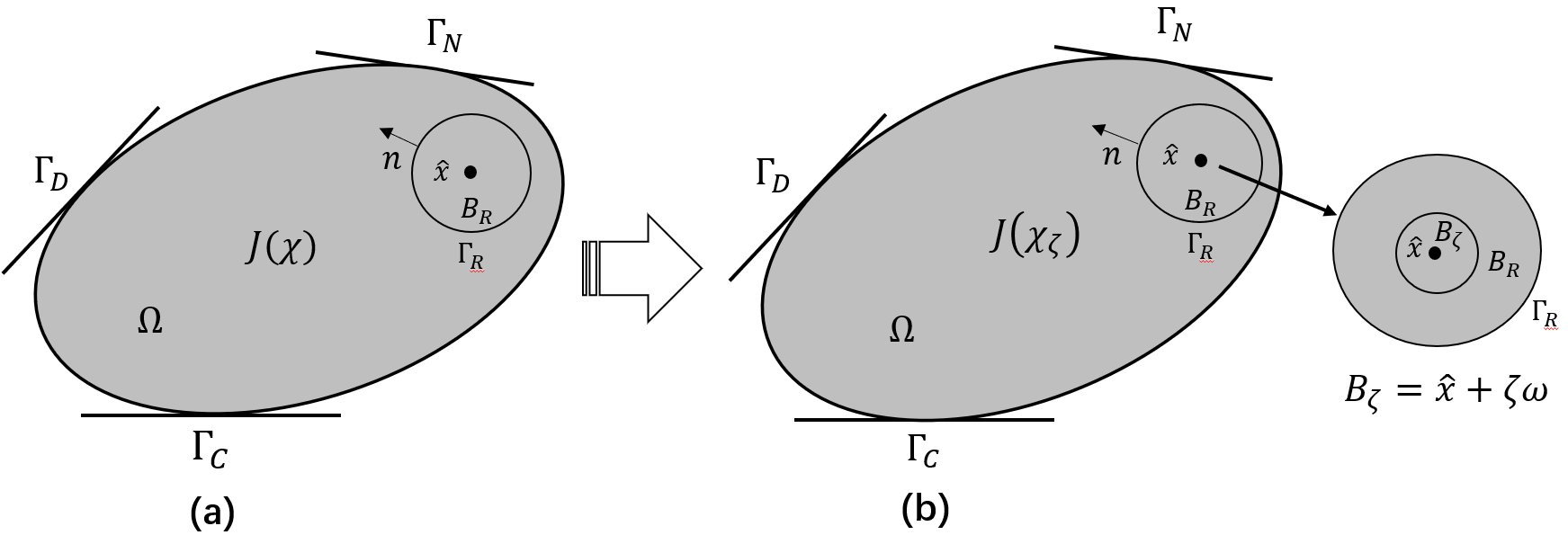}
\caption{Illustration of topological derivative.}\label{topological_derivative}
\end{figure}
We introduce a characteristic function $\chi = \chi(\bm{x})$ for $\bm{x} \in \mathbb{R}^d$, associated with the reference domain, namely $\chi:=1_{D}$, such that
${\rm Vol} (D) = \int_{\mathbb{R}^d} \chi{\rm d}x$.
In the case of a perforation, $\chi_{\zeta }(\hat{\bm{x}}):=1_{D}-1_{\omega_{\zeta }(\hat{\bm{x}})}$ and the
perturbed domain is obtained as $D_{\zeta }=D \textbackslash \omega_{\zeta }$. We assume that a given shape
functional $J(\chi_{\zeta }(\hat{\bm{x}}))$, associated with the topologically perturbed domain, admits a topological asymptotic expansion of the form 
\begin{equation} \label{Eq_TP_2}
J(\chi_{\zeta }(\hat{\bm{x}}))= J(\chi)+f(\zeta )d_TJ(\hat{\bm{x}})+R(\zeta ),
\end{equation}
where $J(\chi)$ is the shape functional associated with the reference domain, $f(\zeta )$ is a positive first-order correction function, which decreases monotonically such that $f(\zeta ) \rightarrow 0$ as $\zeta  \searrow 0$, and $R(\zeta )$ is the remainder term, i.e., $R(\zeta ) / f(\zeta ) \rightarrow 0$ with $\zeta  \searrow 0$. The function $\hat{\bm{x}} \rightarrow d_T J(\hat{\bm{x}})$ is recognized as the \emph{topological derivative} of $J$ at $\hat{\bm{x}}$. In addition, after rearranging \eqref{Eq_TP_2} and taking the limit $\zeta  \searrow 0$, we have the general definition for the topological derivative, namely
\begin{equation} \label{Eq_TP_3}
d_T J(\hat{\bm{x}}):= \lim_{\zeta  \searrow 0} \frac{J(\chi_{\zeta }(\hat{\bm{x}}))-J(\chi)}{f(\zeta )}.
\end{equation}

The topological derivative can be derived for elastic contact problems under given friction. We consider to minimize under the volume constraint the energy of the contact problem:
\begin{equation}\label{jtop}
J(\phi) =\frac{1}{2} \int_{D}  \mathcal F^{*}(\phi) \bm \varepsilon(\bm u):\bm \varepsilon(\bm u)+ \int_{\partial D_C} j_{\tau,\epsilon}(u_{\tau}) - \int_{\partial D_N} \bm g_N \cdot \bm u,
\end{equation}
where $\bm{u}$ satisfies \eqref{State_weak1} with $\bm f\equiv0$. 

We obtain the topological derivative (see derivations in Appendix)
\begin{equation} \label{Topological}
d_{T} J(\bm{x})= - \frac{1}{2} \mathbb{P}_r \bm \sigma(\bm u(\bm{x})):\bm\varepsilon( \bm u(\bm{x})), \,\, \forall \bm{x} \in D,
\end{equation}
where a fourth-order isotropic polarization tensor
\begin{equation}\label{jiojw}
\mathbb{P}_r = \frac{\pi (1-r)}{1+ra_2}\bigg( (1+a_2)\mathbb{I} + \frac{1}{2}(a_1 - a_2) \frac{1-r}{1+ra_1} I \otimes I \bigg).
\end{equation}
In \eqref{jiojw}, $r =\hat{\epsilon}$, $0<\hat{\epsilon}\ll 1$ is the contrast and $\mathbb{I}$ is a fourth-order identity tensor, $I$ is a second-order identity tensor, and the parameters
\begin{equation}\label{a1a2}
a_1 = \frac{\lambda +\mu}{\mu} \quad \text{ and } \quad a_2 = \frac{\lambda +3\mu}{\lambda +\mu},
\end{equation}
where the Lam\'e coefficients $$\lambda=\frac{E \nu}{(1+\nu)(1-2 \nu)} \quad{\rm and}\quad \mu=\frac{E}{2(1+\nu)}$$ with $E$ and $\nu$ being the Young's modulus and Poisson's ratio, respectively. We introduce holes ( using weak material to replace holes) into the domain $D$ with the topological derivative \cite{ADJT2005}. The total volume of the holes made at each stage is limited to 1\% of the total volume of the design domain. 

For the algorithm based on topological derivative, according to the phase field evolution equation, we replace the sensitivity $J'$ in \eqref{Phase_Evolution} with topological derivative and consider to solve
\begin{equation} \label{Phase_Topological}
{\partial_t} \phi= \kappa_1 \Delta \phi + \phi(1-\phi)\bigg( \phi - \frac{1}{2} -30 \eta \frac{d_{T} J(\bm x)}{\| d_{T} J(\bm x) \|_{L^2(D)}} (1-\phi)\phi \bigg).
\end{equation}
A similar space-time discretization scheme was used as for the phase field method in Section \ref{sec4.1}. Notice that topological derivative has been already applied in the phase field method \cite{TNK10}, the level set method \cite{SH2006,He} and 3D discrete variable sensitivity analysis method \cite{SCL24} for structural topology optimization. However, our treatment with topological derivative is different from those in literature (e.g., \cite{TNK10,SH2006,He,SCL24}). Based on the phase field evolution equation \eqref{Phase_Topological}, it has new features of creating holes as well as introducing materials from numerical experiences (see numerical examples in next Section). 

The algorithms for topology optimization without and with topological derivatives are presented in the forms of Algorithm \ref{Agl3} (phase field method (II)) and Algorithm \ref{Agl4} (phase field method (I)), respectively. Both algorithms adopt a similar stopping rule as Algorithm \ref{Agl2}.

\begin{algorithm}[htbp]
\DontPrintSemicolon
\SetAlgoLined
\KwIn{$N_m, \mathcal{T}, \rho, a, b, \alpha, \epsilon, C, r, \kappa_1, \varsigma, \eta$}

\textbf{Initialization:} Set iteration counter $n \leftarrow 0$. Initialize phase field variable $\ell_0$ and Lagrange multiplier $\gamma_0$.\;

\While{\rm stopping criterion is not satisfied}{
	Solve the regularized hemivariational inequality \eqref{State_weak1} to obtain the current state.\;
	
	Compute the adjoint variables by solving \eqref{Phase_Adjoint}.\;
	
	Solve the phase field evolution equation \eqref{Phase_Evolution} and update $\phi$.\;
	
	Compute the new Lagrange multiplier using equation \eqref{Eq30}.\;
	
	Use the new phase field $\phi$ to update the elasticity tensor $\mathcal{F}^*(\phi)$.\;
	
	$n \leftarrow n+1$\;
}

\KwOut{Final topology represented by phase field function $\phi$}
\caption{Topology Optimization of a Structural Hemivariational Inequality Using Phase Field Method (I)}
\label{Agl2}
\end{algorithm}

\begin{algorithm}[htbp]
\DontPrintSemicolon
\SetAlgoLined
\KwIn{$N_m, \mathcal{T}, \rho, a, b, \alpha, \epsilon, C, \gamma$}

\textbf{Initialization:} \\
Set iteration counter $n \leftarrow 0$.\;
Initialize phase field variable $\ell_0$ and Lagrange multiplier $\gamma_0$.\;

\While{\rm stopping criterion is not satisfied}{
	Obtain state variable by solving the regularized hemivariational inequality \eqref{State_weak1}.\;
	
	Compute the adjoint variable by solving the adjoint equation \eqref{Phase_Adjoint}.\;
	
	Solve the phase field evolution equation \eqref{noep} to update $\phi$.\;
	
	Compute the new Lagrange multiplier using equation \eqref{Eq30}.\;
	
	$n \leftarrow n+1$\;
}

\KwOut{Final topology represented by phase field function $\phi$}
\caption{Topology Optimization of a Structural Hemivariational Inequality Using Phase Field Method (II)}
\label{Agl3}
\end{algorithm}

\begin{algorithm}[htbp]
\DontPrintSemicolon
\SetAlgoLined
\KwIn{$N_m, \mathcal{T}, \rho, a, b, \alpha, \epsilon, \hat{\epsilon}, C, r, \kappa_1, \varsigma, \eta$}

\textbf{Initialization:} \\
Set iteration counter $n \leftarrow 0$.\;
Initialize phase field variable $\ell_0$ and Lagrange multiplier $\gamma_0$.\;

\While{\rm stopping criterion is not satisfied}{
	Obtain the state variable by solving the regularized hemi-variational inequality \eqref{State_weak1}.\;
	
	Solve the equation \eqref{Phase_Topological} incorporating topological derivatives \eqref{Topological} to update $\phi$.\;
	
	Compute the new Lagrange multiplier using equation \eqref{Eq30}.\;
	
	Incorporate the new phase field $\phi$ into the elasticity tensor $\mathcal{F}^*(\phi)$.\;
	
	$n \leftarrow n+1$\;
}

\KwOut{Final topology represented by phase field function $\phi$}
\caption{Topology Optimization of a Structural Hemivariational Inequality by the Phase Field Method with Topological Derivative}
\label{Agl4}
\end{algorithm}

\section{Numerical experiments} \label{sec8} 
Up to now, we have not discussed spatial discretizations. It is worth noting that the discrete formulation of the original hemivariational inequality \eqref{hemi}
involves non-smooth and non-convex term $j_{\tau}$. We used the regularization method to approximate it by $j_{\tau,\varepsilon}$ in \eqref{smooth-potential}. Compared to the convex function of variational inequalities, the expression of this term is often more complex, which leads to strong nonlinearity after regularization and more complex computation as well.
Newton's method is used for regularized nonlinear state problems of hemivariational inequality. We discretize and solve the state problem  \eqref{State} or \eqref{State_weak1}, adjoint problem \eqref{Adjoint} or \eqref{Phase_Adjoint}, gradient flow \eqref{Smoothing}, and phase field equations by the \emph{finite element method}. 
Lagrange linear finite element is used after 2D triangulation and 3D tetrahedral mesh generation with Gmsh. 

We choose step size for stability and decrease of the objective:
$\tau_n = 5\times 10^{-3}{h}/\|\mathcal{V}\|_{L^\infty(D)}$, where $h$ is maximum mesh size. For topology optimization methods, the time step $\Delta t$ is chosen similarly by $\mathcal{V}=J'(\phi)$ and $\mathcal{V}=d_{T} J(\bm x)$ associated without and with topological derivative, respectively. Set $E=1.0$ and $\nu=0.3$. Set the coefficients $a=4\times 10^{-3},\ b=2\times 10^{-3} \text{ and } \alpha=100$. Numerical experiments are implemented using FreeFem++ \cite{HF12}. The applied force $\bm g_N=(0,-0.3)^{\rm T}$ and $\bm g_N=(0,0,-1)^{\rm T}$ for 2D and 3D, respectively. The Newton iteration stops when the norm for the relative difference of two consecutive iterates is less than $10^{-6}$. Set $\mathcal{T}=10$ and $\omega=0.01$. Assume that the body force $\bm{f} \equiv 0$. For topology optimization with a random initial design below, we generate uniformly distributed random numbers in the interval [0, 1] to form a random initial phase field function. 

\subsection{Numerical results by shape optimization Algorithm \ref{Agl1}}
We aim to minimize the compliance functional under volume constraint. using Algorithm \ref{Agl1}, where the compliance $J(\Omega) = \int_{\Gamma_N} \bm g_N \cdot \bm u.$ 
The quality of the mesh may deteriorate due to deformations. To maintain numerical stability and accuracy, the mesh is adjusted and optimized every five iterations to ensure a well-shaped discretization of the current domain.

The algorithm parameters are set as follows. The initial Lagrange multiplier is chosen as $\ell_0 = 0.01$, and the initial penalty factor is set to $\gamma_0 = 0.02$. The volume update factor is given by $\rho = 1.02$. Set the maximum number of iterations $N_m = 200$. The stopping criterion is determined by the tolerance $\text{Tol} = 10^{-3}$. The target volume is prescribed as $C = 0.95$ to guide the optimization process.

\begin{exam}\label{Ex1}
\textbf{Single Load.} 
Consider an initial design domain that is a square region $(0,1)^2$ with a circular hole of radius $0.2$ centered at $(0.5,0.5)$. A vertical force is applied at the midpoint of the top boundary of the domain. The left and right sides of the design domain are fully fixed, preventing any displacement. Boundary conditions are imposed as follows: a potential contact boundary condition is enforced on the bottom side of the domain, while a free boundary condition is applied to the inner surface of the circular hole. A schematic representation of the problem setup is provided in Fig. \ref{fig3} (a). The optimization process involves iterative shape evolution, where topological changes occur as the structure adapts to minimize the objective function. The intermediate and final shapes obtained during this evolution are illustrated in Figs. \ref{fig3} (b)-(d). The final deformed structure, which results from the applied loading conditions, is also depicted. The convergence behavior of the objective function throughout the optimization is presented in Fig. \ref{fig3} (e), demonstrating a consistent decay. Additionally, the volume constraint is well maintained after convergence, as shown in Fig. \ref{fig3} (f), ensuring that the optimization respects the prescribed material usage.
\end{exam}

\begin{figure}[htbp]
\begin{minipage}[h]{0.385\linewidth}
	\centering
	\includegraphics[height=4.05cm,width=4.1cm]{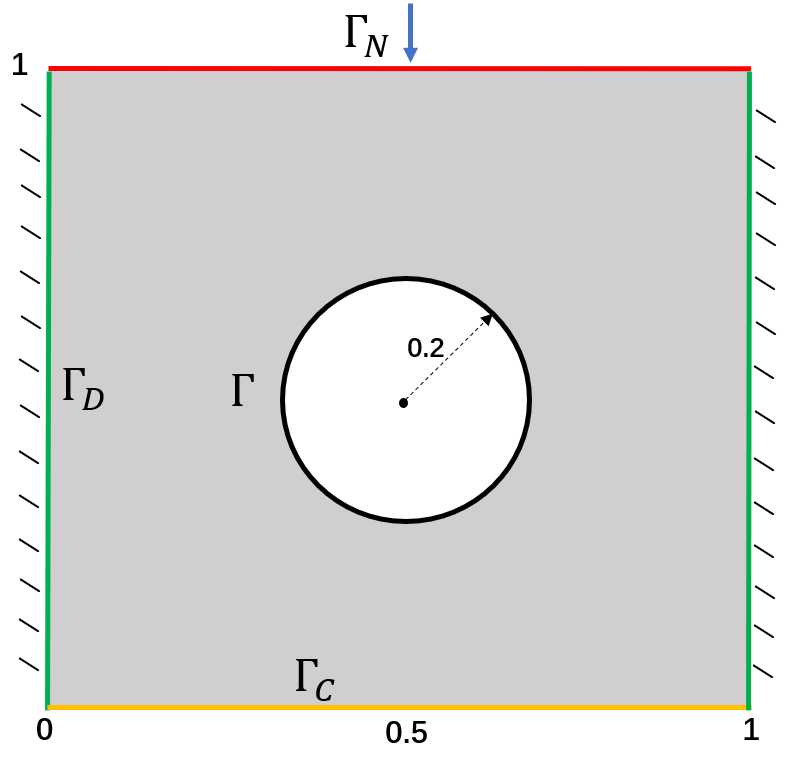} 
	\centerline{(a) Problem setting}
\end{minipage}
\hfill
\begin{minipage}{0.3\linewidth}
	\centering
	\includegraphics[height=3.6cm,width=3.6cm]{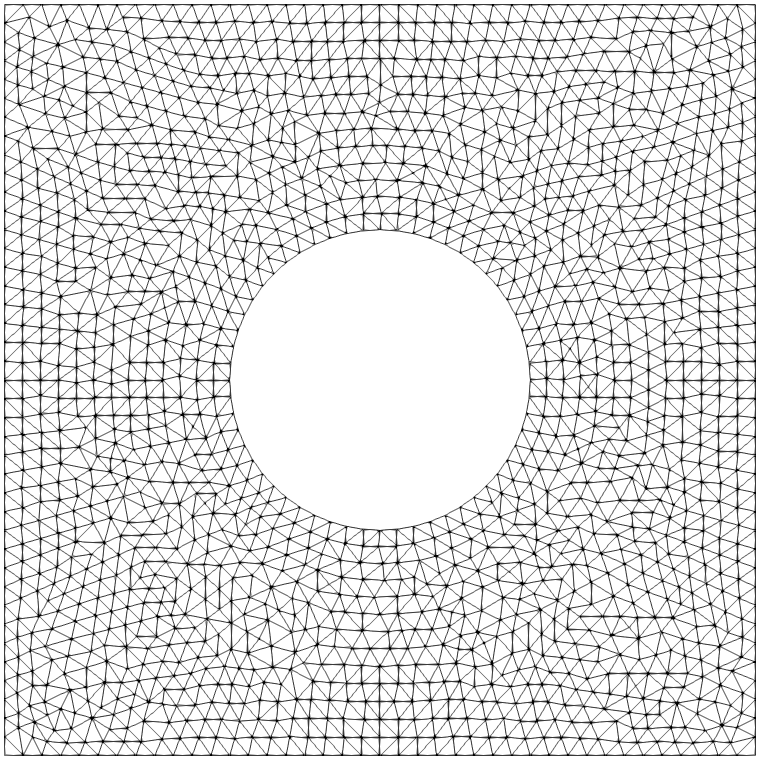} 
	\centerline{(b) Initial shape}
\end{minipage}
\hfill
\begin{minipage}{0.3\linewidth}
	\centering
	\includegraphics[height=3.6cm,width=3.6cm]{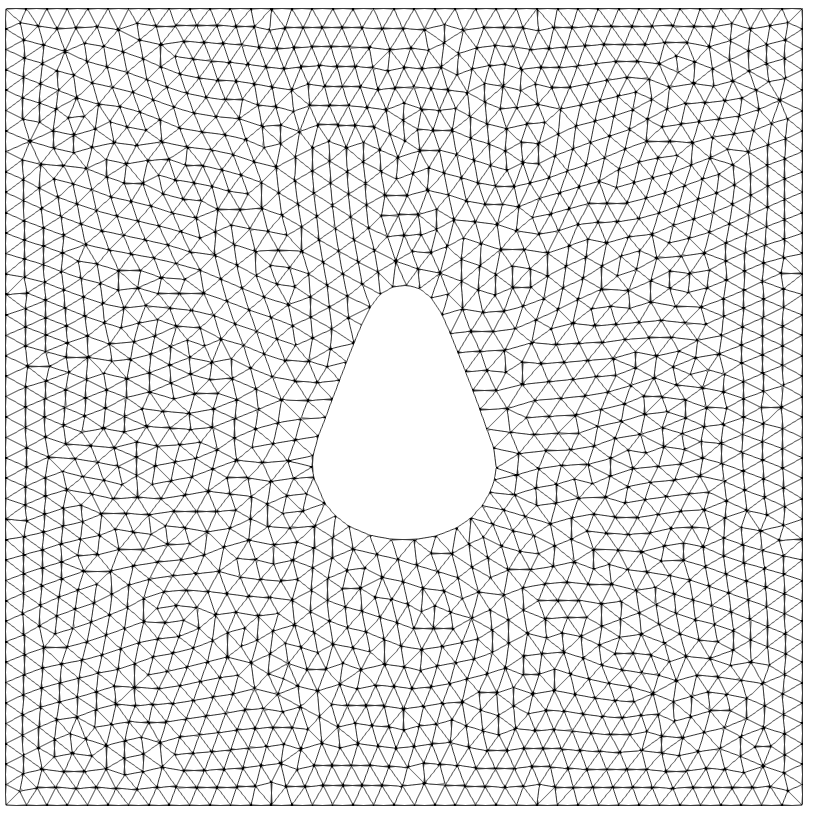} 
	\centerline{(c) Result shape}
\end{minipage}
\hfill
\begin{minipage}{0.31\linewidth}
	\centering
	\includegraphics[height=4.cm,width=4.cm]{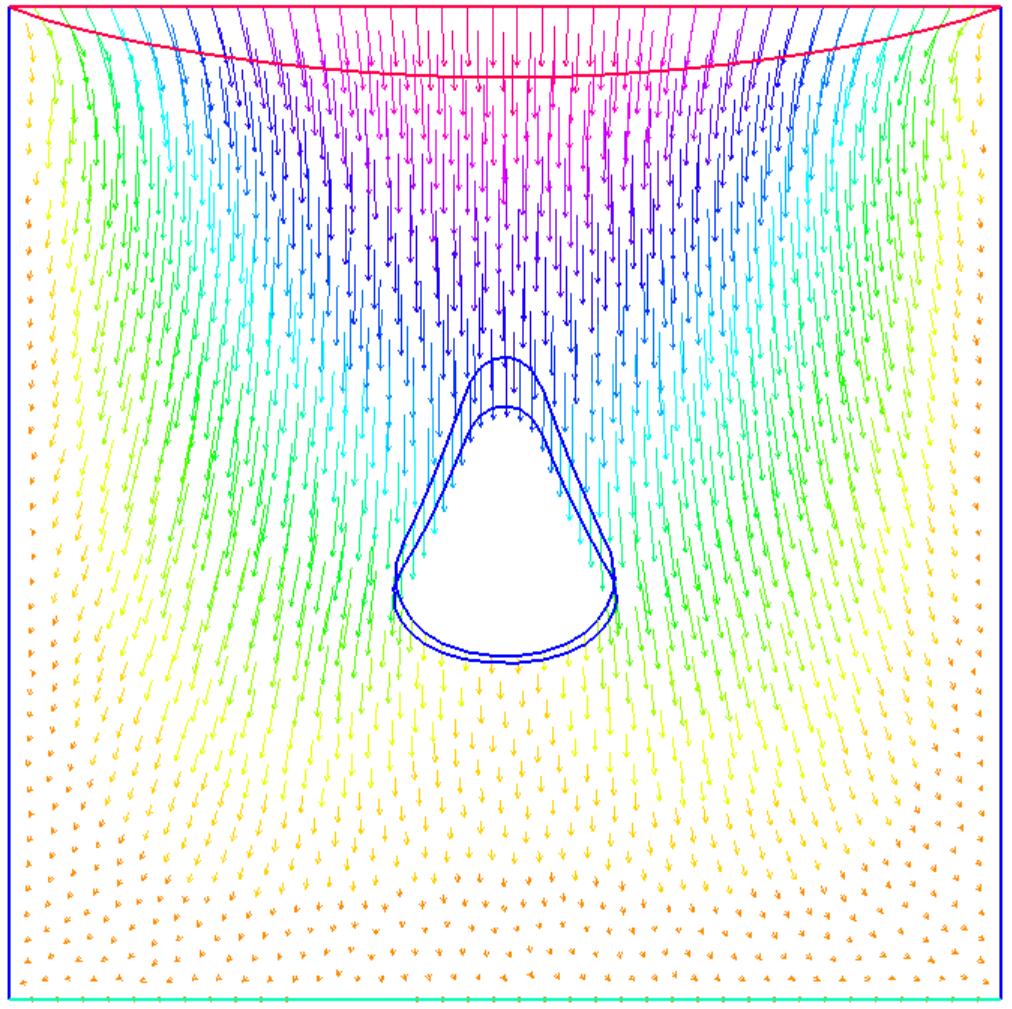} 
	\centerline{(d) Displacement of elastic structure}
\end{minipage}
\hfill
\begin{minipage}{0.31\linewidth}
	\centering
	\includegraphics[height=5cm,width=6cm]{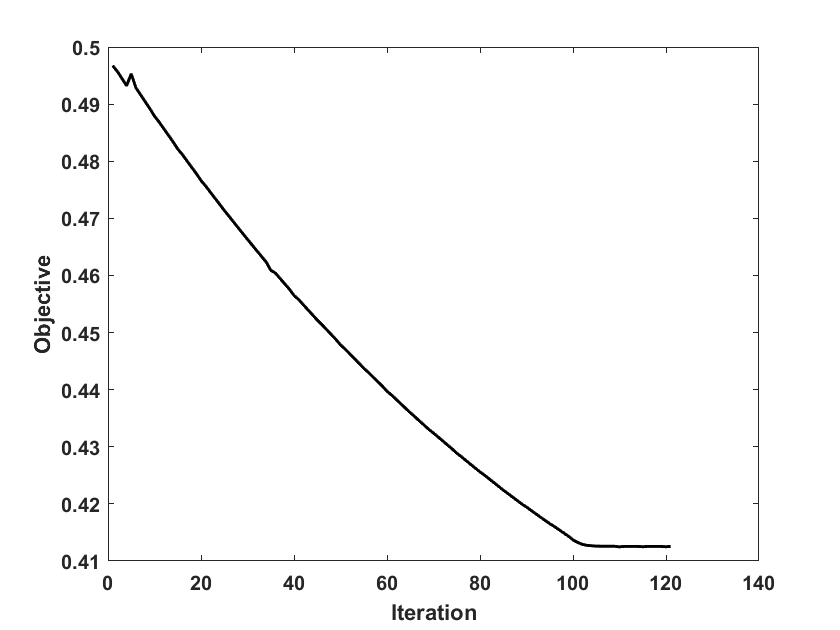} 
	\centerline{(e) Convergence of objective}
\end{minipage}
\hfill
\begin{minipage}{0.31\linewidth}
	\centering
	\includegraphics[height=5cm,width=6cm]{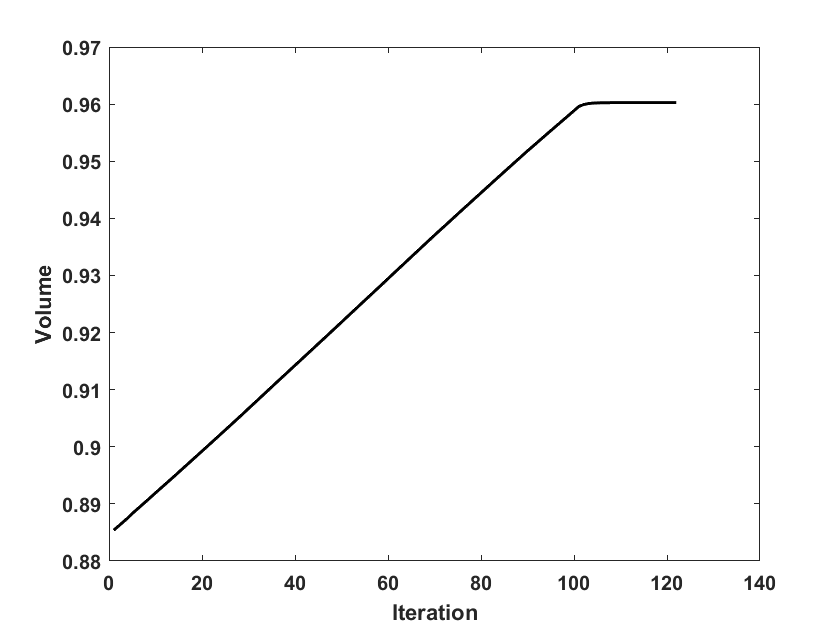} 
	\centerline{(f) Convergence of volume}
\end{minipage}
\caption{Example \ref{Ex1}: problem setting and numerical results.}\label{fig3}
\end{figure}

\begin{exam}\label{Ex2}
\textbf{Multiple Loads.} 
This problem follows the same initial setup as Example \ref{Ex1}, with the design domain being a square region $(0,1)^2$ containing a circular hole of radius $0.2$ centered at $(0.5,0.5)$. However, instead of a single vertical force, identical vertical forces are applied at three distinct locations along the upper boundary, as illustrated in Fig. \ref{fig4} (a). These additional load points introduce a more complex structural response compared to Example \ref{Ex1}. The optimization process starts from the same initial design as in Fig. \ref{fig3} (b). As the structure evolves to minimize the objective function while satisfying the volume constraint, notable differences in the final topology emerge. In particular, the inner boundary of the optimized structure appears rounder than that observed in Example \ref{Ex1}, as shown in Fig. \ref{fig4} (b). This indicates that the redistribution of forces influences the final shape of the optimized design.
The convergence behavior of the optimization is presented in Fig. \ref{fig4} (c) and Fig. \ref{fig4} (d), which depict the evolution of the objective function and volume fraction, respectively. These plots demonstrate that the optimization process achieves stable convergence while maintaining the prescribed volume constraint.
\end{exam}

\begin{figure}[htbp]
\begin{minipage}[h]{0.45\linewidth}
	\centering
	\includegraphics[height=4.05cm,width=4.1cm]{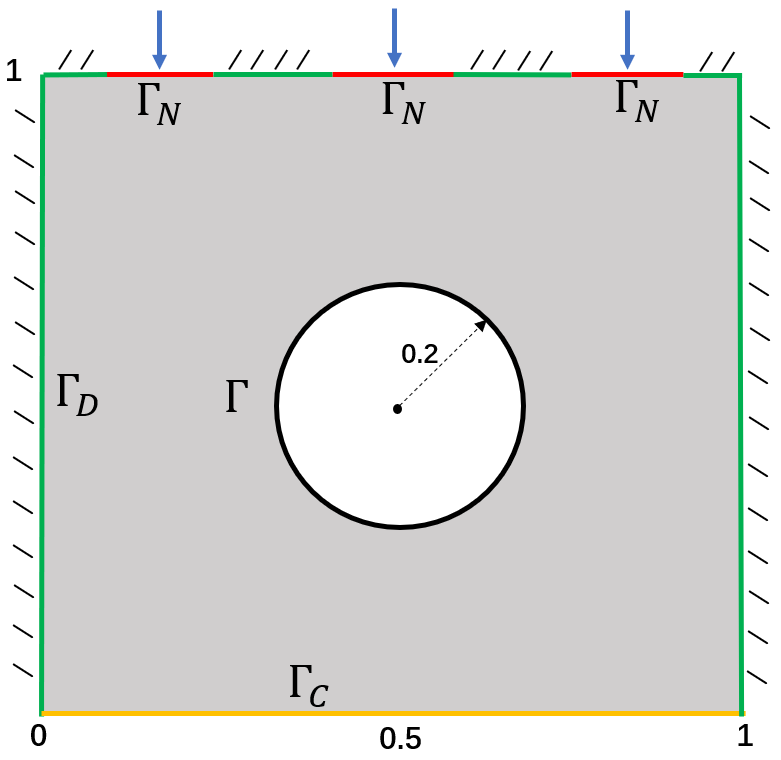} 
	\centerline{(a) Problem setting}
\end{minipage}
\hfill
\begin{minipage}[h]{0.45\linewidth}
	\centering
	\includegraphics[height=3.7cm,width=3.7cm]{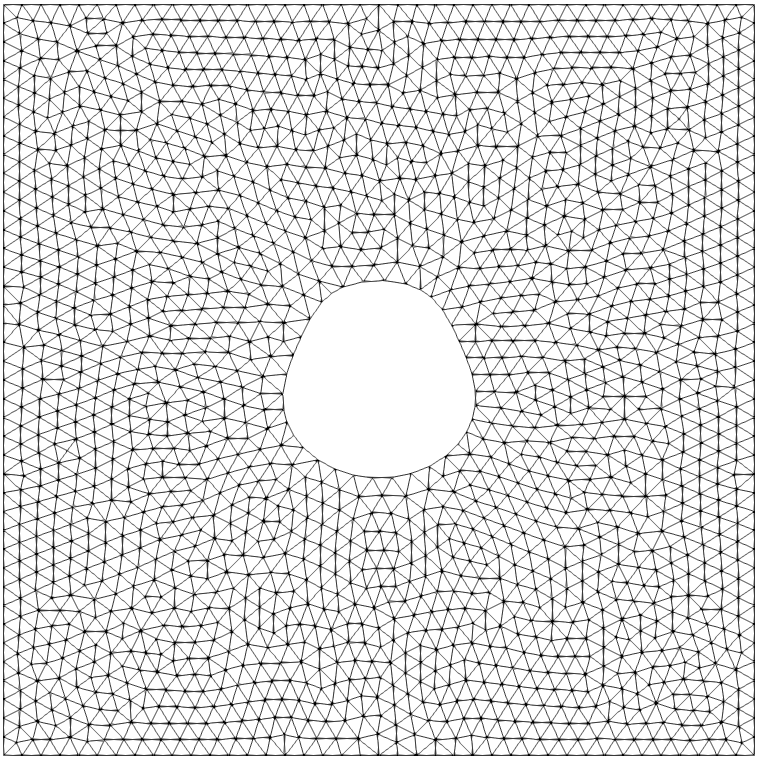} 
	\centerline{(b) Result shape}
\end{minipage}
\hfill
\begin{minipage}[h]{0.45\linewidth}
	\centering
	\includegraphics[height=5cm,width=7cm]{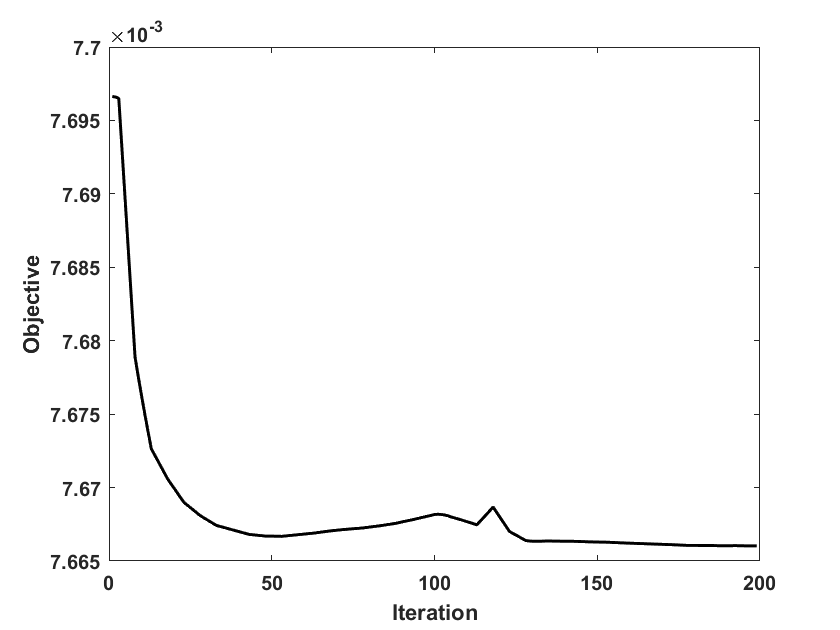} 
	\centerline{(c) Convergence history of objective}
\end{minipage}
\hfill
\begin{minipage}[h]{0.45\linewidth}
	\centering
	\includegraphics[height=5cm,width=7cm]{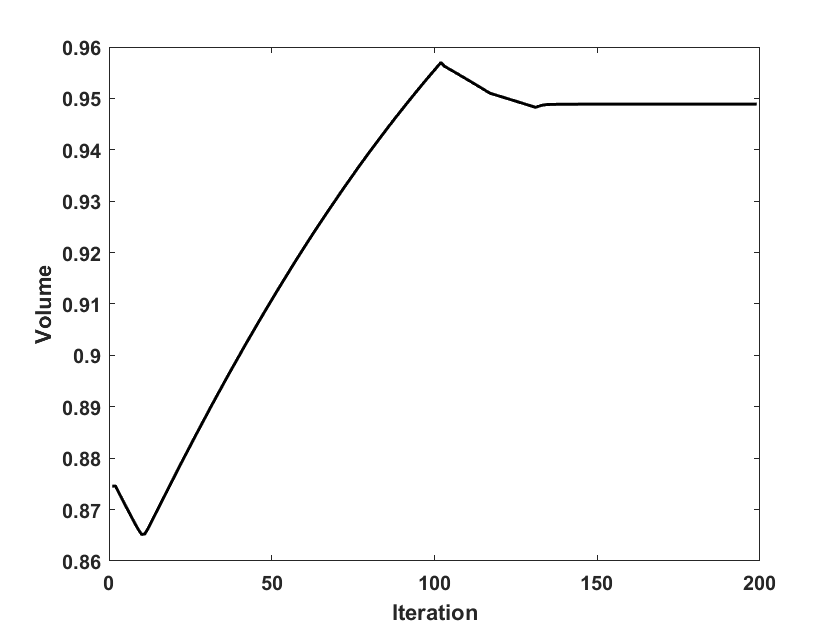} 
	\centerline{(d) Convergence history of volume}
\end{minipage}
\caption{Example \ref{Ex2}: problem setting and numerical results.}\label{fig4}
\end{figure}

\subsection{Numerical results of phase field methods without topological derivative}
We consider minimizing compliance $J(\Omega) = \int_{\partial_{D_N}} \bm g_N \cdot \bm u$ under volume constraints using Algorithm \ref{Agl2}. For Case A of Example \ref{Ex3}, we utilize Algorithm \ref{Agl3} to perform the numerical computations.
The parameters are set as follows. We take $k_{\min} = 10^{-5}$, $p = 3$, $\eta = 20$, and $\kappa_1 = 10^{-5}$. For 2D cases, we set $\varsigma = 10^{-3}$, $\ell_0 = 0$, $\gamma_0 = 20$, and $\rho = 1.05$. For 3D cases, we set $\varsigma = 10^{-2}$, $\ell_0 = 0$, $\gamma_0 = 0.1$, and $\rho = 1.02$. The evolution steps of the phase-field function are set to $\mathcal{T} = 10$. The target volume fraction is computed by
$V_f = {C}/ {{\rm Vol}(D)}$, where ${\rm Vol}(D)$ denotes the total volume of the design domain.

\begin{exam}\label{Ex3}
\textbf{Single Load: Cantilever.} 
\textbf{Case A}: 
Consider the design domain $D = (0,2) \times (0,1)$. A portion of the bottom boundary located at the center is in contact with a rigid foundation as illustrated in Fig. \ref{fig5} (a). 
Set $V_f = 0.32$, and $N_m = 300$. A fine triangular mesh resolution with mesh size $h = 0.01$ is employed.
The optimization process starts from a randomly generated initial design. The final optimized topologies obtained using two different phase field methods are shown in Fig. \ref{fig5} (c) for Algorithm \ref{Agl2} and in Fig. \ref{fig5} (d) for Algorithm \ref{Agl3}. The convergence history of the objective and volume fraction, presented in Figs. \ref{fig5} (e)-(f), demonstrates the effectiveness of both phase field methods. Furthermore, the results indicate that the volume constraint is well maintained throughout the optimization process.

\textbf{Case B}: 
{\rm Set $D = (0,2) \times (0,1)$, with the problem setting illustrated in Fig. \ref{fig66} (a). Choose $V_f = 0.35$. 
	Starting from an initial configuration depicted in Fig. \ref{fig66} (b), the optimized structure by Algorithm \ref{Agl3} is shown in Fig. \ref{fig66} (c). Additionally, we test the initial phase field function is a constant as displayed in Fig. \ref{fig66} (d), which yields a final design shown in Fig. \ref{fig66} (e). 
	The convergence histories of the objective and volume fraction are presented in Fig. \ref{fig66} (f) and Fig. \ref{fig66} (g), respectively. These results further validate the robustness of the optimization approach and its ability to achieve stable convergence while satisfying the prescribed volume constraint.}
\end{exam}

\begin{figure}[htbp]
\begin{minipage}[h]{0.48\linewidth}
	\centering
	\includegraphics[height=3.cm,width=6.cm]{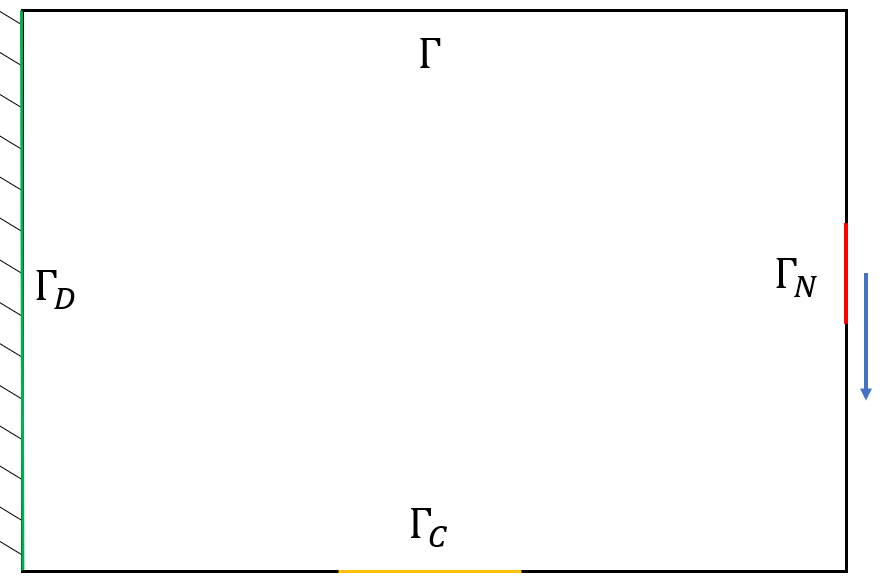} 
	\centerline{(a) Problem setting}
\end{minipage}
\hfill
\begin{minipage}[h]{0.48\linewidth}
	\centering
	\includegraphics[height=3.cm,width=5.8cm]{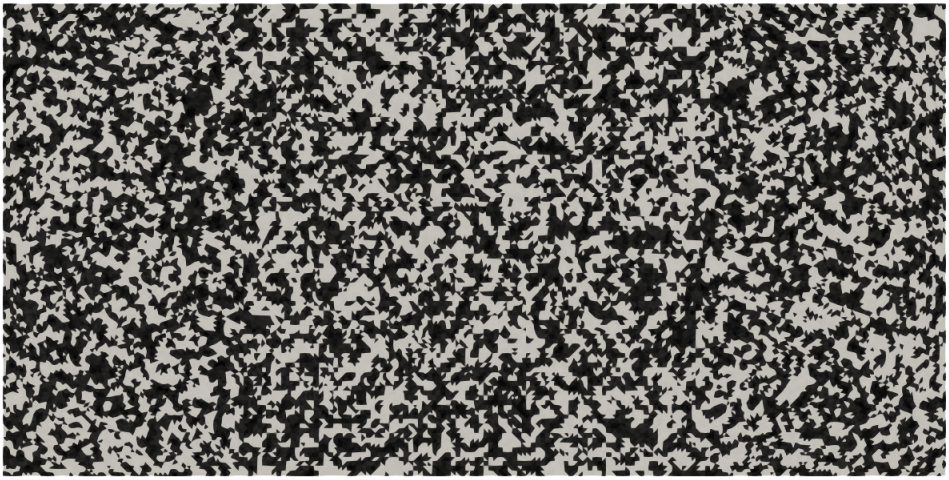} 
	\centerline{(b) Initial design}
\end{minipage}
\hfill
\begin{minipage}[h]{0.48\linewidth}
	\centering
	\includegraphics[height=3.cm,width=5.8cm]{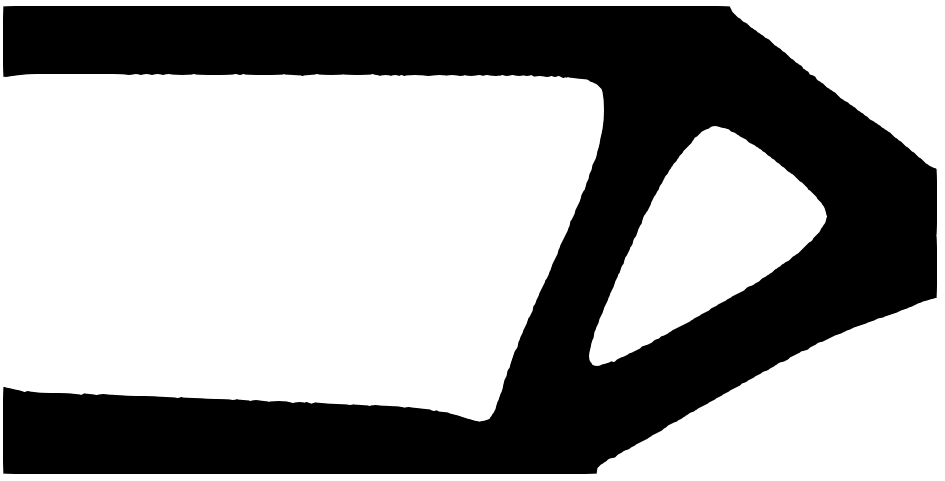} 
	\centerline{(c) Result by the phase field method (I)}
\end{minipage}
\hfill
\begin{minipage}[h]{0.48\linewidth}
	\centering
	\includegraphics[height=3.cm,width=5.8cm]{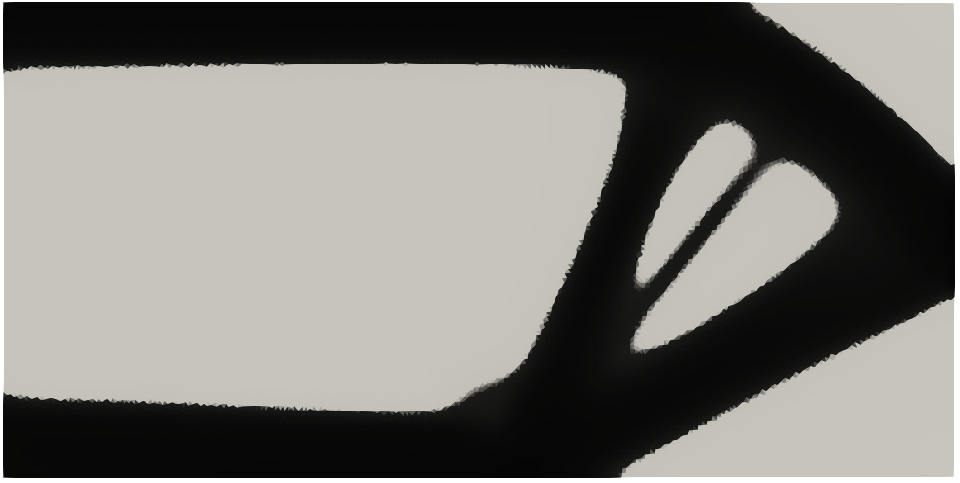} 
	\centerline{(d) {Result with the phase field method (II) }}
\end{minipage}
\hfill
\begin{minipage}[h]{0.48\linewidth}
	\centering
	\includegraphics[height=4.5cm,width=7.2cm]{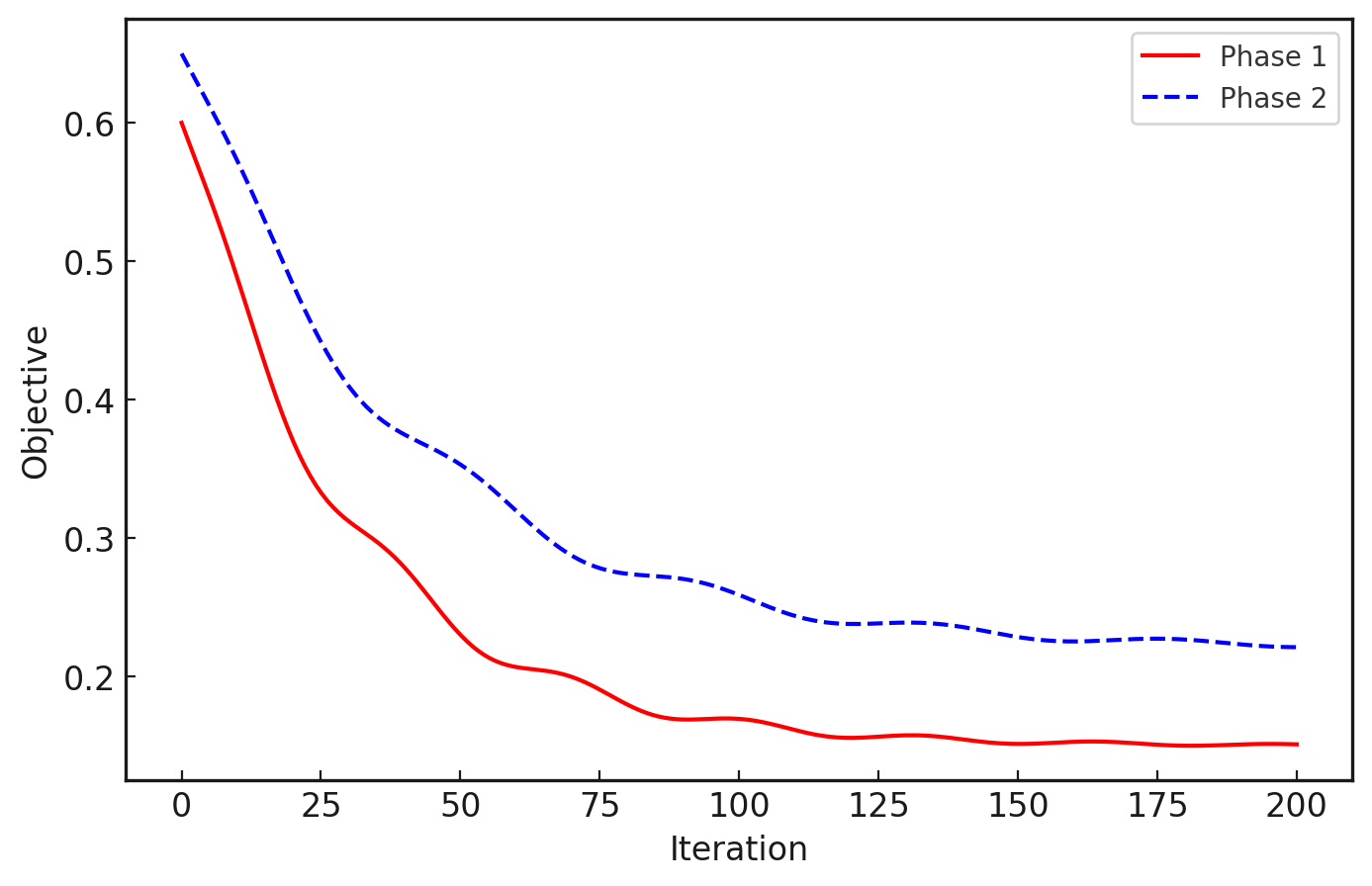} 
	\centerline{(e) {Objective}}
\end{minipage}
\hfill
\begin{minipage}[h]{0.48\linewidth}
	\centering
	\includegraphics[height=4.5cm,width=7.2cm]{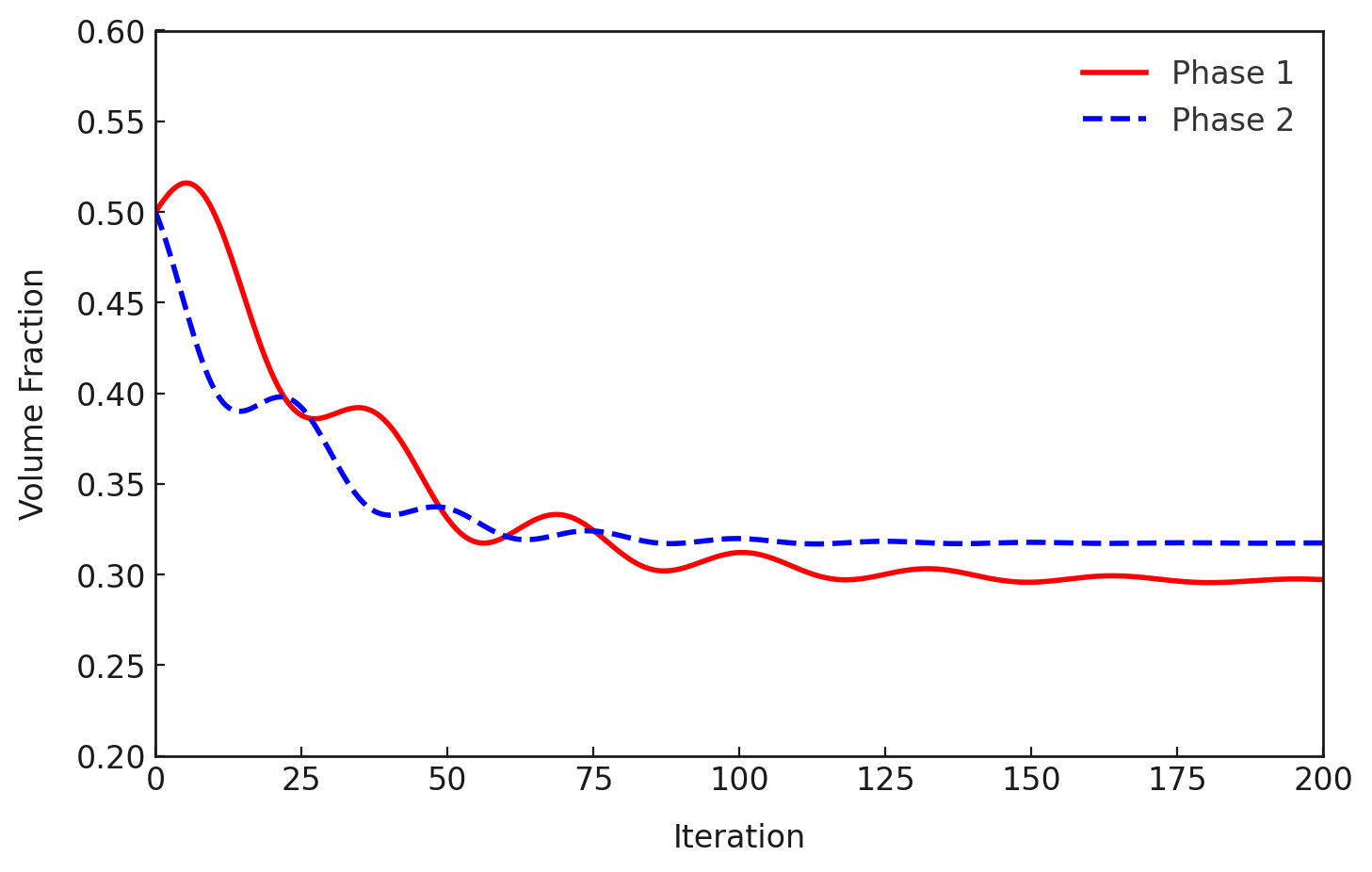} 
	\centerline{(f) Volume fraction}
\end{minipage}
\caption{Example \ref{Ex3}: Case A: problem setting and numerical results with different phase field methods.}\label{fig5}
\end{figure}

\begin{figure}[htbp]
\begin{minipage}[!h]{0.32\linewidth}
	\centering
	\includegraphics[height=2.7cm,width=5cm]{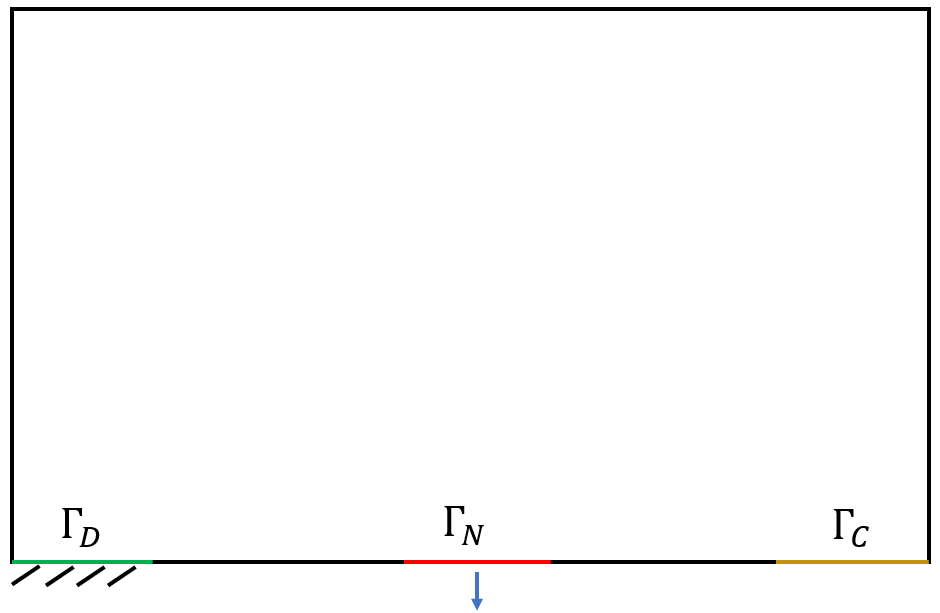} 
	\centerline{(a) Problem settings}
\end{minipage}
\hfill
\begin{minipage}[h]{0.3\linewidth}
	\centering
	\includegraphics[height=2.5cm,width=5cm]{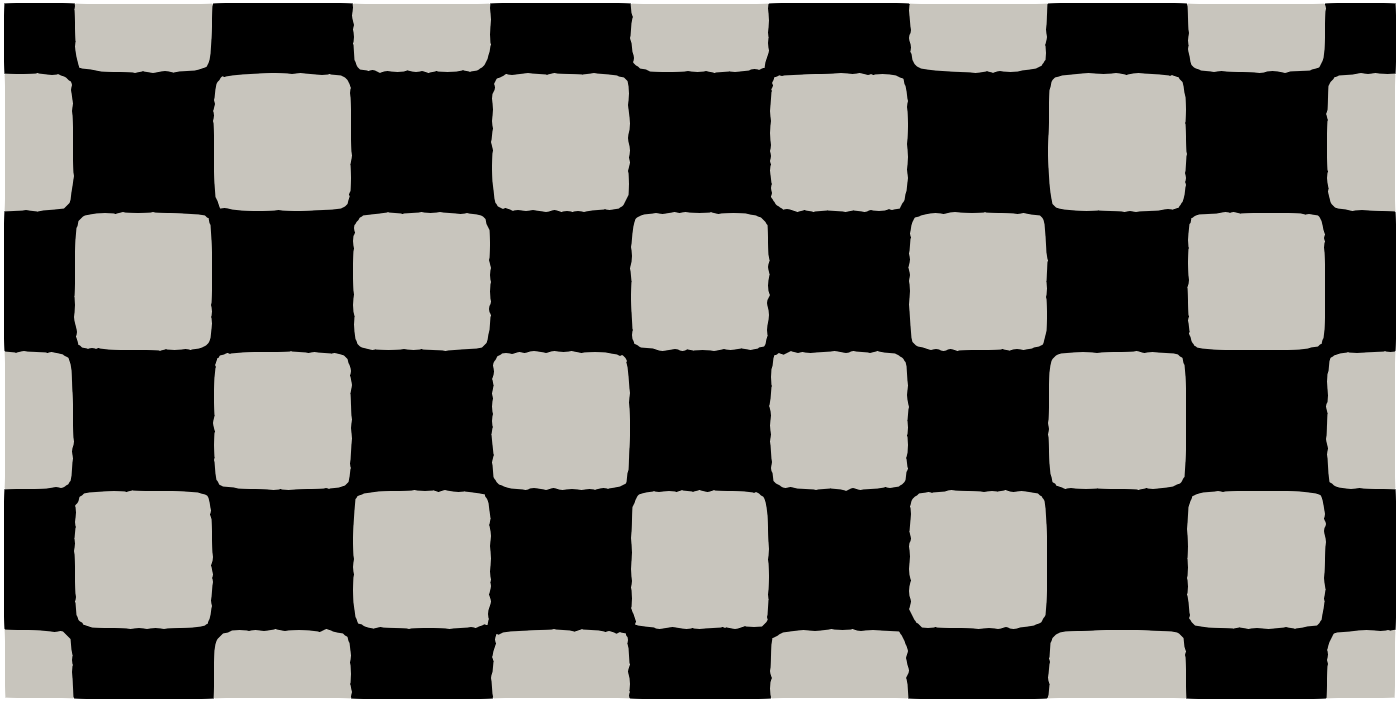} 
	\centerline{(b) Initial design 1}
\end{minipage}
\hfill
\begin{minipage}[h]{0.3\linewidth}
	\centering
	\includegraphics[height=2.5cm,width=5cm]{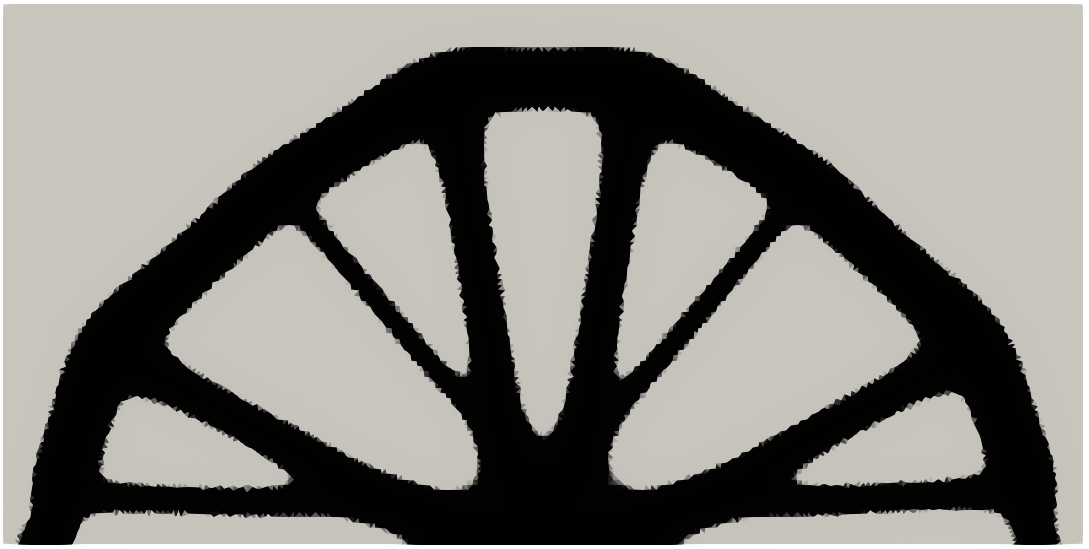} 
	\centerline{(c) Result 1}
\end{minipage}
\hfill
\begin{minipage}[h]{0.45\linewidth}
	\centering
	\includegraphics[height=2.5cm,width=5cm]{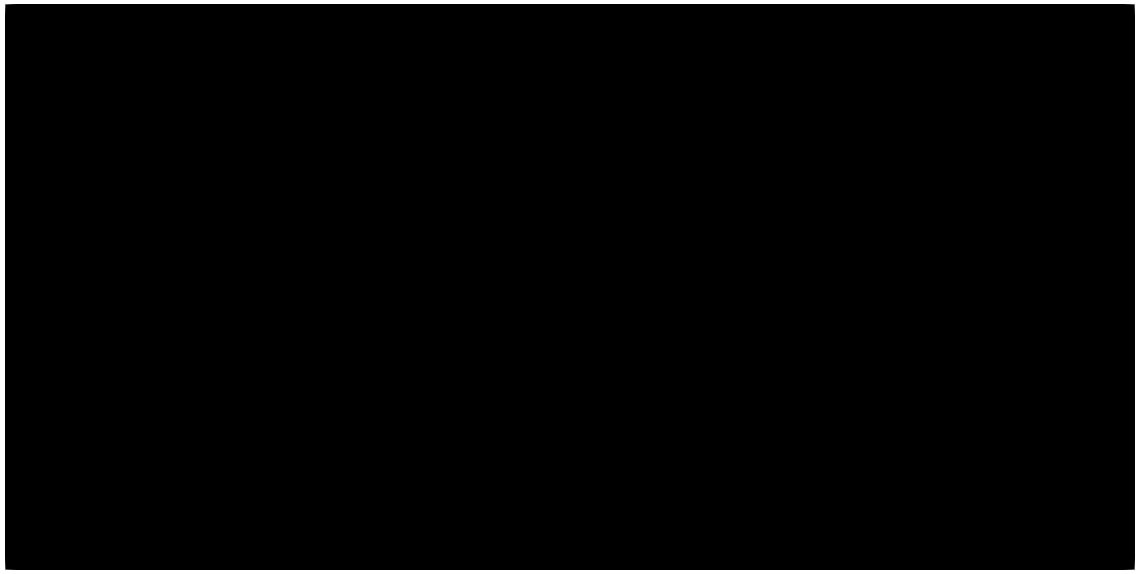} 
	\centerline{(d) Initial design 2}
\end{minipage}
\hfill
\begin{minipage}[h]{0.45\linewidth}
	\centering
	\includegraphics[height=2.5cm,width=5cm]{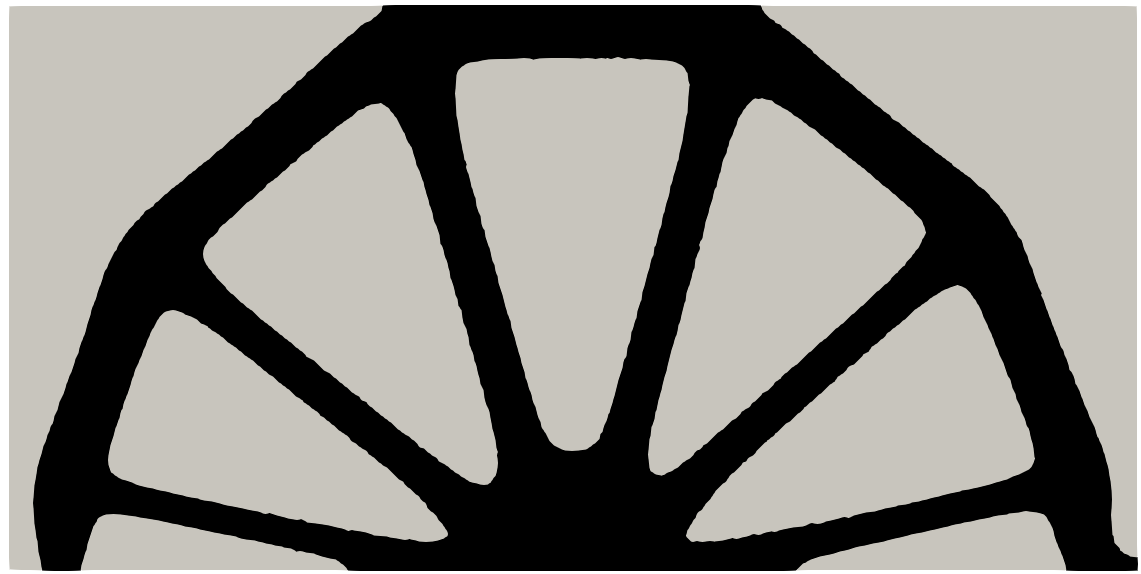} 
	\centerline{(e) Result 2}
\end{minipage}
\hfill
\begin{minipage}[h]{0.48\linewidth}
	\centering
	\includegraphics[height=4cm,width=7.5cm]{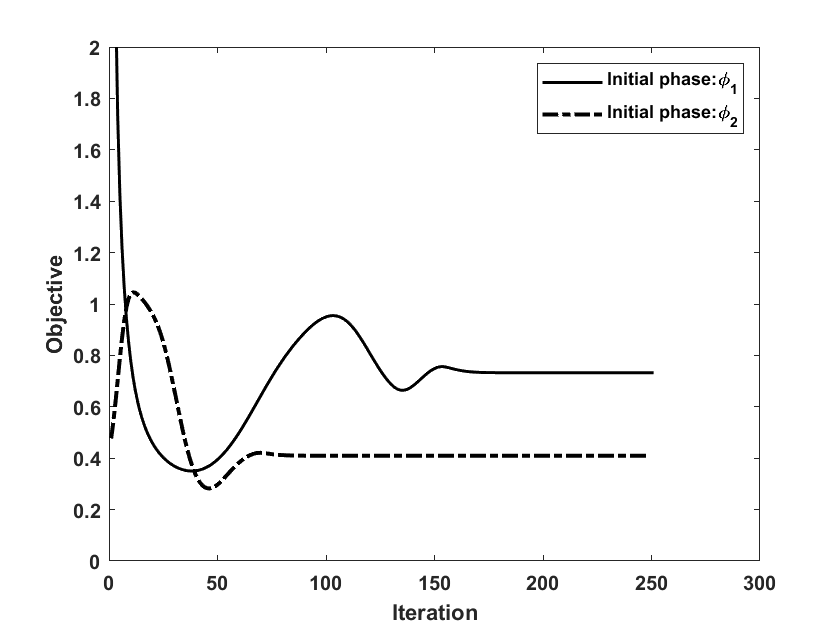} 
	\centerline{(f) Objective}
\end{minipage}
\hfill
\begin{minipage}[h]{0.48\linewidth}
	\centering
	\includegraphics[height=4cm,width=7.5cm]{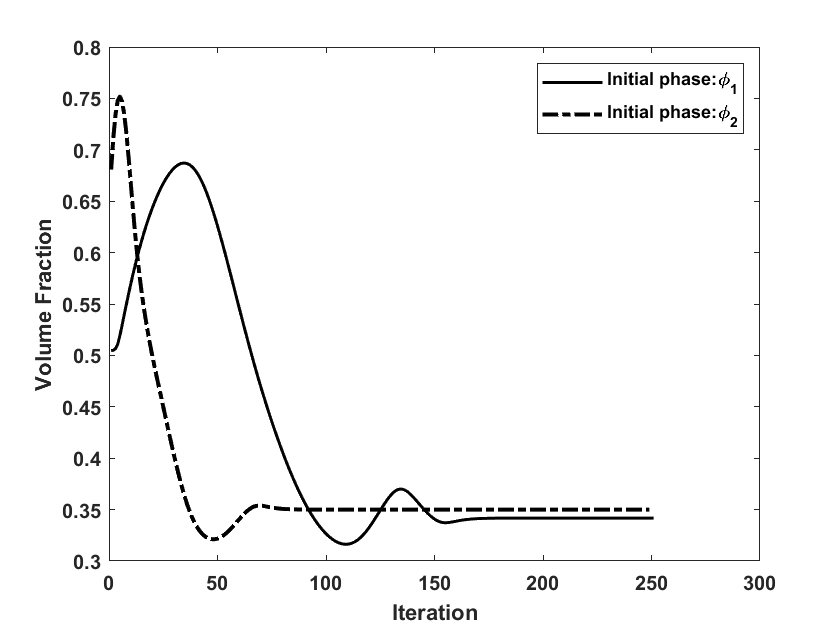} 
	\centerline{(g) Volume fraction}
\end{minipage}
\caption{Example \ref{Ex3}: Case B: problem setting and numerical results.}\label{fig66}
\end{figure}

\begin{exam}\label{Ex4}
\textbf{Multiple loads.} {\rm Consider the same domain as Example 3 (using Algorithm \ref{Agl2}). The left and right boundaries are fixed. The bottom boundary is in contact with a rigid foundation. Apply the same vertical force $\bm g_N$ to the three small parts on the left, right, and center of the top boundary (see Fig. \ref{fig7} (a) for illustration). Set $N_m=250$ and $V_f=0.45$. The optimized design and convergence history of objective and volume fraction are shown in Fig. \ref{fig7}.}
\end{exam}

\begin{figure}[htbp]
\begin{minipage}[h]{0.48\linewidth}
	\centering
	\includegraphics[height=3.6cm,width=6.2cm]{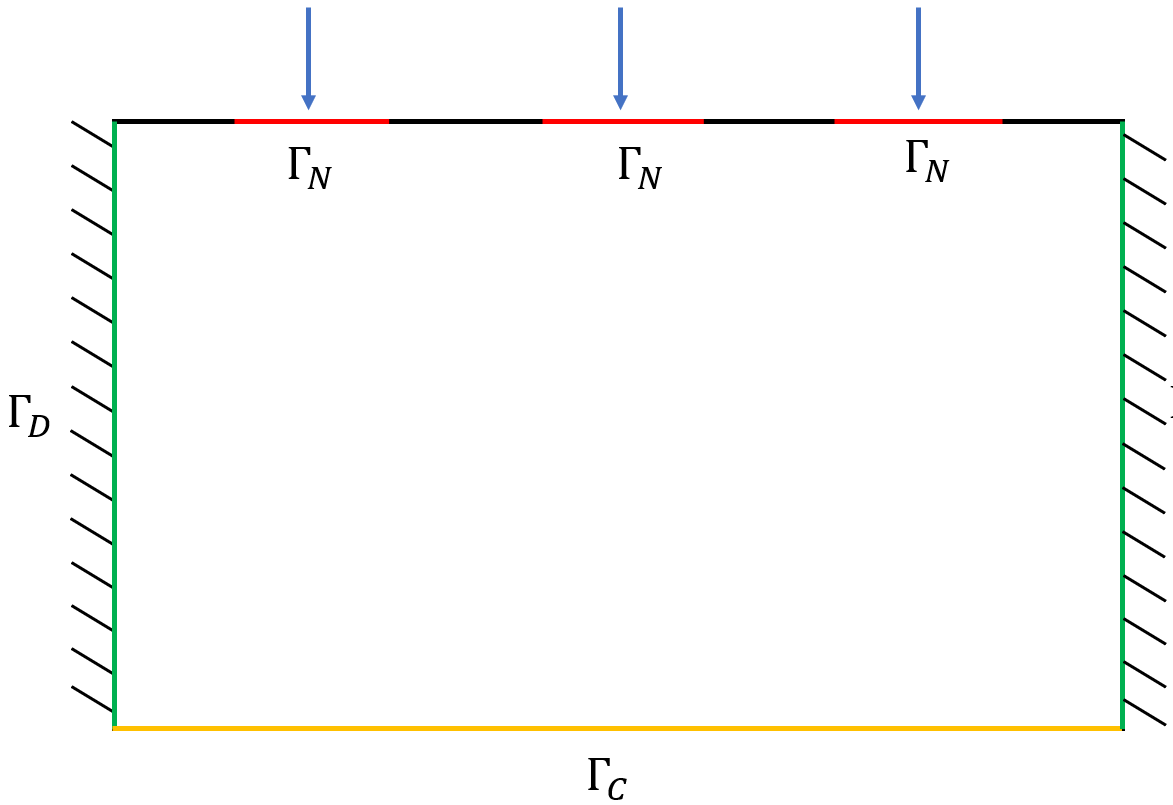} 
	\centerline{(a) Problem settings}
\end{minipage}
\hfill
\begin{minipage}[h]{0.48\linewidth}
	\centering
	\includegraphics[height=2.85cm,width=5.7cm]{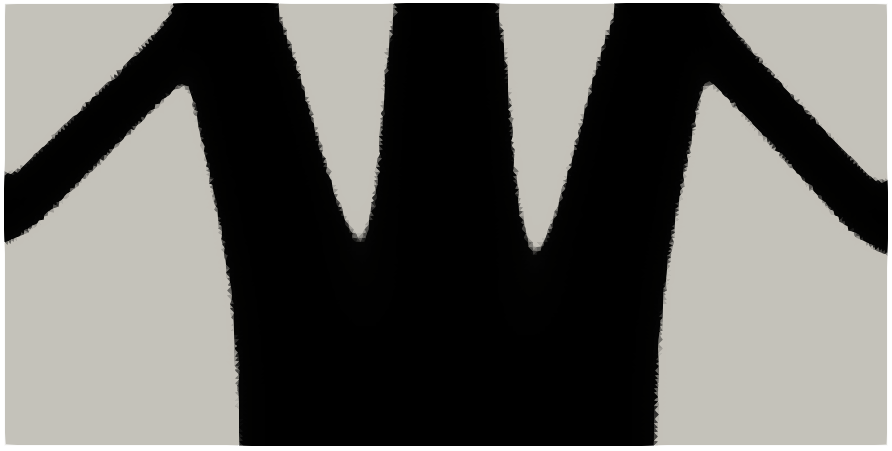} 
	\centerline{(b) Optimized shape}
\end{minipage}
\hfill
\begin{minipage}[h]{0.48\linewidth}
	\centering
	\includegraphics[height=4.5cm,width=7.2cm]{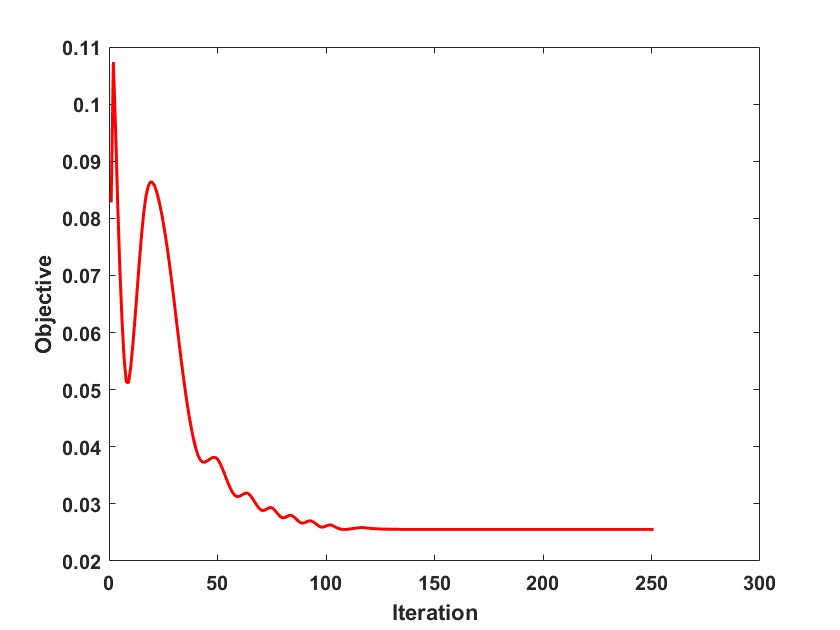} 
	\centerline{(c) Objective}
\end{minipage}
\hfill
\begin{minipage}[h]{0.48\linewidth}
	\centering
	\includegraphics[height=4.5cm,width=7.2cm]{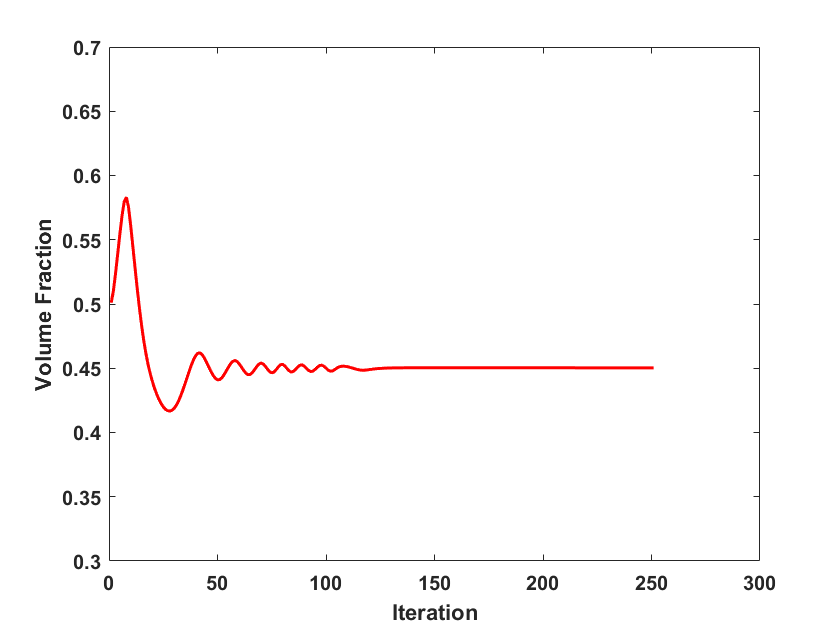} 
	\centerline{(d) Volume fraction}
\end{minipage}
\caption{Example \ref{Ex4}: problem setting and numerical results.}\label{fig7}
\end{figure}

\begin{exam}\label{Ex5}
\textbf{L-shaped cantilever in 2D.} {\rm Set $D$ to be L-shaped domain $(0,2)^2 \setminus ( [1,2)\times (0,1])$ (see Fig. \ref{fig8} (a)). 
	Set $N_m=200$ and $V_f=0.4$. Triangulation is shown in Fig. \ref{fig8} (b). A random initial phase field function with uniform distribution in $[0,1]$ shown in Fig. \ref{fig8} (c) evolves to the final optimized result (using Algorithm 2) in Fig. \ref{fig8} (e). The other initial phase field function taking a constant value $0.5$ converges to a design Fig. \ref{fig8} (f). The convergence histories of objective and volume fraction are shown in Fig. \ref{fig8} (e) and Fig. \ref{fig8} (f), respectively. The comparisons show that two cases coincide with each other.}
\end{exam}

\begin{figure}[htbp]
\begin{minipage}[!h]{0.32\linewidth}
	\centering
	\includegraphics[height=4.1cm,width=4.cm]{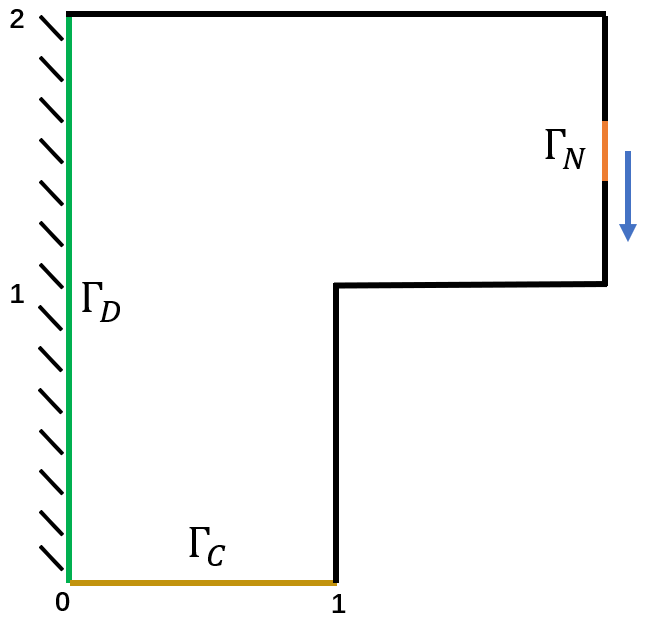} 
	\centerline{(a) Problem setting}
\end{minipage}
\hfill
\begin{minipage}[h]{0.32\linewidth}
	\centering
	\includegraphics[height=3.8cm,width=3.8cm]{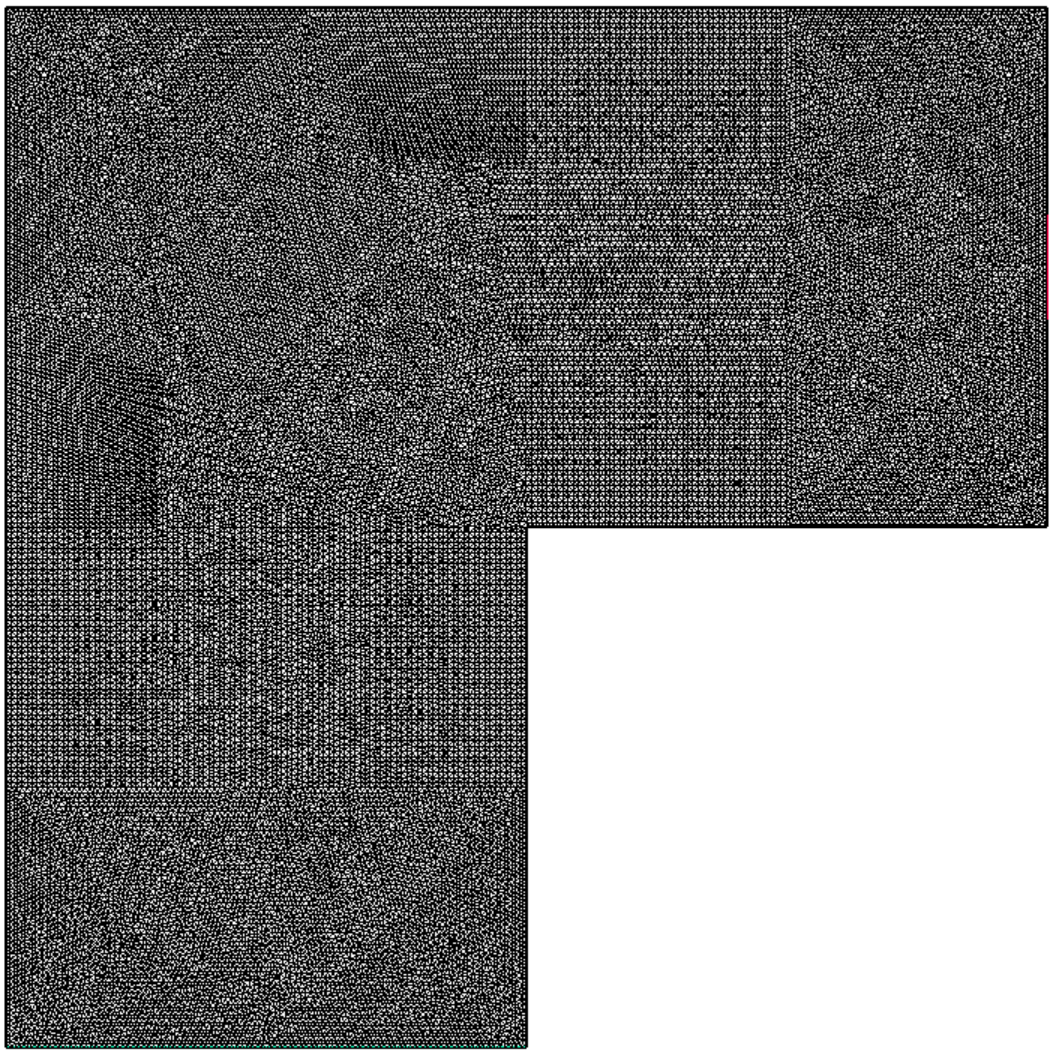} 
	\centerline{(b) Mesh with $h=0.01$ }
\end{minipage}
\hfill
\begin{minipage}[h]{0.32\linewidth}
	\centering
	\includegraphics[height=3.8cm,width=3.8cm]{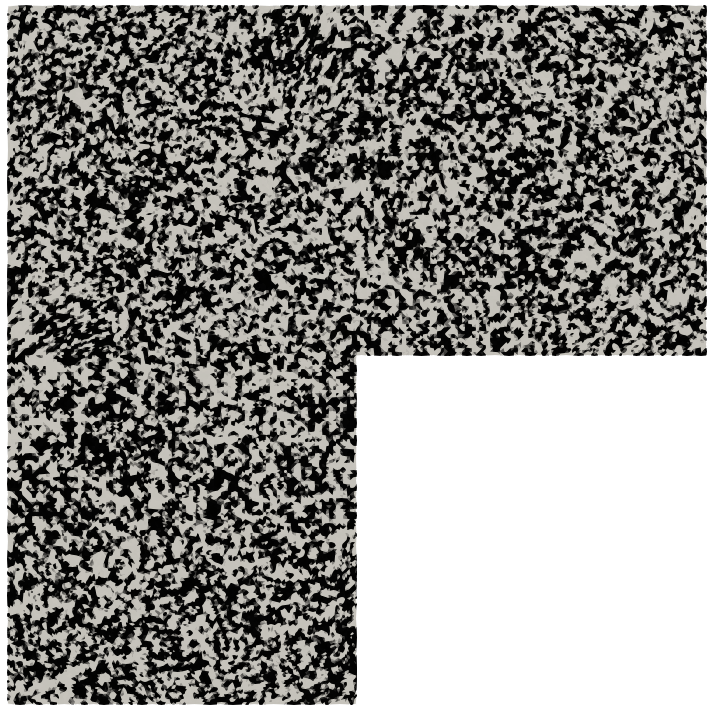} 
	\centerline{(c) Initial design}
\end{minipage}
\hfill
\begin{minipage}[h]{0.32\linewidth}
	\centering
	\includegraphics[height=3.8cm,width=3.8cm]{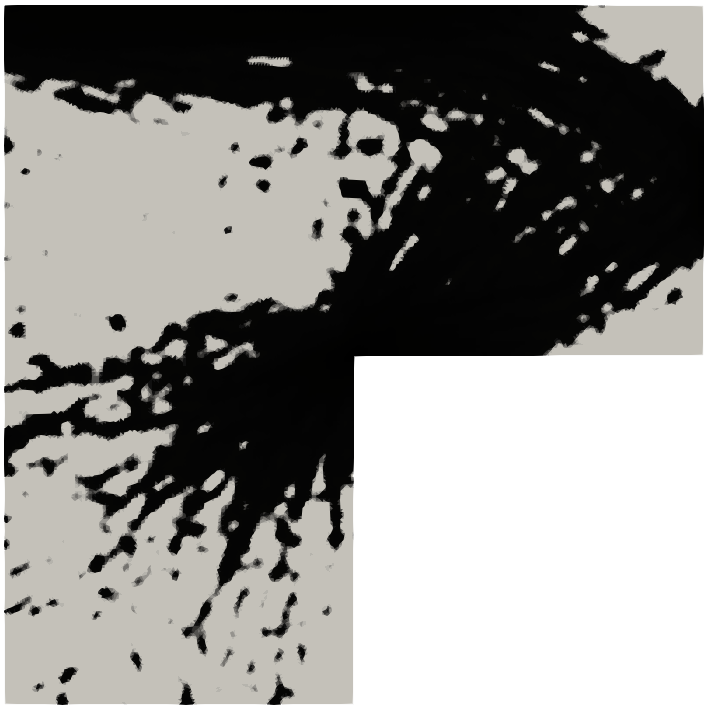} 
	\centerline{(d) Iteration=10}
\end{minipage}
\hfill
\begin{minipage}[h]{0.32\linewidth}
	\centering
	\includegraphics[height=3.8cm,width=3.8cm]{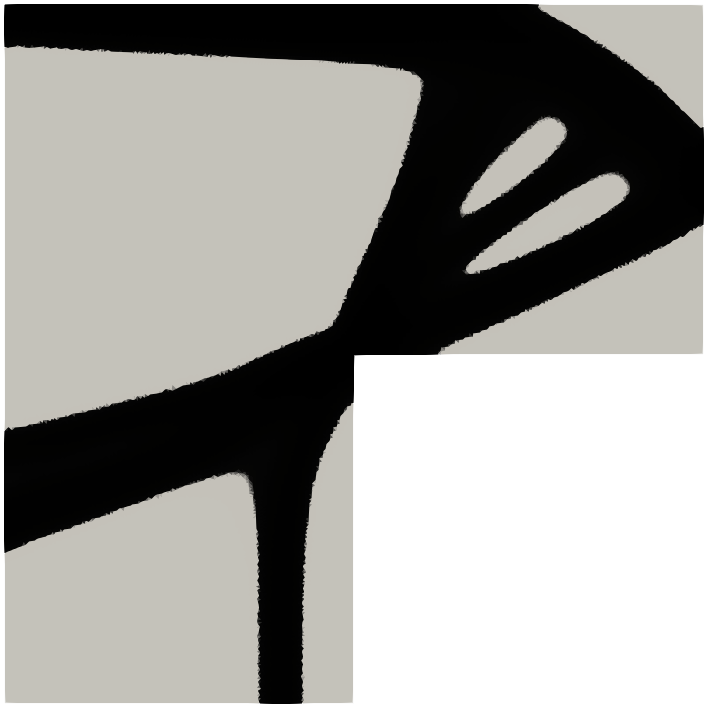} 
	\centerline{(e) Final design}
\end{minipage}
\hfill
\begin{minipage}[h]{0.32\linewidth}
	\centering
	\includegraphics[height=3.8cm,width=3.8cm]{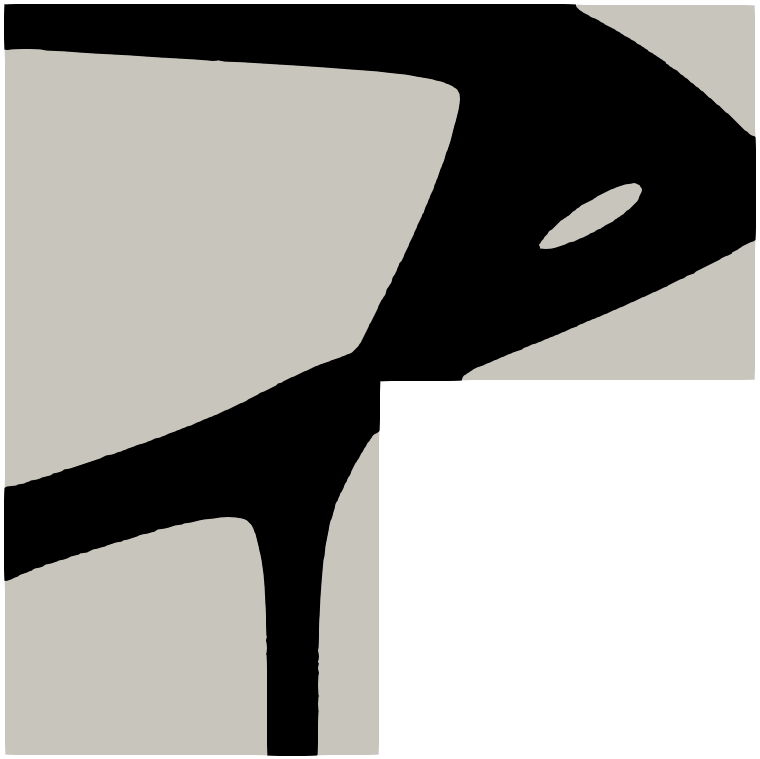} 
	\centerline{(f) Final design from initial $\phi=0.5$}
\end{minipage}
\hfill
\begin{minipage}[h]{0.44\linewidth}
	\centering
	\includegraphics[height=5cm,width=7cm]{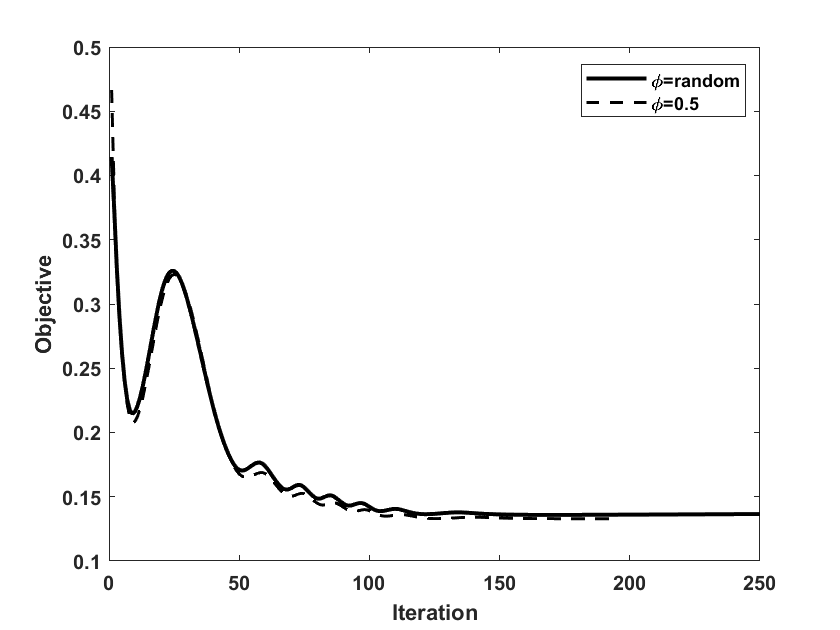} 
	\centerline{(g) Objective}
\end{minipage}
\hfill
\begin{minipage}[h]{0.45\linewidth}
	\centering
	\includegraphics[height=5cm,width=7cm]{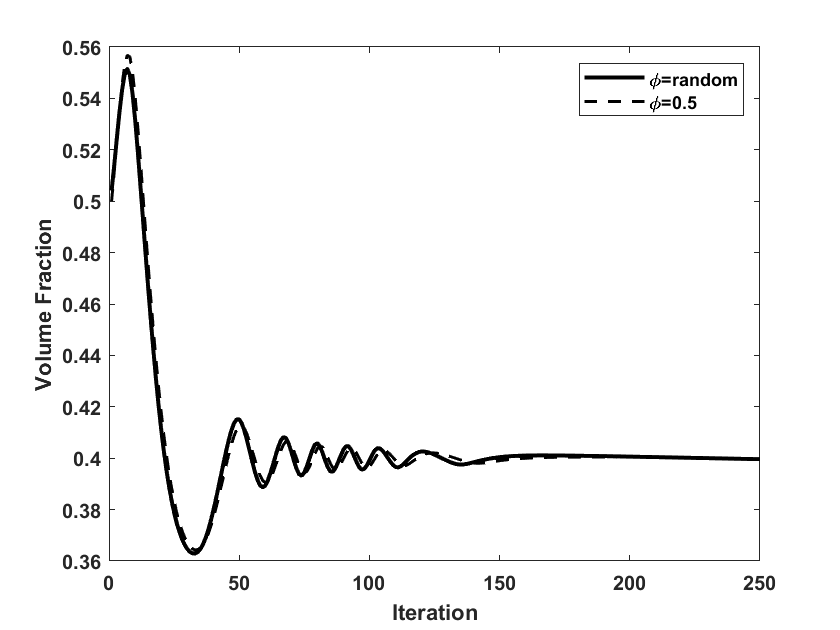} 
	\centerline{(h) Volume fraction}
\end{minipage}
\caption{Example \ref{Ex5}: problem setting and numerical results.}\label{fig8}
\end{figure}

\subsection{Numerical results with Algorithm \ref{Agl4}}
We consider to use Algorithm \ref{Agl4} with topological derivatives to solve the problem \eqref{Obj} of minimizing an energy functional, which represents the total potential energy of the system and is influenced by material distribution and external forces. 
\begin{exam}\label{Ex6}
\textbf{Cube Cantilever in 3D.} 
{\rm Set $D = (0,2) \times (0,1) \times (0,1)$, which extends the 2D cantilever problem in Fig. \ref{fig5}(a) to a three-dimensional setting as illustrated in Fig. \ref{fig9}(a). $N_m=300$ is the maximum number of iterations. $13911$ tetrahedrons are used for mesh elements. Set $V_f = 0.2$. 
	The optimization process begins with a randomly initialized design and undergoes iterative refinement towards an optimized topology shown in Fig. \ref{fig9}. 
}
\end{exam}

\begin{figure}[htbp]
\begin{minipage}[h]{0.3\linewidth}
	\centering
	\includegraphics[height=2.7cm,width=4.9cm]{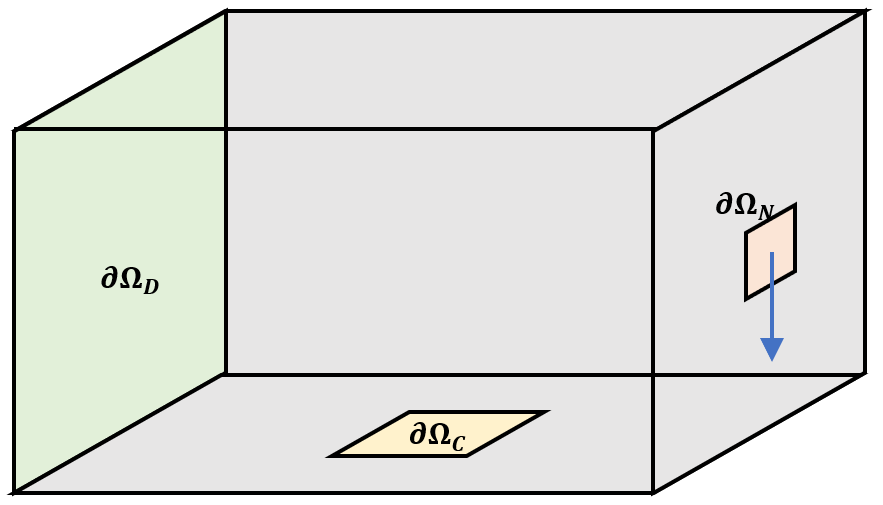} 
	\centerline{(a) Problem setting}
\end{minipage}
\hfill
\begin{minipage}[h]{0.3\linewidth}
	\centering
	\includegraphics[height=3.15cm,width=4.9cm]{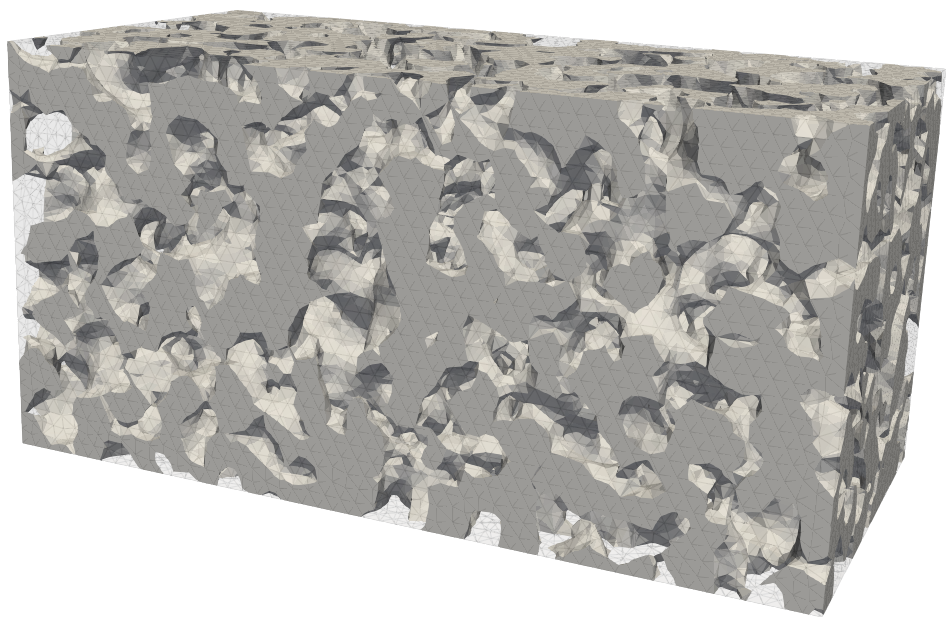} 
	\centerline{(b) Initial design}
\end{minipage}
\hfill
\begin{minipage}[h]{0.3\linewidth}
	\centering
	\includegraphics[height=3.15cm,width=4.9cm]{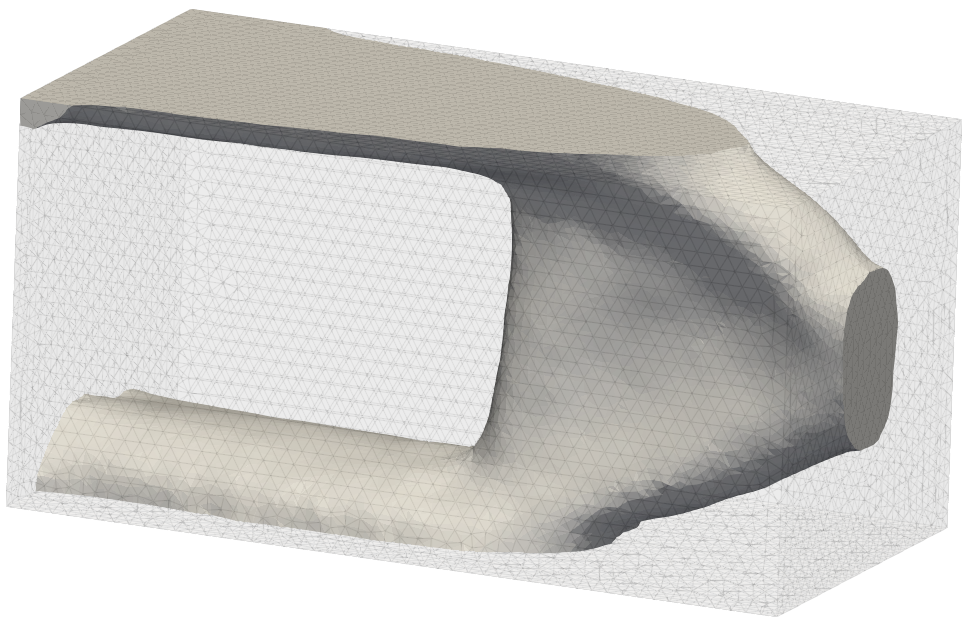} 
	\centerline{(c) Final design}
\end{minipage} 
\caption{Example \ref{Ex6}: problem setting and numerical results.}\label{fig9}
\end{figure}

\begin{exam}\label{Ex7}
\textbf{$L$-Shaped Cantilever in 3D.} 
{\rm We consider to generalizes the 2D case in Example \ref{Ex5} to a 3D L-shaped cantilever, where $D= (0,2) \times (0,1) \times (0,2) \setminus ([1,2] \times [0,1] \times [0,1])$. The domain boundary consists of four types of regions, as illustrated in Fig. \ref{fig10}(a). The Dirichlet boundary, denoted by $\partial \Omega_D$, represents the region where displacements are fully constrained. The inhomogeneous Neumann boundary, denoted by $\partial \Omega_N$, is the region where external forces are applied. The contact boundary $\partial \Omega_C$, represents the portion of the structure that is in contact with a rigid base. The remaining boundaries are free to optimize. Set $V_f = 0.35$ and $N_m=200$, and $20867$ tetrahedron elements. The final optimized structure along with numerical results presented in Fig. \ref{fig10} demonstrate the convergence behavior.}
\end{exam}

\begin{figure}[htbp]
\begin{minipage}[!h]{0.32\linewidth}
	\centering
	\includegraphics[height=4.cm,width=4.cm]{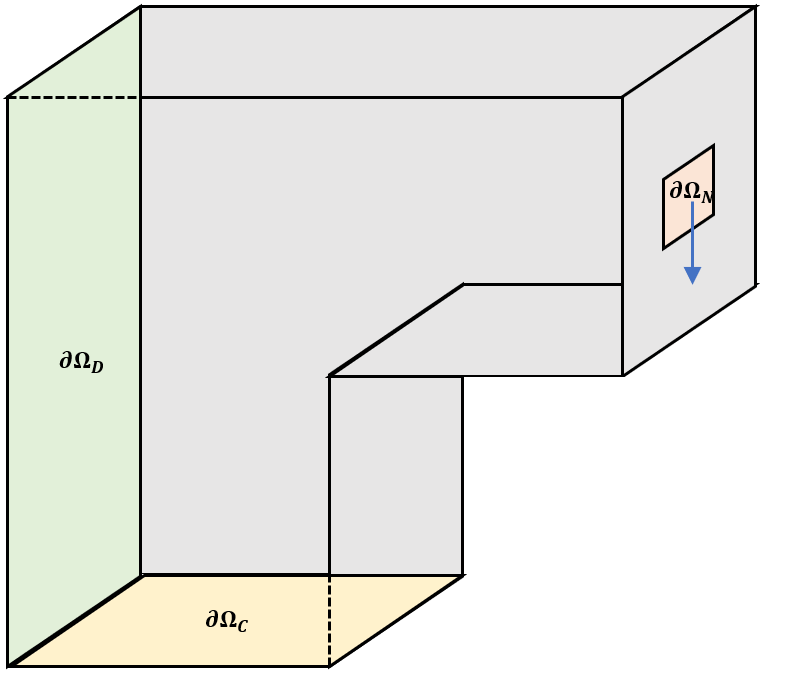} 
	\centerline{(a) Problem setting}
\end{minipage}
\hfill
\begin{minipage}[h]{0.32\linewidth}
	\centering
	\includegraphics[height=4.cm,width=4.cm]{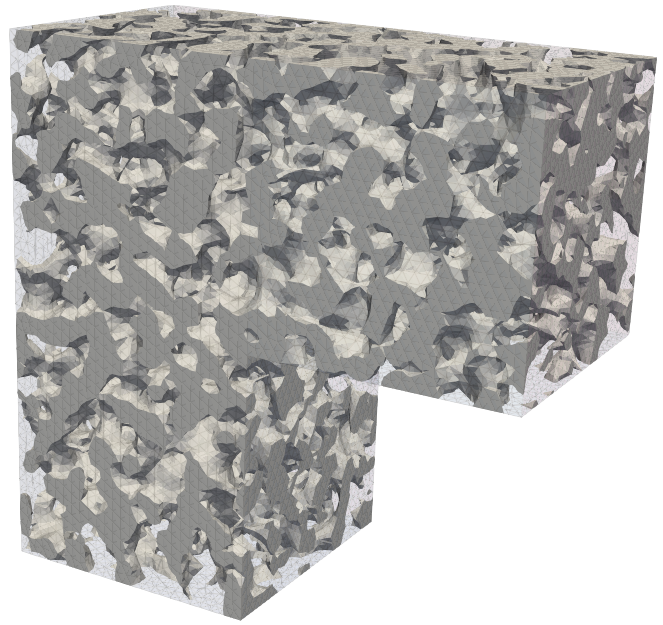} 
	\centerline{(b) Initial design}
\end{minipage}
\hfill
\begin{minipage}[h]{0.32\linewidth}
	\centering
	\includegraphics[height=4.cm,width=4.cm]{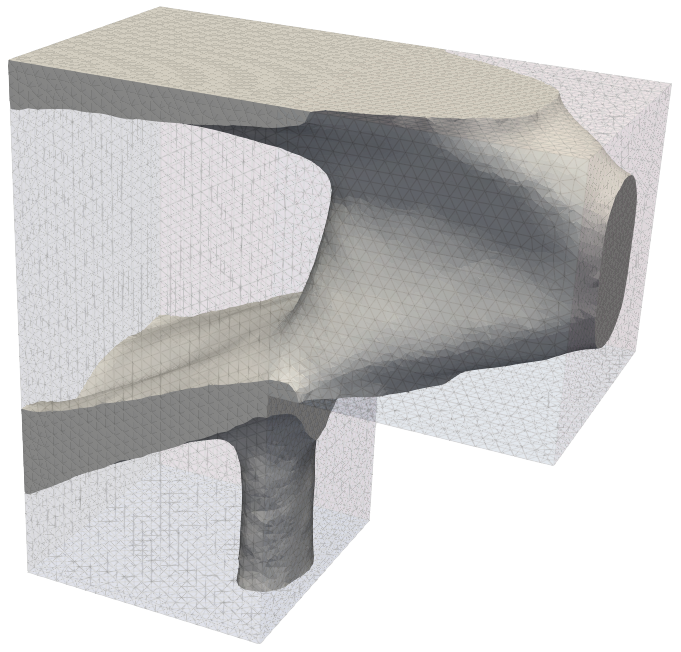} 
	\centerline{(c) Final design}
\end{minipage}
\hfill
\begin{minipage}[h]{0.32\linewidth}
	\centering
	\includegraphics[height=3.6cm,width=3.6cm]{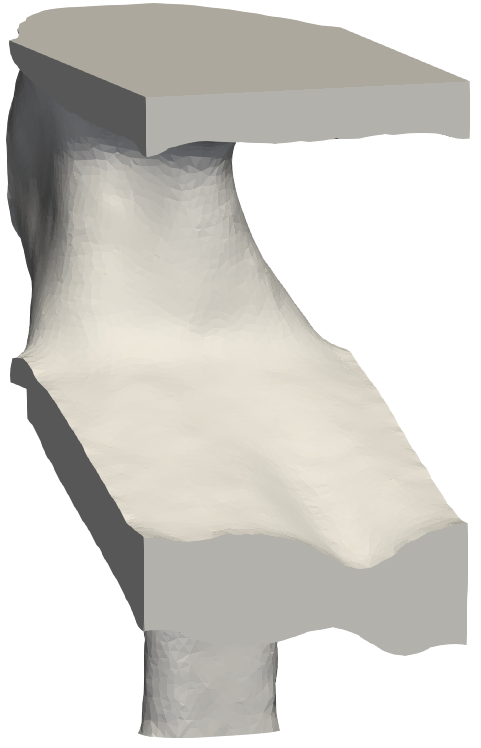} 
	\centerline{(d) Another view}
\end{minipage}
\hfill
\begin{minipage}[h]{0.32\linewidth}
	\centering
	\includegraphics[height=4.cm,width=5.6cm]{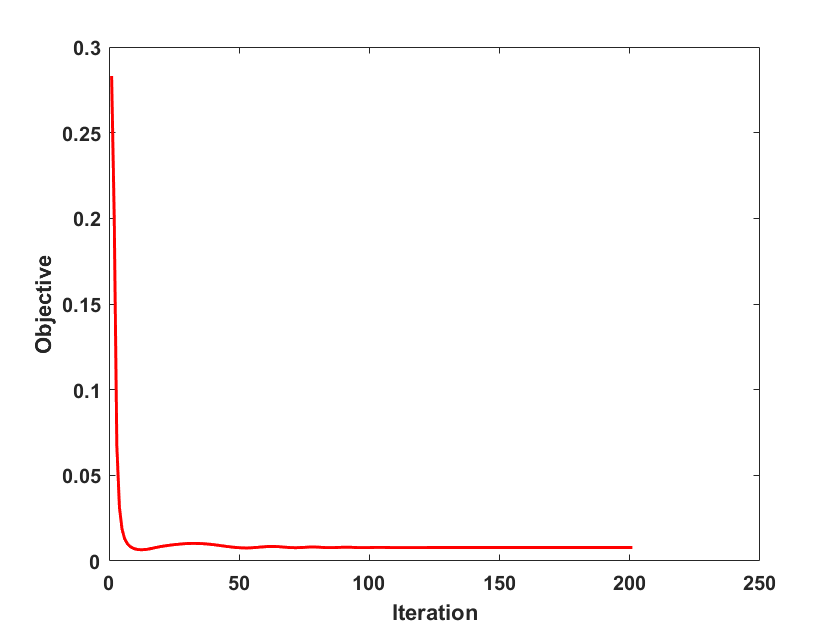} 
	\centerline{(e) Objective}
\end{minipage}
\hfill
\begin{minipage}[h]{0.32\linewidth}
	\centering
	\includegraphics[height=4.cm,width=5.6cm]{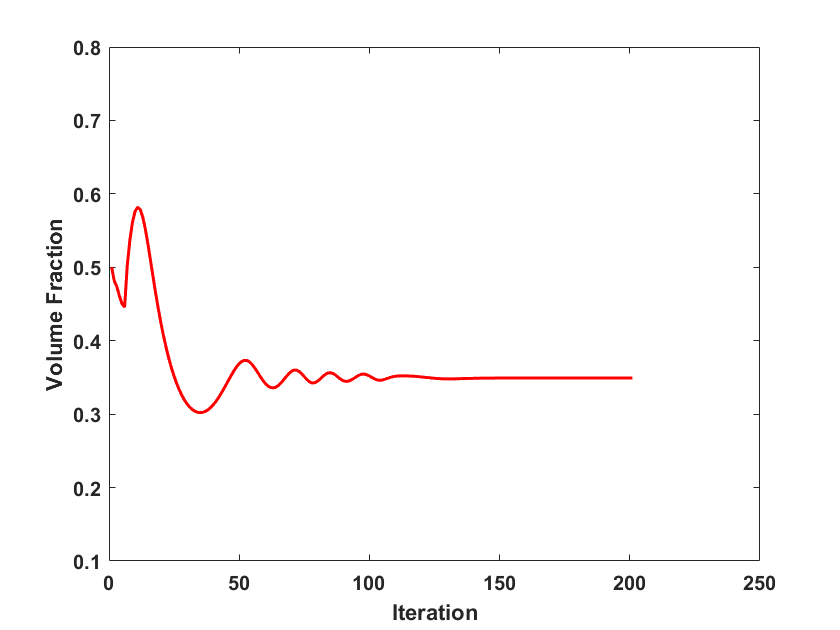} 
	\centerline{(f) Volume fraction}
\end{minipage}
\caption{Example \ref{Ex7}: problem setting and numerical results.}\label{fig10}
\end{figure}

\begin{exam}\label{Ex8}
\textbf{Bridge-like design in 3D.}  
{\rm Set $D = (0,2) \times (0,1) \times (0,1)$, which is configured as illustrated in Fig. \ref{fig11}. The region consists of four distinct types of surfaces: $\partial \Omega_D$ is the Dirichlet boundary, where prescribed displacements are enforced; $\partial \Omega_N$ is the surface subjected to an applied load force; $\partial \Omega_C$ is in contact with a rigid base, imposing constraints on displacement; and the remaining surfaces form free boundaries, which are not subject to any specific constraints. Fig. \ref{fig11} presents the optimized structural design obtained from an initially random configuration. As the optimization progresses, the objective function exhibits a steady decrease, indicating an improvement in the design. Additionally, the volume fraction successfully converges to the target value $V_f = 0.2$, demonstrating the effectiveness of the optimization process in achieving the desired material distribution.}
\end{exam}

\begin{figure}[htbp]
\begin{minipage}[h]{0.5\linewidth}
	\centering
	\includegraphics[height=2.7cm,width=4.5cm]{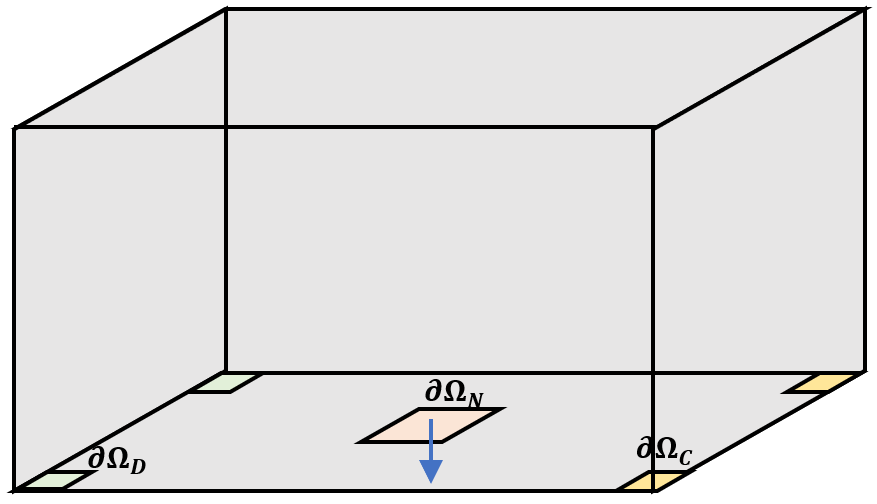} 
	\centerline{(a) Problem setting}
\end{minipage}
\hfill
\begin{minipage}[h]{0.5\linewidth}
	\centering
	\includegraphics[height=3.15cm,width=5.cm]{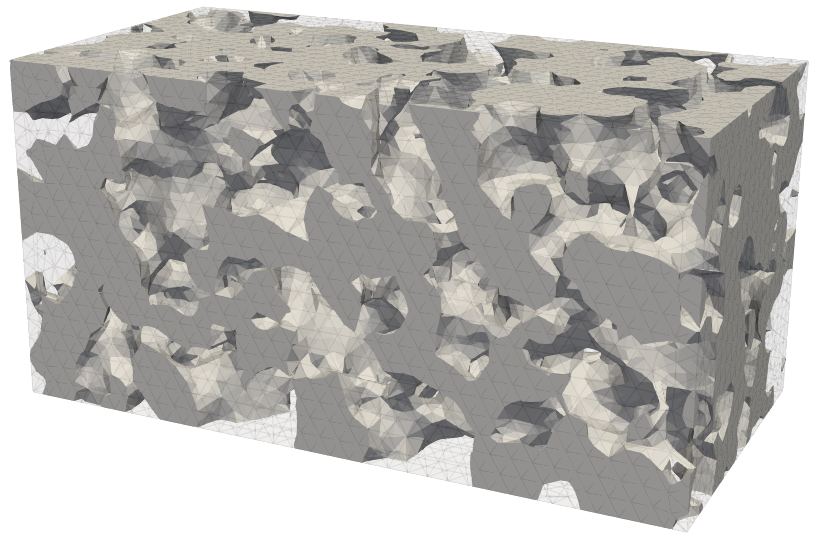} 
	\centerline{(b) Initial shape}
\end{minipage}
\hfill
\begin{minipage}[h]{0.3\linewidth}
	\centering
	\includegraphics[height=3.6cm,width=5.cm]{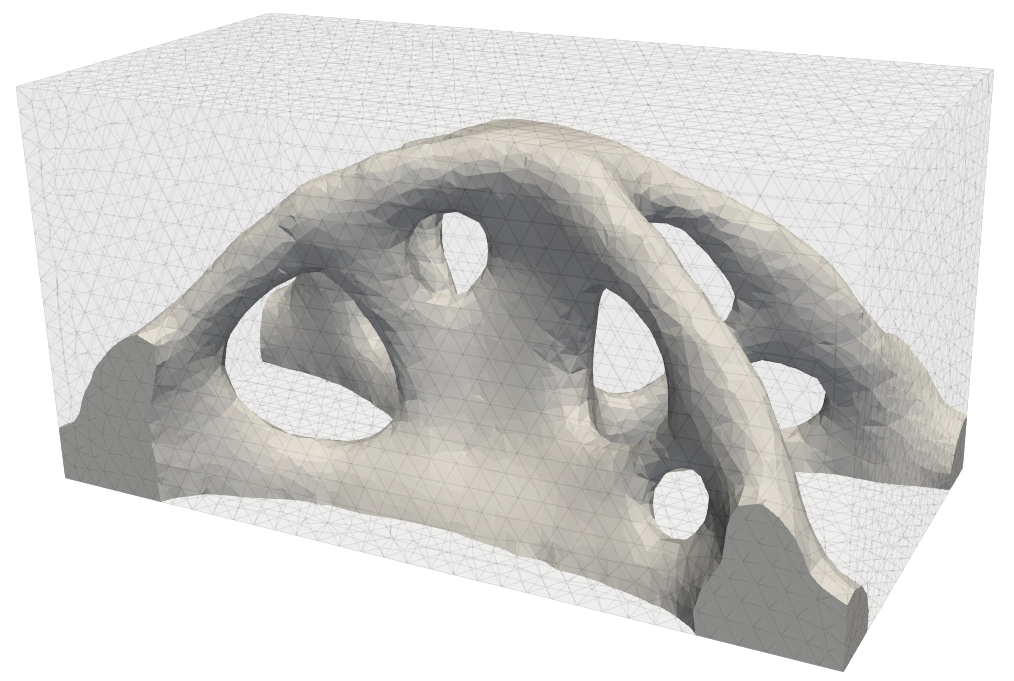} 
	\centerline{(c) Final shape}
\end{minipage}
\hfill
\begin{minipage}[h]{0.3\linewidth}
	\centering
	\includegraphics[height=4.cm,width=4.cm]{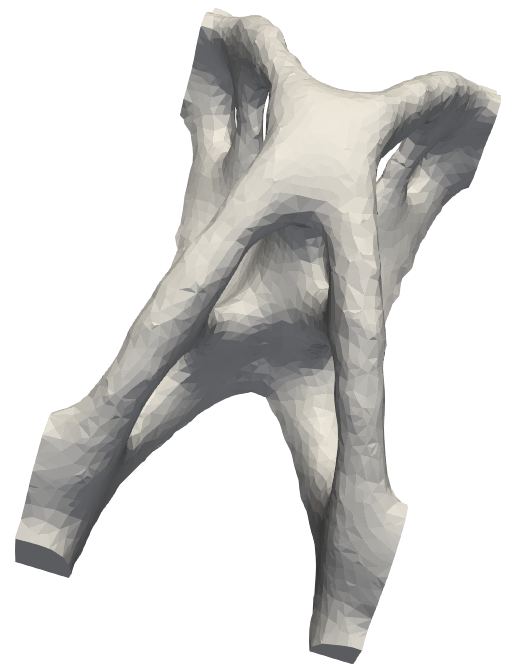} 
	\centerline{(d) Another view}
\end{minipage}
\hfill
\begin{minipage}[h]{0.32\linewidth}
	\centering
	\includegraphics[height=4.cm,width=5.6cm]{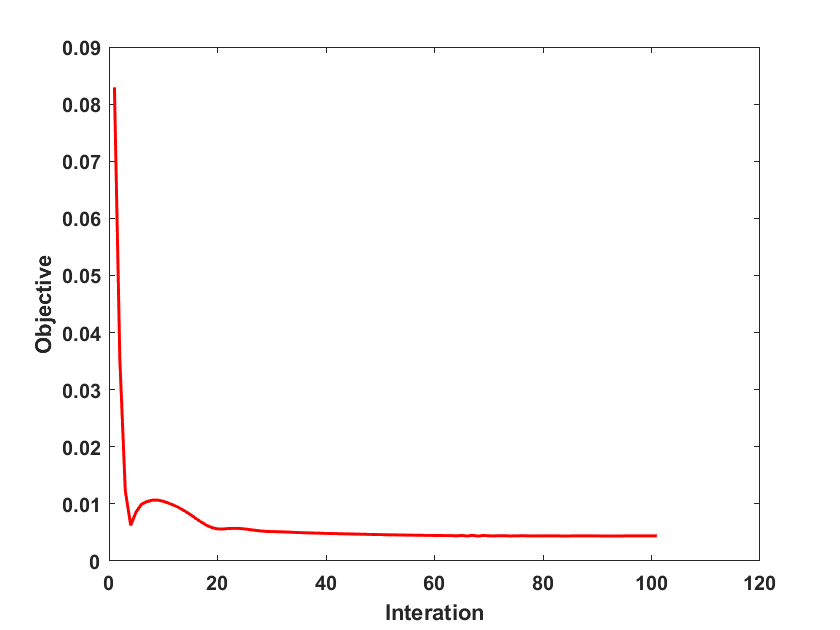} 
	\centerline{(e) Objective}
\end{minipage}
\caption{Example \ref{Ex8}: problem setting and numerical results.}\label{fig11}
\end{figure}


The model settings for \textbf{Example} \ref{Ex9}, \textbf{Example} \ref{Ex10}, and \textbf{Example} \ref{Ex11} are consistent with \textbf{Example} \ref{Ex3}, \textbf{Example} \ref{Ex4}, and \textbf{Example} \ref{Ex5}, respectively. The only difference lies in the objective \eqref{jtop} being optimized. 

\begin{exam}\label{Ex9}
{\rm 
	The optimization process under different initial designs is illustrated in Fig. \ref{fig6}, where the results highlight significant topological changes, including the formation of holes and redistribution of material, to show the effectiveness of the topological derivative in refining the structure towards an optimized configuration.  Fig. \ref{fig12} presents convergence histories for both the objective and volume fraction, corresponding to the two different initial designs. The trends indicate similar convergence behaviors, confirming the robustness of the optimization process regardless of the initial configuration.}
\end{exam}

\begin{exam}\label{Ex10}
{\rm 
	The optimized structural design obtained through the proposed topological optimization method is displayed in Fig. \ref{fig67} (left). 
}
\end{exam}

\begin{exam}\label{Ex11}
{\rm 
	Fig. \ref{fig67} (right) presents the optimization result derived from the initial design shown in Fig. \ref{fig8} (c). 
}
\end{exam}

\begin{figure}[htbp]
\begin{minipage}[h]{0.3\linewidth}
	\centering
	\includegraphics[height=2.25cm,width=4.5cm]{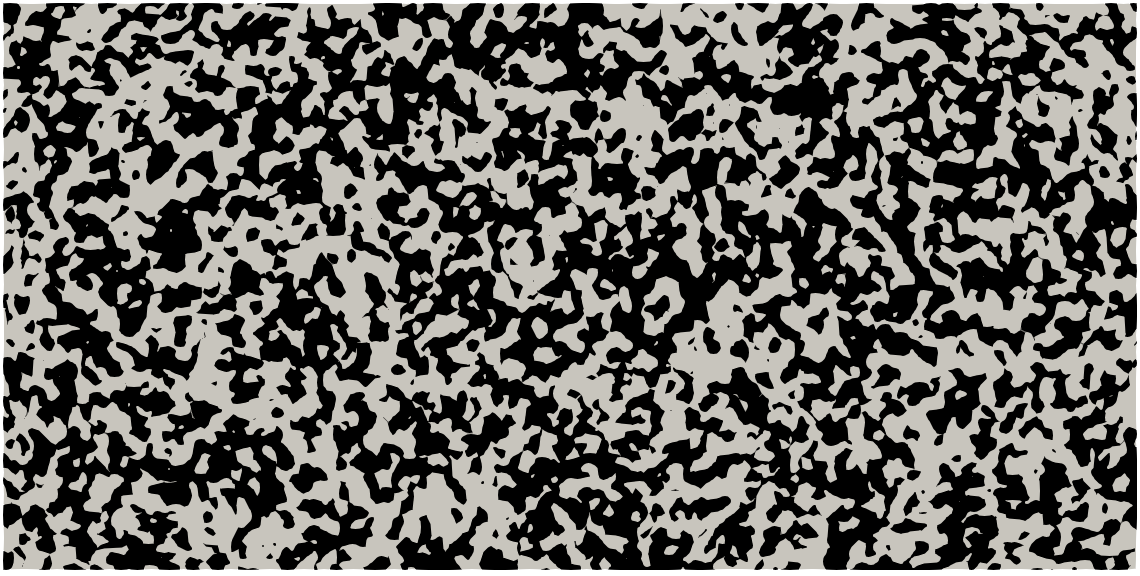} 
	\centerline{(a) Initial design $\phi_1$}
\end{minipage}
\hfill
\begin{minipage}[h]{0.3\linewidth}
	\centering
	\includegraphics[height=2.25cm,width=4.5cm]{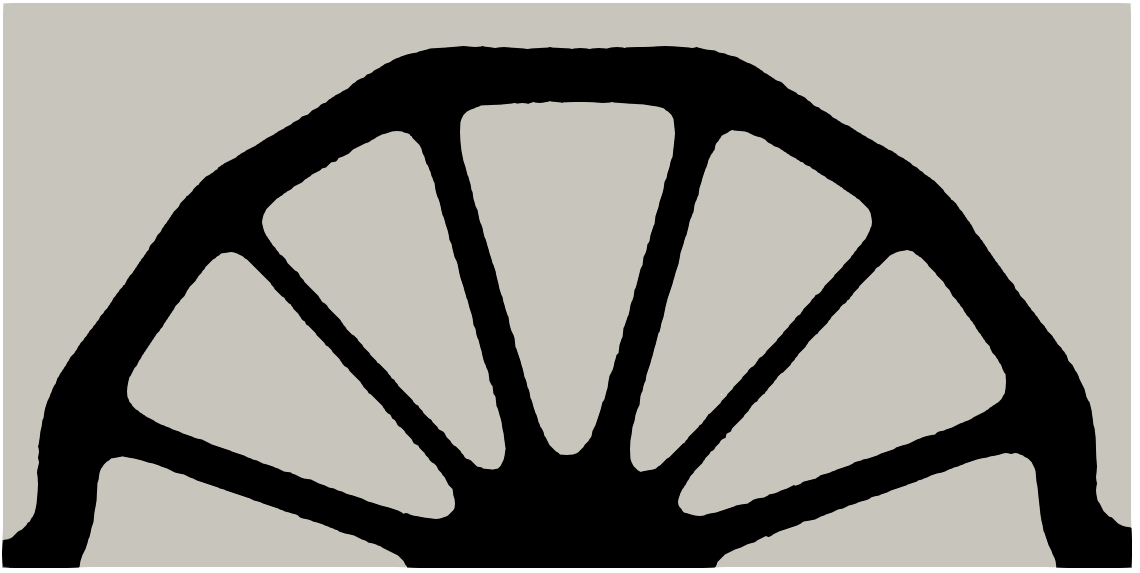} 
	\centerline{(b) Final design 1}
\end{minipage}
\hfill
\begin{minipage}[h]{0.3\linewidth}
	\centering
	\includegraphics[height=2.25cm,width=4.5cm]{exam4_initial_tp.png} 
	\centerline{(c) Initial design $\phi_2$}
\end{minipage}
\hfill
\begin{minipage}[h]{0.3\linewidth}
	\centering
	\includegraphics[height=2.25cm,width=4.5cm]{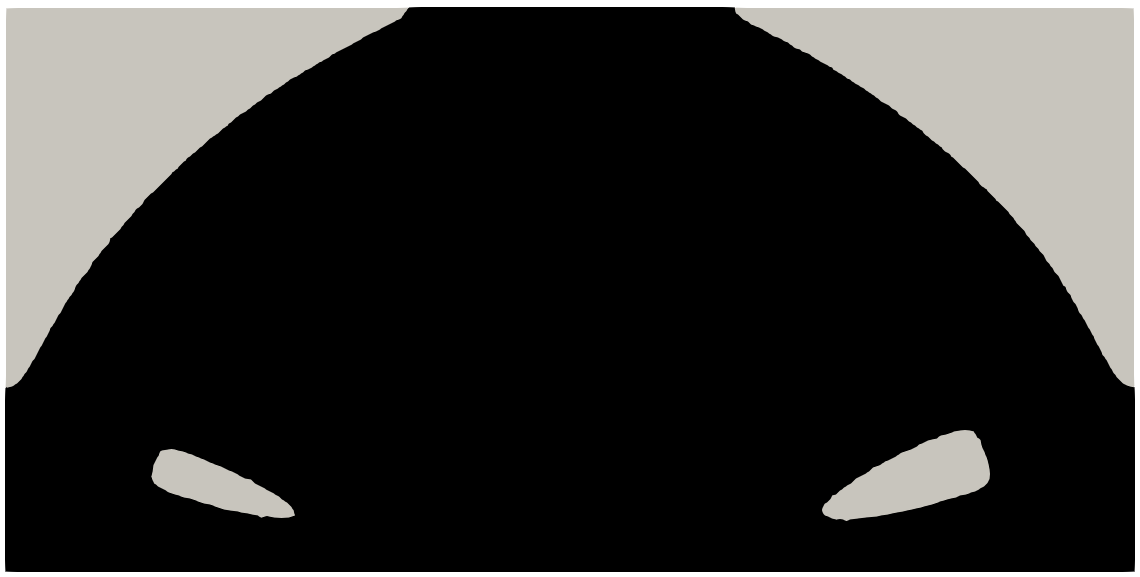} 
	\centerline{(d) Iteration 28 of $\phi_2$}
\end{minipage}
\hfill
\begin{minipage}[h]{0.3\linewidth}
	\centering
	\includegraphics[height=2.25cm,width=4.5cm]{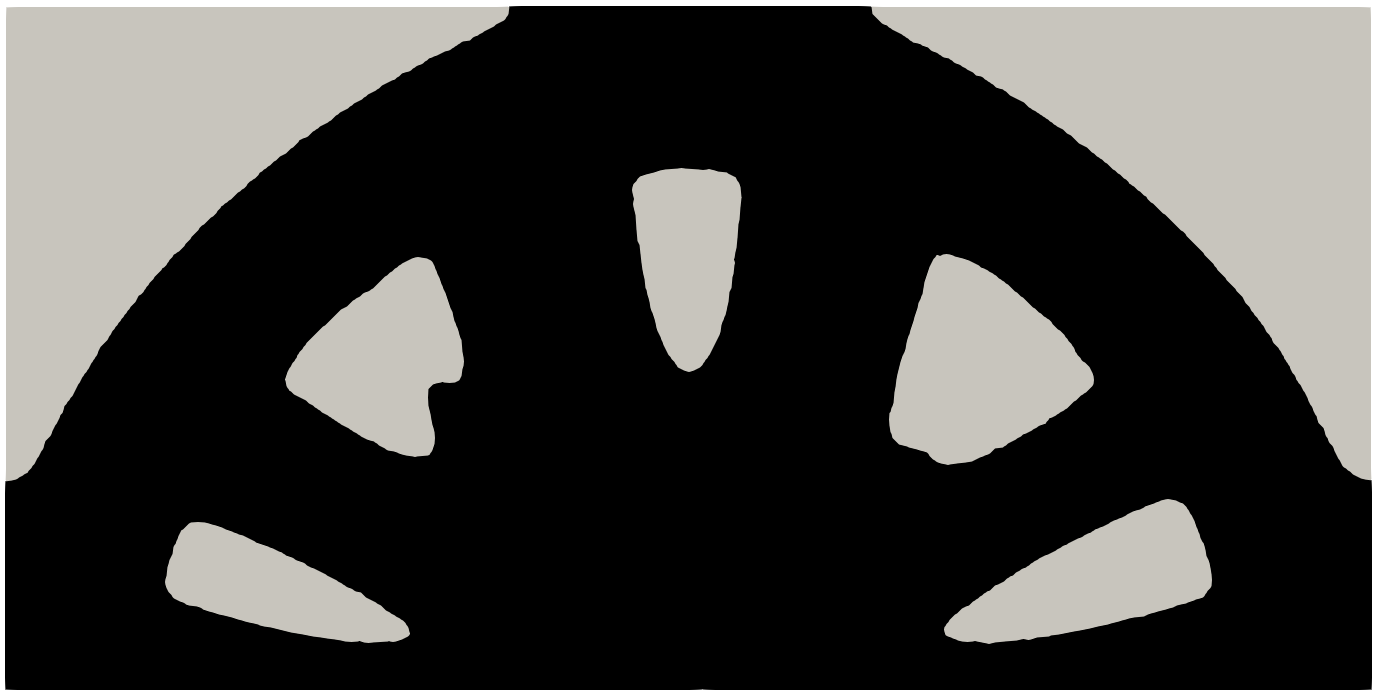} 
	\centerline{(e) Iteration 32 of $\phi_2$}
\end{minipage}
\hfill
\begin{minipage}[h]{0.3\linewidth}
	\centering
	\includegraphics[height=2.25cm,width=4.5cm]{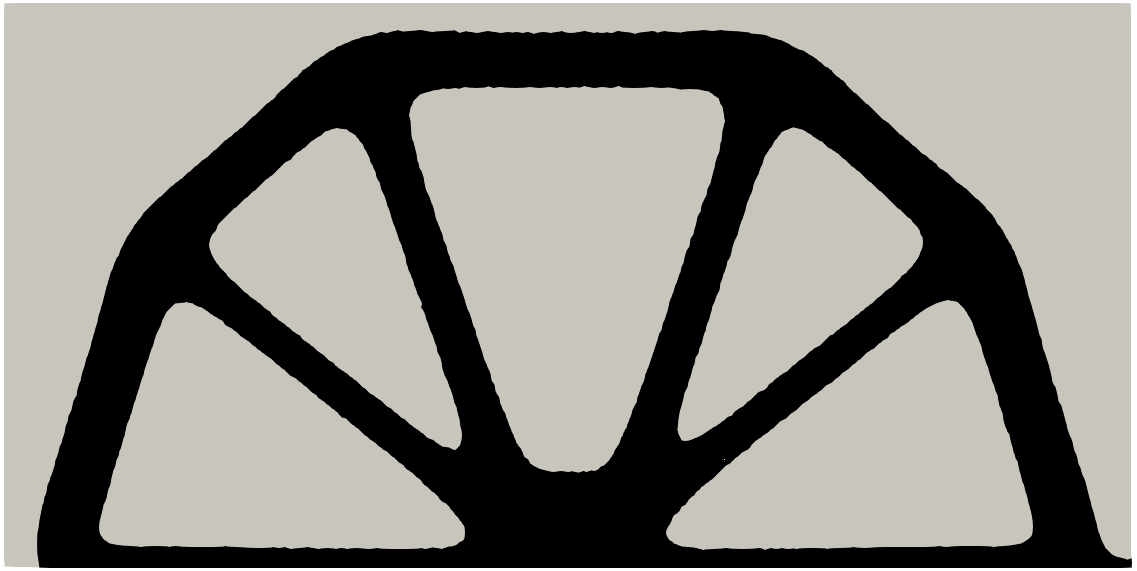} 
	\centerline{(f) Final design 2 }
\end{minipage}
\caption{Shape design processes for Example \ref{Ex9} with topological derivative.}\label{fig6}
\end{figure}

\begin{figure}[htb]
\begin{minipage}[h]{0.5\linewidth}
	\centering
	\includegraphics[height=5.cm,width=7.cm]{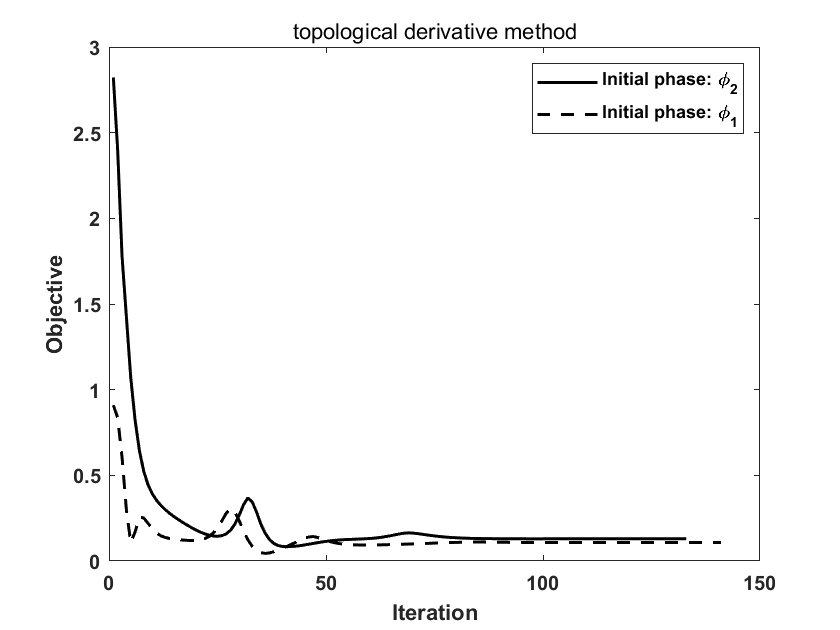} 
\end{minipage}
\hfill
\begin{minipage}[h]{0.5\linewidth}
	\centering
	\includegraphics[height=5.cm,width=7.cm]{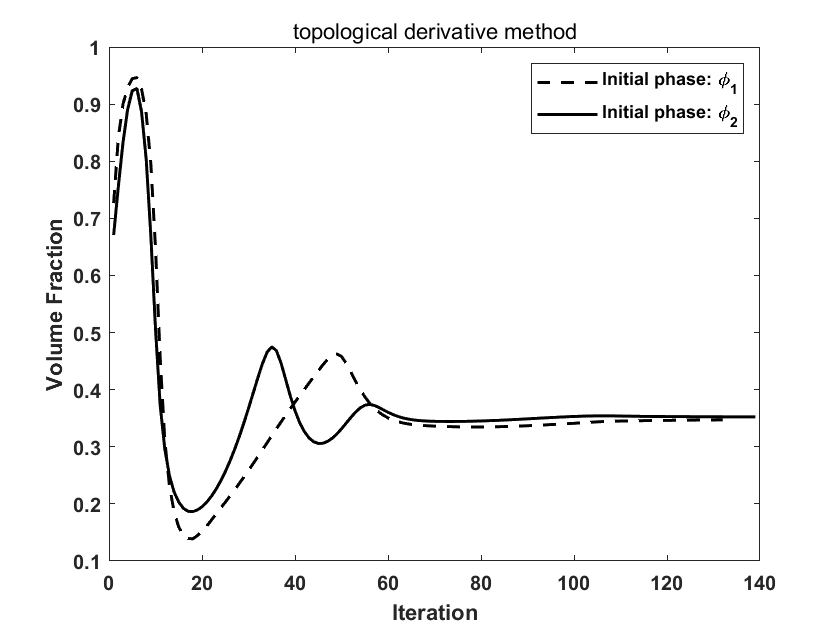} 
\end{minipage}
\caption{Example \ref{Ex9}: convergence histories of objective (left) and volume fraction (right).}\label{fig12}
\end{figure}

\begin{figure}[htbp]
\centering
\begin{minipage}[h]{0.48\linewidth}
	\centering
	\includegraphics[height=2.75cm,width=5.5cm]{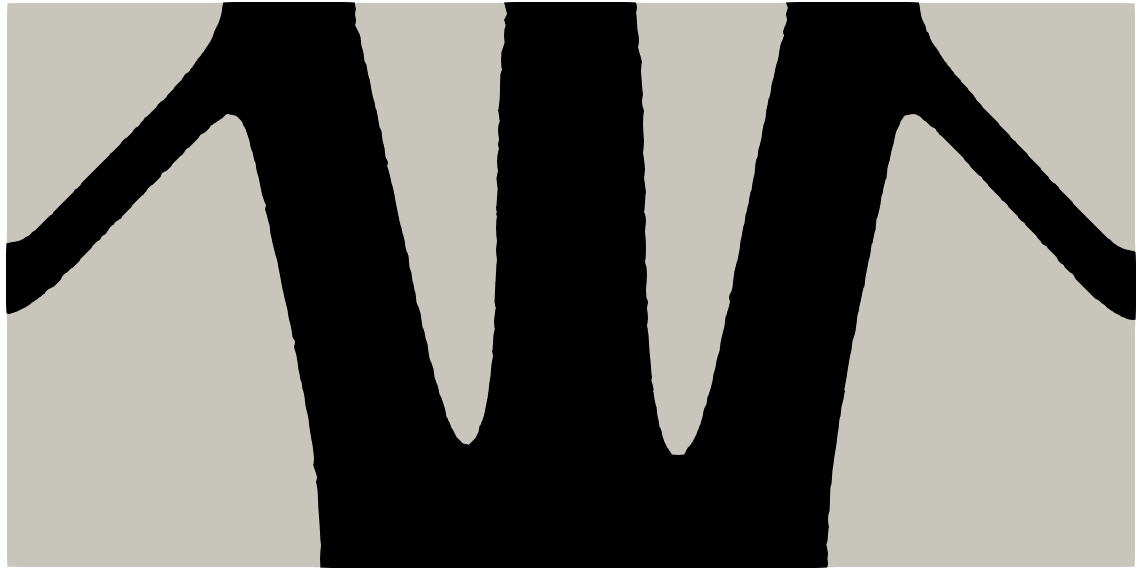} 
\end{minipage}
\hfill
\begin{minipage}[h]{0.48\linewidth}
	\centering
	\includegraphics[height=3cm,width=3cm]{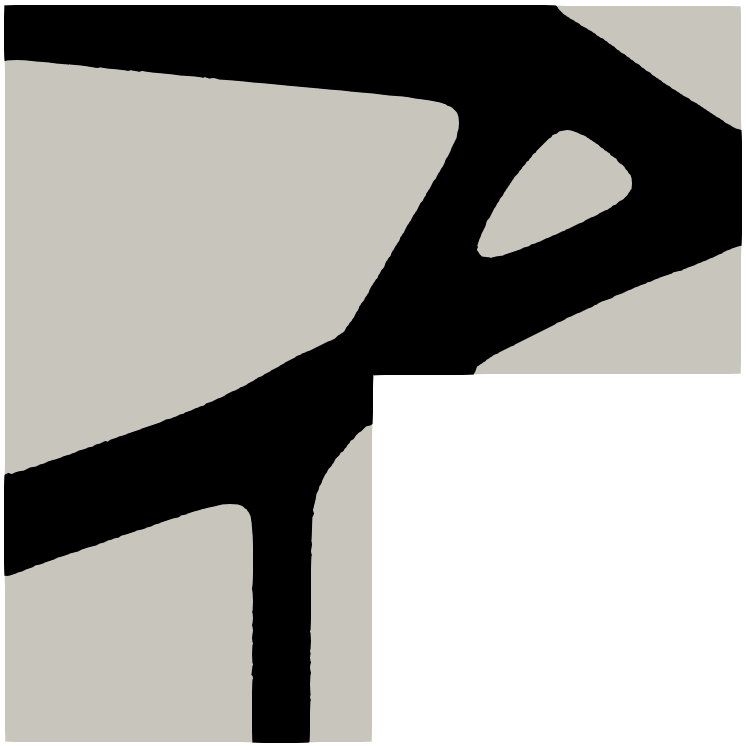} 
\end{minipage}
\caption{Example \ref{Ex10} (left) and Example \ref{Ex11} (right): final designs.}\label{fig67}
\end{figure}

\section{Conclusions}\label{secfinal}
We have studied modeling and numerical methods for shape and topology optimization in structural contact friction between elastic and rigid bodies. Both geometry constraint and mechanical constraint of hemivariational inequality are considered. For the nonlinear, non-smooth, and non-convex structural hemivariational inequality, {a regularization method is used to handle the non-smoothness of a general shape functional in sensitivity analysis. The Eulerian derivative of the shape functional of energy type has been derived by strict shape sensitivity analysis. The rationality of a regularization method has been demonstrated through asymptotic analysis of the regularized Eulerian derivative. Then, a numerical boundary variational shape optimization method is proposed. Moreover, three topology optimization algorithms} have been proposed based on phase-field method and topological derivatives by topological asymptotic analysis. One phase-field method is the traditional Allen-Cahn type and the other type involving Willmore functional of mean curvature is designed for the first time to the best of our view for topology optimization. Its effectiveness is expected to be applied for other topology optimization problems, e.g., in linear elasticity (compliance, weight, eigenfrequency, etc) and fluid flows. A variety of numerical results have shown the effectiveness and robustness of the algorithms in shape and topological changes.

\section*{Data availability}
The code for reproduction is uploaded on Github with link: https://github.com/Tyiyiyi/HVI.git.

\appendix
\section{Topological asymptotic analysis} 
For numerical topology optimization with phase field method combined with topological derivative, we use the weak material approach to approximate voids and consider nonzero contrast cases to avoid singularities in numerical computations.
For the topological derivative \eqref{jtop}, we use the topological asymptotic analysis of the objective in the nucleation of small circular inclusions \cite{NSZ2019} for the present case of non-convex hemivariational inequalities. We refer to \cite[Chapter 9]{NSZ2019} for the derivation of the topological derivatives of the convex variational inequalities (see \cite{SZ}).

Let us consider 2D for simplicity. Similar derivation can be done for 3D.
Decomposing $\Omega$ into two parts, namely $\Omega=\Omega_R \cup \overline{B_R}$, where $\Omega_R:=$ $\Omega \backslash \overline{B_R}$ and $B_R$ with boundary $\Gamma_R$ denotes a ball of radius $R>0$ centered at an arbitrary point $\widehat{\bm{x}} \in \Omega$. See sketch in Fig. \ref{topological_derivative}. Consider the following system of linear elasticity in $B_R$: Given $\bm \psi \in H^{1 / 2}\left(\Gamma_R ; \mathbb{R}^2\right)$, find the displacement $\bm w$ : $B_R \mapsto \mathbb{R}^2$ such that
\begin{equation}\label{Eq_TP1}
\left\{\begin{array}{rlrl}
	-\operatorname{div}( \bm \sigma(\bm w)) & =0 & & \text { in } B_R, \\
	\bm \sigma(\bm w) & =\mathbb{C} \bm \varepsilon(\bm w) && \text { in } B_R, \\
	\bm w & =\bm \psi & & \text { on } \Gamma_R .
\end{array}\right.
\end{equation}
Define the Steklov-Poincaré pseudo-differential boundary operator:
\begin{equation}
\mathscr{A}:\bm \psi \in H^{1/2}\left(\Gamma_R ; \mathbb{R}^2\right) \mapsto \bm \sigma(\bm w) \bm n \in H^{-1 / 2}\left(\Gamma_R;\mathbb{R}^2\right),
\end{equation}
where $\bm n$ is the outward normal vector to the boundary $\Gamma_R$. Notice that by the definition of operator $\mathscr{A}$, the solution $\bm w$ of \eqref{Eq_TP1} satisfies:
\begin{equation}\label{Eq_tp3}
\int_{B_R} \bm \sigma(\bm w) : \bm \varepsilon(\bm w)=\int_{\Gamma_R} \mathscr{A}(\bm \psi) \cdot \bm \psi .
\end{equation}

Setting $\bm \psi=\bm u_{\left.\right|_R}$, we have $\bm w=\bm u_{\left.\right|_{B_R}}$ and $\bm u^R=\bm u_{\left.\right|_{\Omega_R}}$. Thus, we can get $J(\bm u)=J^R\left(\bm u^R\right)$, where the functional $J^R\left(\bm u^R\right)$ is defined as
\begin{equation}\label{Eq_tp5}
J^R\left(\bm u^R\right):=\frac{1}{2} \int_{\Omega_R} \bm \sigma\left(\bm u^R\right) : \bm \varepsilon( \bm u^R) -\int_{\Gamma_N}\bm  g_N \cdot \bm u^R + \int_{\Gamma_C} j_{\tau}(u_{\tau}^R)  +\frac{1}{2} \int_{\Gamma_R} \mathscr{A}\left(\bm u^R\right) \cdot \bm u^R,
\end{equation}
with the minimizer $u^R \in \mathbf{V}_1^R:= \mathbf{V}_1 \left(\Omega_R\right)$ being a solution to the following hemivariational inequality: $\forall \bm v\in \mathbf{V}_1^R$
\begin{equation}\label{Eq_tp6}
\begin{aligned}
	& \int_{\Omega_R} \bm \sigma\left(\bm u^R\right) : \bm \varepsilon( \bm v-\bm u^R) -\int_{\Gamma_N} \bm g_N \cdot\left(\bm v-\bm u^R\right) + \int_{\Gamma_C} (j_{\tau}(v_{\tau}^R)-j_{\tau}(u_{\tau}^R)) +\int_{\Gamma_R} \mathscr{A}\left(\bm u^R\right) \cdot\left(\bm v-\bm u^R\right) \geq 0,
\end{aligned}
\end{equation}
where $\mathbf{V}_1(\Omega_R)$'s definition is similar as $\bm{V}_1(\Omega)$ except $\Omega$ is replaced by $\Omega_R$.

Topological perturbation is characterized by the nucleation of a small circular inclusion $\omega_{\zeta}(\widehat{\bm{x}}):=B_{\zeta}$ of radius $0<\zeta<R$ and center at $\widehat{\bm{x}} \in \Omega$, which is assumed to be far enough from the potential contact region $\Gamma_C$. This inclusion is filled with different material properties represented by a piecewise constant function $r_{\zeta}$ defined as
\begin{equation}
r_{\zeta}=r_{\zeta}(\bm{x}):=\left\{\begin{array}{l}
	1, \text { if } \bm{x} \in \Omega \backslash \overline{B_{\zeta}}, \\
	r, \text { if } \bm{x} \in B_{\zeta},
\end{array}\right.
\end{equation}
with $r \in \mathbb{R}^{+}$ being used to represent the contrast of the two material properties. 

We consider the perturbed domain in $B_R$: Given $\bm \psi \in H^{1 / 2}\left(\Gamma_R ; \mathbb{R}^2\right)$, find the displacement $\bm w_{\zeta}: B_R \mapsto \mathbb{R}^2$, such that
\begin{equation}\label{Eq_tp9}
\left\{\begin{array}{rlrl}
	-\operatorname{div}\left(r_{\zeta} \bm \sigma\left(\bm w_{\zeta}\right)\right) & =0 & & \text { in } B_R, \\
	\bm \sigma\left(\bm w_{\zeta}\right) & =\mathbb{C} \bm \varepsilon( \bm w_{\zeta}) && \text { in } B_R, \\
	\bm w_{\zeta} & =\bm \psi & & \text { on } \Gamma_R, \\
	\llbracket \bm w_{\zeta} \rrbracket & =0 & & \text { on } \partial B_{\zeta}, \\
	\llbracket r_{\zeta} \bm \sigma\left(\bm w_{\zeta}\right) \rrbracket \bm n & =0 & & \text { on } \partial B_{\zeta} .
\end{array}\right.
\end{equation}
Then we know the solution $\bm w_{\zeta}$ of \eqref{Eq_tp9} satisfies
\begin{equation}\label{Eq_tp10}
\int_{B_R} r_{\zeta} \bm \sigma\left(\bm w_{\zeta}\right) : \bm \varepsilon (\bm w_{\zeta})=\int_{\Gamma_R} \mathscr{A}_{\zeta}(\bm \psi) \cdot \bm \psi.
\end{equation}
Setting $\bm \psi=\bm u_{\zeta \vert \Gamma_R}$ we have $\bm w_{\zeta}=\bm u_{\zeta \vert B_R}$ and $\bm u_{\zeta}^R=\bm u_{\zeta \vert {\Omega_R}}$, which implies the equality $J_{\zeta}\left(\bm u_{\zeta}\right)=J_{\zeta}^R\left(\bm u_{\zeta}^R\right)$, where the functional $J_{\zeta}^R\left(\bm u_{\zeta}^R\right)$ is written as
\begin{equation} \label{Eq_tp11}
J_{\zeta}^R\left(\bm u_{\zeta}^R\right):=\frac{1}{2} \int_{\Omega_R} \bm \sigma\left(\bm u_{\zeta}^R\right) : \bm \varepsilon (\bm u_{\zeta}^R)-\int_{\Gamma_N} \bm g_N \cdot \bm u_{\zeta}^R+ \int_{\Gamma_C} j_{\tau}(u_{\tau,\zeta}^R)+\frac{1}{2} \int_{\Gamma_R} \mathscr{A}_{\zeta}\left(\bm u_{\zeta}^R\right) \cdot \bm u_{\zeta}^R,
\end{equation}
with the minimizer $\bm u_{\zeta}^R \in \mathbf{V}_1^R$ given by the unique solution of the hemivariational inequality:
\begin{equation}\label{Eq_tp12}
\begin{aligned}
	& \int_{\Omega_R} \bm \sigma\left(\bm u_{\zeta}^R\right) : \bm \varepsilon \left(\bm v-\bm u_{\zeta}^R\right)-\int_{\Gamma_N} \bm g_N \cdot\left(\bm v-\bm u_{\zeta}^R\right)+ \int_{\Gamma_C} (j_{\tau}(v_{\tau})-j_{\tau}(u_{\tau,\zeta}^R)) +\int_{\Gamma_R} \mathscr{A}_{\zeta}\left(\bm u_{\zeta}^R\right) \cdot\left(\bm v-\bm u_{\zeta}^R\right) \geq 0.
\end{aligned}
\end{equation}
\begin{thm}\label{Theroem1}
Let $\bm u^R$ and $\bm u_{\zeta}^R$ be solution to \eqref{Eq_tp6} and \eqref{Eq_tp12}, respectively, and assume that
\begin{equation}\label{Eq_tp13}
	\mathscr{A}_{\zeta}=\mathscr{A}-\zeta^2 \mathscr{B}+\mathscr{R}_{\zeta},
\end{equation}
in the operator norm $\mathscr{L}\left(H^{1 / 2}\left(\Gamma_R ; \mathbb{R}^2\right) ; H^{-1 / 2}\left(\Gamma_R ; \mathbb{R}^2\right)\right)$, with $\mathscr{B}$ denotes a bounded linear operator and $\left\|\mathscr{R}_{\zeta}\right\|_{\mathscr{L}\left(H^{1 / 2}\left(\Gamma_R ; \mathbb{R}^2\right) ; H^{-1 / 2}\left(\Gamma_R ; \mathbb{R}^2\right)\right)}=o\left(\zeta^2 \right.)$. Then, the following estimate holds:
\begin{equation}\label{Eq_tp15}
	\left\|\bm u_{\zeta}^R-\bm u^R\right\|_{H^1\left(\Omega_R ; \mathbb{R}^2\right)} \leq C_0 \zeta^2 .
\end{equation}
\end{thm} 
\begin{proof}
By taking $\bm v=\bm u_{\zeta}^R$ in \eqref{Eq_tp6} and $\bm v=\bm u^R$ in \eqref{Eq_tp12} and adding them together, we can write
\begin{equation}
	\int_{\Omega_R} \bm\sigma\left(\bm u_{\zeta}^R-\bm u^R\right) : \bm \varepsilon \left(\bm u_{\zeta}^R-\bm u^R\right)+\int_{\Gamma_R}\left[\mathscr{A}_{\zeta}\left(\bm u_{\zeta}^R\right)-\mathscr{A}\left(\bm u^R\right)\right] \cdot\left(\bm u_{\zeta}^R-\bm u^R\right) \leq 0 .
\end{equation}
Assuming \eqref{Eq_tp13}, we have that
\begin{equation}\label{Eq_tp19}
	\begin{aligned}
		\int_{\Omega_R} \bm \sigma\left(\bm u_{\zeta}^R-\bm u^R\right) : \bm \varepsilon \left(\bm u_{\zeta}^R-\bm u^R\right)+\int_{\Gamma_R} \mathscr{A}\left(\bm u_{\zeta}^R-\bm u^R\right) \cdot\left(\bm u_{\zeta}^R-\bm u^R\right) \leq 
		\zeta^2 \int_{\Gamma_R} \mathscr{B}\left(\bm u_{\zeta}^R\right) \cdot\left(\bm u_{\zeta}^R-\bm u^R\right)+o\left(\zeta^2\right).
	\end{aligned}
\end{equation}
and by the coercivity of the bilinear form on the left-hand side of \eqref{Eq_tp19} it follows
\begin{equation}
	C_1\left\|\bm u_{\zeta}^R-\bm u^R\right\|_{H^1\left(\Omega_R ; \mathbb{R}^2\right)}^2 \leq \zeta^2 \int_{\Gamma_R} \mathscr{B}\left(\bm u_{\zeta}^R\right) \cdot\left(\bm u_{\zeta}^R-\bm u^R\right)+o\left(\zeta^2\right) .
\end{equation}
We can get by duality and the trace theorem
\begin{equation}
	\begin{aligned}
		\left\|\bm u_{\zeta}^R-\bm u^R\right\|_{H^1\left(\Omega_R ; \mathbb{R}^2\right)}^2 & \leq C_2 \zeta^2\left\|\mathscr{B}\left(\bm u_{\zeta}^R\right)\right\|_{H^{-1 / 2}\left(\Gamma_R ; \mathbb{R}^2\right)}\left\|\bm u_{\zeta}^R-\bm u^R\right\|_{H^{1 / 2}\left(\Gamma_R ; \mathbb{R}^2\right)} \\
		&\leq C_0 \zeta^2\left\|\bm u_{\zeta}^R-\bm u^R\right\|_{H^1\left(\Omega_R ; \mathbb{R}^2\right)} .
	\end{aligned}
\end{equation}
We obtain \eqref{Eq_tp15} with $C_0$ independent of the small parameter $\zeta$.
\end{proof}
\begin{lem} \label{lemma21}
The perturbed energy shape functional $J_{\zeta}^R\left(\bm u_{\zeta}^R\right)$ in \eqref{Eq_tp11} is differentiable w.r.t. $\zeta \rightarrow 0$. In particular, it admits the asymptotic expansion
\begin{equation}\label{Eq_tp20}
	J_{\zeta}^R\left(\bm u_{\zeta}^R\right)=J^R\left(\bm u^R\right)-\frac{\zeta^2}{2}\left\langle\mathscr{B}\left(\bm u^R\right), \bm u^R\right\rangle_{\Gamma_R}+o\left(\zeta^2\right).
\end{equation}
\end{lem}
\begin{proof}
By taking into account that $\bm u_{\zeta}^R \in \mathbf{V}_1^R$ is the minimizer of \eqref{Eq_tp11} and $\bm u^R \in \mathbf{V}_1^R$ is the minimizer of \eqref{Eq_tp5}, the following inequalities hold true
\begin{equation}\label{Eq_tp22}
	J_{\zeta}^R\left(\bm u_{\zeta}^R\right)-J^R\left(\bm u_{\zeta}^R\right) \leq J_{\zeta}^R\left(\bm u_{\zeta}^R\right)-J^R\left(\bm u^R\right) \leq J_{\zeta}^R\left(\bm u^R\right)-J^R\left(\bm u^R\right) .
\end{equation}
Using the definitions of $J_{\zeta}^R$ and $J^R$ and considering the expansion \eqref{Eq_tp13}, we have
\begin{equation}
	\frac{J_{\zeta}^R\left(\bm u^R\right)-J^R\left(\bm u^R\right)}{\zeta^2}=-\frac{1}{2}\left\langle\mathscr{B}\left(\bm u^R\right), \bm u^R\right\rangle_{\Gamma_R}+\left\langle\mathscr{R}_{\zeta}\left(\bm u^R\right), \bm u^R\right\rangle_{\Gamma_R} .
\end{equation}
Thus, we can obtain
\begin{equation}
	\lim _{\zeta \rightarrow 0^+}\frac{J _\zeta^R\left(\bm u^R\right)-J^R\left(\bm u^R\right)}{\zeta^2}=-\frac{1}{2}\left\langle\mathscr{B}\left(\bm u^R\right), \bm u^R\right\rangle_{\Gamma_R}.
\end{equation}
Similarly, we can get the limit on the left of \eqref{Eq_tp22}. Then we conclude that
\begin{equation}
	\lim _{\zeta \rightarrow 0^+}\frac{J_{\zeta}^R\left(\bm u_{\zeta}^R\right)-J^R\left(\bm u^R\right)}{\zeta^2}=-\frac{1}{2}\left\langle\mathscr{B}\left(\bm u^R\right), \bm u^R\right\rangle_{\Gamma_R} .
\end{equation}
\end{proof}
\begin{lem}\cite{AJ2013}\label{Lemma2}
The energy inside $B_R$ admits the asymptotic expansion:
\begin{equation}
	\int_{B_R} \bm \sigma_{\zeta}\left(\bm w_{\zeta}\right) : \bm \varepsilon \left(\bm w_{\zeta}\right)=\int_{B_R} \bm \sigma(\bm w) : \bm \varepsilon (\bm w)-\zeta^2 \mathbb{P}_r \bm \sigma(\bm w) : \bm \varepsilon ( \bm w)+o\left(\zeta^2\right),
\end{equation}
where the fourth-order isotropic polarization tensor $\mathbb{P}_r$ is given by \eqref{jiojw}.
\end{lem}

By Theorem \ref{Theroem1}, we can write
\begin{equation}\label{Eq_tp25}
\left\langle\mathscr{A}_{\zeta}(\bm \phi), \bm \varphi \right\rangle=\langle\mathscr{A}(\bm \phi), \bm \varphi \rangle-\zeta^2\langle\mathscr{B}(\bm \phi), \bm \varphi \rangle+o\left(\zeta^2\right), \quad \forall \bm \phi, \bm \varphi .
\end{equation}
By Lemma \ref{Lemma2}, and \eqref{Eq_tp25}, \eqref{Eq_tp3}, \eqref{Eq_tp10}, it follows that
\begin{equation}\label{Eq_tp26}
\int_{\Gamma_R} \mathscr{B}(\bm \phi) \cdot \bm \varphi=\mathbb{P}_r \bm \sigma(\bm \phi) : \bm \varepsilon (\bm  \varphi), \quad \forall \bm \phi, \bm \varphi .
\end{equation}
Thus, by Lemma \ref{lemma21} together with Eq. \eqref{Eq_tp26} we have
\begin{equation}
J_{\zeta}\left(\bm u_{\zeta}\right)-J(\bm u)=-\frac{1}{2} \zeta^2 \mathbb{P}_r \bm \sigma(\bm u) : \bm \varepsilon  (\bm u)+o\left(\zeta^2\right).
\end{equation}
By choosing $f(\zeta)=\zeta^2$, the topological derivative of the shape functional is given by \eqref{Topological}.

\end{document}